\documentclass[11pt,oneside,onecolumn,reqno]{amsart}
\usepackage{amsgen, amstext,amsbsy,amsopn, amsthm, amsfonts,amssymb,amscd,amsmat
	h,euscript,enumerate,url,verbatim,calc,oldgerm,mathrsfs,bm}
\usepackage{tikz}
\tikzset{->-/.style={decoration={
			markings,
			mark=at position #1 with {\arrow{>}}},postaction={decorate}}}
\usetikzlibrary{shapes,backgrounds,calc}
\usepackage[utf8]{inputenc}
\usepackage{latexsym}
\usepackage{graphics}
\usepackage{color}
\usepackage{bm}
\theoremstyle{plain}
\newtheorem{theorem}{Theorem}[section]

\newtheorem{lemma}[theorem]{Lemma}

\theoremstyle{definition}

\newtheorem{remark}[theorem]{Remark}

\theoremstyle{remark}
\newtheorem*{ack*}{\textbf{Acknowledgement}}
\usetikzlibrary{decorations.markings}

\usepackage{ragged2e}
\usepackage{pifont}
\usepackage{setspace}
\makeatletter

\newcommand{\Rmnum}[1]{\expandafter\@slowromancap\romannumeral #1@}
\makeatother
\usepackage{csquotes} 
\usepackage{cite}
\usepackage{amsaddr}
\usepackage{caption}
\usepackage{subcaption}
\begin{document}
	\title[$ \delta $-shocks in realistic Chaplygin Aw-Rascle model]{Intrinsic phenomena of delta shock waves in a more realistic Chaplygin Aw-Rascle model}
	\author[M. Zafar]{Priyanka, \, M. Zafar}
	\address{Department of Mathematics \& Computing, Dr. B.R. Ambedkar NIT Jalandhar, India\\ Phone: +91-0181-5037700}
	\email[Priyanka]{priyanka.ma.21@nitj.ac.in}
	\email[M.Zafar]{zafarm@nitj.ac.in}
%	\author[M. Zafar]{Priyanka, \, M. Zafar}
%	\address{Department of Mathematics and Computing, Dr. B.R. Ambedkar NIT Jalandhar, India\\ Phone: +91-0181-5037700}
%	\email[Priyanka]{priyanka.ma.21@nitj.ac.in}
%	\email[M.Zafar]{zafarm@nitj.ac.in}
	\begin{abstract}
		The motivation of this study is to find the Riemann solutions of Aw-Rascle model with a more realistic version of extended Chaplygin gas. Firstly, we establish the Riemann solutions with two different structures, viz., a shock wave followed by the contact discontinuity and a rarefaction wave followed by the contact discontinuity. Further, by analyzing the limiting behavior, it is found that one of the Riemann solutions converges to $ \delta $-shock solution as the pressure approaches to generalized Chaplygin gas pressure. Moreover, numerical simulations have been performed to validate the theoretical analysis.

		\noindent \textbf{Keywords.} Riemann problem;  Chaplygin gas; Delta shock;  Transport equations; Contact discontinuity
	\end{abstract}
	\maketitle
	\section{introduction}

	The formation of delta shock waves in the solutions of Riemann problems is one of the major challenges to contemporary mathematical research. It has not only physical importance such as mass accumulation and the creation of galaxies in the universe  \cite{yang2012, li1998, panov2006, liu2020, xin2024, li2024a, zhang2021, zhang2024}, but also some special interest and difficulty in mathematics. The research on delta shock waves has been developed in the last three decades; several researchers have achieved numerous outstanding achievements  involving non-classical solutions for various hyperbolic system of conservation laws \cite{tan1994d, nedeljkov2004, danilov2005, zhang2023, sheng1999, li1998}; for more details see the references cited therein. Unlike the ordinary shock, $ \delta $-shock is an over-compressive shock in which more  characteristics impinge to  the line of discontinuity.

	To reveal some more intrinsic phenomena of delta shock waves in the theory of hyperbolic conservation laws, we consider the following system of conservation laws \cite{cheng2014}
 \begin{equation}\label{p1}
 	\begin{cases}
 		\varrho_{t}+(\varrho \upsilon)_{x}=0,
 		\\(\varrho (\upsilon+p))_{t}+(\varrho \upsilon (\upsilon+p))_{x}=0,
 	\end{cases} 
 \end{equation}with a more realistic version of extended Chaplygin gas \cite{mzafar2022, barthwal2022}
\begin{equation} \label{p2}
	p=A\left( \frac{\varrho}{1-a\varrho}\right)^{\Gamma}-\frac{B}{\varrho^{\kappa}}, \quad (1\leq\Gamma\leq3, \,0<\kappa\leq1),     
\end{equation}
 where $ \varrho>0 $, $ \upsilon $ $ \geq  0,$ and $ p,$  stand for the density, velocity, and pressure,  respectively. More details about the model \eqref{p1} can be found in \cite{zhang2002, daganzo1995}. In the equation of state \eqref{p2}, $a$ is the
 van der Waals excluded volume which sets a limit on the density $\varrho_{max}=1/a$ \cite{zhang2021} implying thereby that it is necessary to consider the limit $a \to 0$ while studying the formation of delta shock waves in the Riemann solutions of \eqref{p1}-\eqref{p2}. In addition, for $a=0,$ the parameter $\Gamma$ recovers the barotropic fluid having an equation of state up to order three; indeed, it captures the more complex behaviors of dark energy and dark matter \cite{pourhassan2014}. The parameter $\kappa$ describes the transition of generalized Chaplygin gas from dark matter to dark energy, influencing the late time evolution of the gravitational potential and large scale structure whereas the parameter $B$ determines the late time vacuum energy density that drives the late time accelerated expansion of the universe \cite{bento2002}. The parameter $A$ governs early time dynamics which describes the behaviors of fluid at high densities. For $A\to 0,$ the equation of state \eqref{p2} converts into the generalized Chaplygin gas and this limit recovers the original unification of dark matter and dark energy via a single fluid, \textit{i.e.,} at high density, the fluid pressure is negligible (dark matter) and at low density fluid behaves like dark energy (vacuum-like negative pressure).

 Pan and Han \cite{pan2013} studied the model \eqref{p1} with Chaplygin pressure and obtained that their Riemann solutions coincide with the Riemann solutions of pressureless gas dynamics when the pressure drops to zero. Shen and Sun  \cite{shen2010} analyzed the appearance of $ \delta $-shock and vacuum states in the perturbed Aw-Rascle model. Also, they \cite{shen2024} established the limiting behavior of the Riemann solutions for hydrodynamic Aw-Rascle traffic flow model. In 2018, Yin and Chen  \cite{yin2018} investigated the Riemann problem along with the stability of Riemann solutions to the non-homogeneous Aw-Rascle model. Zhang \cite{zhang2016} precisely analyzed the vanishing pressure limit of the Riemann solutions of Chaplygin Aw-Rascle model with friction term. Li \cite{li2022riemann} obtained that their Riemann solutions lose self-similarity nature due to the presence  of friction term. Shao in his motivational work \cite{shao2009}, established the global existence of a unique weakly discontinuous solution for the first order quasi-linear hyperbolic systems; indeed, the global existence of classical discontinuous solutions to genuinely nonlinear hyperbolic systems is discussed for a small bounded variations (BV) perturbations of the Riemann initial data (see, \cite{shao2013}).

%  considered the mixed initial boundary value problem for first order quasi-linear
% hyperbolic systems and established the global existence of a unique weakly discontinuous solution. Indeed, under a small BV perturbations of the Riemann initial
% data, he \cite{shao2013} proved the global existence of classical discontinuous solutions to genuinely nonlinear hyperbolic systems.
 
 As the pressure  $p$ falls to zero, the model \eqref{p1} converts into the following transport equations \cite{brenier1998, weinan1996}
	\begin{equation}\label{p3}
		\begin{cases}
			\varrho_{t}+(\varrho \upsilon)_{x}=0,
			\\(\varrho \upsilon)_{t}+(\varrho \upsilon^{2} )_{x}=0,
		\end{cases} 
	\end{equation}
	which describe the motion of free particles stuck under collision. Since 1994, numerous researchers have extensively studied the transport equations \cite{ huang2001, weinan1996, sheng1999, lij2001, li1998, zafar2016, zhang1989}; in particular, E, Rykov and Sinai \cite{weinan1996} investigated the behavior of global weak solution to $1$-D Riemann problem with random initial data. With the use of vanishing viscosity and characteristic analysis method, Sheng and Zhang \cite{sheng1999} studied the solutions of $1$-D and $2$-D Riemann problems. Li and Yang \cite{lij2001} examined $1$-D Riemann problem and obtained multidimensional planar delta shock waves depending on one-parameter family. Shao \cite{shao2023} proved that the Riemann solutions of Aw-Rascle model coincide with the Riemann solutions of the pressureless gas dynamics system as the traffic pressure converges to zero. Further, the appearance of the delta shock wave is discussed for the relativistic full Euler system with generalized Chaplygin proper energy density–pressure relation under some special conditions on the initial data (see, \cite{shao2018}).
	 Li et. al \cite{li2011} studied the interactions of rarefaction waves for $ 2 $-D Euler equations. Very recently, Jiang and Shen \cite{jiang2024} constructed the Riemann solutions of isothermal three-component model with four different structures. Zhang et. al \cite{zhang2024a} found that the delta shock is formed in the Riemann solutions of logarithmic Euler equations for certain initial data. Kipgen and Singh \cite{kipgen2023}  studied the formation of $\delta$-shocks and vacuum states in the Riemann solution of isothermal van der Waals dusty gas under the flux approximation.

	In the current study, we focus on the system \eqref{p1}-\eqref{p2} with the following initial data 
	 \begin{equation} \label{p4}
		(\varrho, \upsilon)_{t=0} =	\begin{cases}
			(\varrho_{l}, \upsilon_{l}), \quad x<0,
			\\ (\varrho_{r}, \upsilon_{r}), \quad x>0.
		\end{cases}
	\end{equation} 
	Using the method of characteristic analysis, first we construct the Riemann solutions of \eqref{p1}-\eqref{p2} and \eqref{p4} with two different structures, viz., a shock wave followed by the contact discontinuity ($ S+J $) and a rarefaction wave followed by the contact discontinuity ($ R+J $). Further, we obtain that the Riemann solution $S+J$ converges to $\delta$-shock solution as pressure tends to Chaplygin pressure. Moreover, numerical simulations are done to check the validity of the process of formation of $\delta$-shock in the limiting case.

	The organization of this paper is as follows. In sections $2,$ we obtain the solutions to the Riemann problem governed by \eqref{p1}-\eqref{p2} and \eqref{p4}. In section $3,$ we analyze the limiting behavior of the Riemann solutions to the conservative system as pressure approaches to the Chaplygin pressure. Finally,
	section $ 4 $ consists numerical simulations to verify theoretical analysis.
	\section{Riemann Solutions for  \eqref{p1}, \eqref{p2} and \eqref{p4}}  
	 The system \eqref{p1}, in matrix form, can be represent as
	\begin{equation}\label{p5}
		\begin{bmatrix}
			\varrho 
			\\ \upsilon 
		\end{bmatrix}_{t} + \begin{bmatrix}
			\upsilon  & \varrho 
			\\	0 & \upsilon  - \varrho p^{\prime}
		\end{bmatrix}
		\begin{bmatrix}
			\varrho 
			\\ \upsilon
		\end{bmatrix}_{x} = \begin{bmatrix}
		0 \\ 0
		\end{bmatrix},\end{equation}
	where	
			$ 	p^{\prime} = \frac{\Gamma A\varrho^{\Gamma-1}}{(1-a\varrho)^{\Gamma+1}} + \frac{\kappa B}{\varrho^{\kappa+1}}. $ 
Thus, the Jacobian matrix has the following eigenvalues and eigenvectors
	\begin{equation*}
		\lambda_{1}=\upsilon -  \varrho p^{\prime}, \,\, \bm{\overrightarrow{r_{1}}}=\left( 1, - p^{\prime}\right) ^{T}, \,\, \lambda_{2}=\upsilon ,
		\,\, \bm{\overrightarrow{r_{2}}}=(1, 0)^{T},
	\end{equation*}
	with $\bigtriangledown\lambda_{1}\cdot\bm{\overrightarrow{r_{1}}}\neq 0$  and $ \bigtriangledown\lambda_{2}\cdot\bm{\overrightarrow{r_{2}}}=0. $  Thus, the wave associated with $ \lambda_{1} $ is either shock or rarefaction wave and the wave associated with $ \lambda_{2} $ is a contact discontinuity. 
	\subsection{Rarefaction wave and contact discontinuity:}
	Under the self-similar transformation $\zeta=x/t,$ the Riemann problem presented by \eqref{p1}-\eqref{p2} and \eqref{p4} transforms into the following:
	\begin{equation}\label{p6}
		\begin{cases}
			-\zeta\varrho_{\zeta} + (\varrho \upsilon)_{\zeta} = 0,
			\\ -\zeta \left( \varrho \left( \upsilon+p\right) \right) _{\zeta}+\left( \varrho \upsilon \left( \upsilon+p\right) \right) _{\zeta}=0,
			%			\\	-\zeta(\varrho (\upsilon+p))_{\zeta} + (\varrho(\upsilon+\eta t)(\upsilon+p))_{\zeta} = 0, 
		\end{cases}
	\end{equation}
	and
	\begin{equation}\label{p7}	(\varrho, \upsilon) = 
		\begin{cases}
			(\varrho_{l}, \upsilon_{l}),\,\quad\zeta=-\infty,
			\\  (\varrho_{r}, \upsilon_{r}),\quad\zeta=+\infty.
		\end{cases}
	\end{equation}
		For any smooth solution, the system \eqref{p6} in matrix form can be written as \cite{smoller2012shock}
		\begin{equation}\label{p2.4}
			(dF-\zeta I)U_{\zeta}=O,
		\end{equation}where
	\begin{equation}\label{p2.5}
	dF	=\begin{bmatrix}
			\upsilon & \rho 
			\\ 0 &\upsilon-\rho p^{\prime}
		\end{bmatrix}, \quad 	I	=\begin{bmatrix}
		1 & 0 
		\\ 0 &1
		\end{bmatrix}, \quad U_{\zeta}= \begin{bmatrix}
			\rho_{\zeta} \\ \upsilon_{\zeta}
		\end{bmatrix}, \quad O= \begin{bmatrix}
		0 \\ 0
		\end{bmatrix}.
	\end{equation}
	\textbf{Case-(i).} If $ U_{\zeta} = O,$ then  the system \eqref{p6}  admits a constant state solution along with \eqref{p7} and it is possible iff $(\varrho_{l}, \upsilon_{l})=(\varrho_{r}, \upsilon_{r}).$\\ \textbf{Case-(ii).} If $ U_{\zeta} \not\equiv O,$ then $U_{\zeta}$ is an eigenvector of $dF$ corresponding to the eigenvalue $\zeta.$ Since $dF$ has two real and distinct eigenvalues, namely, $\zeta=\lambda_{1}=\upsilon-\varrho p^{\prime}$  and 
	$\zeta=\lambda_{2}=\upsilon.$ Therefore, we have the following two elementary wave curves:

\noindent (i) Rarefaction wave (elementary wave associated with $ \zeta= \lambda_{1} $)
\begin{equation}\label{p9}
	R : \begin{cases} \zeta=\lambda_{1}=\upsilon  -  \varrho p^{\prime}, \qquad \upsilon_{l}\leq\upsilon \leq \upsilon_{r},
		\\ \upsilon+p=\upsilon_{l}+p_{l},\quad\qquad \varrho_{r}\leq\varrho \leq \varrho_{l}, 
	\end{cases}
\end{equation}
along with $\frac{d\lambda_{1}}{d\varrho}=-\frac{(\Gamma^{2}+\Gamma)A\varrho^{\Gamma-1}}{(1-a\varrho)^{\Gamma+2}}+\frac{(\kappa^{2}-\kappa)B}{\varrho^{\kappa+1}}<0$ and $ p_{l}=A\left( \frac{\varrho_{l}}{1-a\varrho_{l}}\right) ^{\Gamma}-\frac{B}{\varrho_{l}^{\kappa}} . $

\noindent (ii) Contact discontinuity (elementary wave associated with $ \zeta= \lambda_{2} $)
	\begin{equation}\label{p8}
		J : \, \zeta=\lambda_{2}=\sigma_{2}=\upsilon_{l}=\upsilon_{r},
	\end{equation} Thus, the contact discontinuity $(J)$ connects the left state  $(\varrho_{l}, \upsilon_{l})$ and the right state $(\varrho_{r}, \upsilon_{r})$ iff $\upsilon_{l}=\upsilon_{r}.$

%	Under the self-similar transformation $\zeta=x/t,$ the Riemann problem presented by \eqref{2.1}-\eqref{2.2} transforms into the following:
%	\begin{equation}\label{2.3}
%		\begin{cases}
%			-\zeta\varrho_{\zeta} + (\varrho (\upsilon+\eta t))_{\zeta} = 0,
%			\\ -\zeta \left( \varrho \left( \upsilon+p\right) \right) _{\zeta}+\left( \varrho (\upsilon+ \eta t) \left( \upsilon+p\right) \right) _{\zeta}=0,
%			%			\\	-\zeta(\varrho (\upsilon+p))_{\zeta} + (\varrho(\upsilon+\eta t)(\upsilon+p))_{\zeta} = 0, 
%		\end{cases}
%	\end{equation}
%	and
%	\begin{equation}\label{2.4}	(\varrho, \upsilon) = 
%		\begin{cases}
%			(\varrho_{l}, \upsilon_{l}),\,\quad\zeta=-\infty,
%			\\  (\varrho_{r}, \upsilon_{r}),\quad\zeta=+\infty.
%		\end{cases}
%	\end{equation}
%For system \eqref{2.3} with \eqref{2.4}, besides the constant state $ (\varrho, \upsilon)(\zeta)=\text{constant} (\varrho>0),$
%	we have the singular solution, i.e., the contact discontinuity
%	\begin{equation}\label{2.6}
%		J : \, \zeta=\lambda_{2}=\sigma_{2}=\upsilon+\eta t,
%			\qquad \upsilon=\upsilon_{l},
%	\end{equation}
	
	\subsection{Shock wave and contact discontinuity:}
	For a bounded discontinuity at $x=x(t),$ the following Rankine-Hugoniot jump relations hold:
	\begin{equation}
		\begin{cases}\label{p10}
			-\sigma(t)[\varrho]+[\varrho \upsilon]=0,\\
			-\sigma(t)\left[ \varrho \left( \upsilon+p\right) \right] +\left[ \varrho \upsilon \left( \upsilon+p\right) \right]=0, 
		\end{cases}
	\end{equation}
	where $ \sigma(t) =x ^{\prime} (t)$ and $ [\varrho]=\varrho_{l}-\varrho_{r} $ denotes the jump of $\varrho$ across the discontinuity, in which
	$\varrho_{l}=\varrho(x(t)-0, t)$  and 	$\varrho_{r}=\varrho(x(t)+0, t)$ etc. In view of system of equations \eqref{p10} and Lax-entropy conditions, we get the following discontinuities:
	\\(i) Shock wave (discontinuity associated with $ \lambda_{1} $)
	\begin{equation}\label{p11}
		S : \begin{cases} \sigma_{1}=\upsilon_{l}-\frac{\varrho_{r}(p_{r}-p_{l})}{\varrho_{r}-\varrho_{l}}= \upsilon_{r}-\frac{\varrho_{l}(p_{r}-p_{l})}{\varrho_{r}-\varrho_{l}},
			\\ \upsilon_{r}+p_{r}=\upsilon_{l}+p_{l}, \qquad\varrho_{r}>\varrho_{l},
		\end{cases}
	\end{equation}
	(ii) Contact discontinuity (discontinuity associated with $ \lambda_{2} $) 
	\begin{equation}\label{p12}
		J :\,  \sigma_{2}=\upsilon_{l}=\upsilon_{r}.
	\end{equation}
%where $(\varrho, \upsilon)$ stands for the arbitrary right state. Now, from equation $\eqref{p11}_{2},$ we have
%\begin{equation}\label{p1.2}
%	\upsilon=-A\left( \frac{\varrho}{1-a\varrho}\right)^{\Gamma}-\frac{B}{\varrho^{\kappa}}+\upsilon_{l}+A\left( \frac{\varrho_{l}}{1-a\varrho_{l}}\right)^{\Gamma}-\frac{B}{\varrho_{l}^{\kappa}}
%\end{equation} Applying limit $\varrho\to 1/a$ in \eqref{p1.2}, we have $\upsilon\to -\infty$; which implies that the shock wave curve intersects the $\varrho$-axis.
%where $(\varrho, \upsilon)$ stands for the arbitrary right state which is connected to a given left state $(\varrho_{l}, \upsilon_l)$ on the right by shock wave curve or contact discontinuity curve.
%\subsection{Rarefaction wave:} The expression for the rarefaction wave is as
%\begin{equation}\label{2.9}
%	R : \begin{cases} \zeta=\lambda_{1}=\upsilon + \eta t -  \varrho p^{\prime},
%		\\ \upsilon+p=\upsilon_{l}+p_{l},\qquad \varrho < \varrho_{l}, 
%	\end{cases}
%\end{equation}
%along with $\frac{d\lambda_{1}}{d\varrho}=-\frac{(\Gamma^{2}+\Gamma)A\varrho^{\Gamma-1}}{(1-a\varrho)^{\Gamma+2}}+\frac{(\kappa^{2}-\kappa)B}{\varrho^{\kappa+1}}<0$ and $ p_{l}=A\left( \frac{\varrho_{l}}{1-a\varrho_{l}}\right) ^{\Gamma}-\frac{B}{\varrho_{l}^{\kappa}} . $ 	

\begin{figure}
	\begin{tikzpicture}
		\tikzset{-<-/.style={decoration={
					markings,
					mark=at position #1 with {\arrow{<}}},postaction={decorate}}}
		\centering
		
		%\draw [help lines] (0,0) grid (8,8);
		\draw [<->](7,0)node [right]{$\upsilon$} --(0,0) --(0,5) node [left]{$\varrho$};
		\draw [thick][blue] (6.9,0.1) to [out=178,in=287](0,4.2);
		\draw [ultra thick, dashed, orange] (3,0)--(3,4.4);
		\node [above]at(1.5,0.5){${I}$};
		\node [above]at(5,1.5){${II}$};
		\node [above]at(6.8,0.1){${R}$};
		\node [above]at(3.2,3.95){${J}$};
		\node [above]at(0.2,4){${S}$};
		\node [above]at(0,4.03){${\bullet}$};
		\node [above]at(3,0.62){$\bullet$};
		\node [above]at(3.7,0.68){$(\varrho_{l}, \upsilon_l)$};
		\node [below]at(3,0){$\upsilon=\upsilon_l$};
		\node[below] at (0,0){O};				
	\end{tikzpicture}
	\caption{Construction of Riemann solutions in $ (\varrho, \upsilon)-$phase plane.}
	\label{Fig.1}	
\end{figure}

%	Also, it may be observed that for a given left state $(\varrho_{l}, \upsilon_l),$ any  state $(\varrho, \upsilon)$ that is connected on the right by shock wave curve or rarefaction wave curve must lie on the following curve
%\begin{equation}
%	\upsilon+p=\upsilon_{l}+p_{l}.
%\end{equation}
Also, it may be observed that for a given left state  $(\varrho_{l}, \upsilon_{l})$ in the $ (\varrho, \upsilon)$-phase plane, the shock and rarefaction wave curves satisfy the same expression, \textit{i.e.,} $\upsilon+p=\upsilon_{l}+p_{l}.$
Indeed, we have $\upsilon_{\varrho}=-p^{\prime}<0,$ implying thereby that, in respect of $ \varrho,$ the shock and rarefaction wave curves are monotonically decreasing. Moreover, $	\lim\limits_{\varrho\longrightarrow 1/a} \upsilon = -\infty, $ which in turn, implies that the shock curve intersects the $\varrho$-axis and $\lim\limits_{\varrho\longrightarrow0^{+}}\upsilon = + \infty$ infers that the rarefaction wave curve has  $ \upsilon$-axis as the asymptotic line. Therefore, the elementary wave curves divide the first quadrant of $(\varrho, \upsilon)$-plane into the following two distinct regions (see, Figure \ref{Fig.1}): 
\begin{equation*}
	I=\left\{(\varrho, \upsilon)|\, 0\leq \upsilon < \upsilon_{l}, \, 0<\varrho\leq 1/a\right\},
\end{equation*}
\begin{equation*}
	II=\left\{(\varrho, \upsilon)|\, \upsilon \geq \upsilon_{l}, \, 0<\varrho\leq 1/a\right\}.
\end{equation*}By using the method of characteristic analysis, it can be analyzed that the Riemann solution consists of a contact discontinuity and a shock (rarefaction) wave when the right state $ (\varrho_{r}, \upsilon_{r}) $ lies in region-I (region-II), respectively (see, Figures \ref{99}(A) and \ref{99}(B)).

	\begin{figure}
	\centering
	
	\begin{subfigure}[b]{0.45\textwidth}
		\centering
		\begin{tikzpicture}
			\tikzset{->-/.style={decoration={
						markings,
						mark=at position .5 with {\arrow{>}}},postaction={decorate}}}
			%\draw [help lines] (0,0) grid (5,5);
			\draw [thick,<-](-2.5,5.1)node [right]{$t$} --(-2.5
			,3.9) ;
			
			\draw[thick, ->] (-2.4,0)--(4.4,0)node [below]{$ x $};
			\draw [thick, orange] (0,0)--(3.1,2.7);

			\draw [thick, blue] (0,0)--(0.7,4.3);
			\node[left] at (0.4,1.9) {$ (\varrho_{l}, \upsilon_{l}) $};
			\node[right] at (0.27,1.6) {$ (\varrho_{*}, \upsilon_{*}) $};
			\node[right] at (1.15,0.9) {$ (\varrho_{r}, \upsilon_{r}) $};
			%	\node[above] at (0.8,4.17) {$ dx/dt=\sigma_{2} $};
			\node[right] at (0.7,4.3) {$ S$};
			%	\node[above] at (3.25,2.57) {$ dx/dt=\sigma_{1} $};
			\node[right] at (3,2.7) {$ J $};
			
			\node[below] at (0,0) {$ O $};
		\end{tikzpicture}
		\subcaption{If $ (\varrho_{r}, \upsilon_{r}) $ lies in region-I.}
		%\label{78}	
	\end{subfigure}
	\hfill
	\begin{subfigure}[b]{0.45\textwidth}
		\centering
		\begin{tikzpicture}
			\tikzset{->-/.style={decoration={
						markings,
						mark=at position .5 with {\arrow{>}}},postaction={decorate}}}
			%\draw [help lines] (0,0) grid (5,5);
			\draw [thick,<-](-2.5,5.1)node [right]{$t$} --(-2.5
			,3.9) ;
			
			\draw[thick, ->] (-2.4,0)--(4.4,0)node [below]{$ x $};
			\draw [thick, orange] (0,0)--(3.1,2.65);

			\draw [thick, green] (0,0)--(1.4,4.2);
			\draw [thick, green] (0,0)--(1.1,4.3);
			%\draw [thick] (0,0) to [out=109,in=240](1.1,4.05);
			\draw [thick, green] (0,0)--(0.7,4.4);
			\draw [thick, green] (0,0)--(0.3,4.45);
			\node[left] at (0.2,1.5) {$ (\varrho_{l}, \upsilon_{l}) $};
			\node[right] at (0.57,1.8) {$ (\varrho_{*}, \upsilon_{*}) $};
			\node[right] at (1.15,0.85) {$ (\varrho_{r}, \upsilon_{r}) $};
			\node[right] at (-0.2,4.5) {$ R $};
			\node[right] at (3,2.65) {$ J $};
			%				\node[above] at (-1.2,4.35) {$ dx/dt=\lambda_{1}(\varrho_{l}, \upsilon_{l}) $};
			%				\node[above] at (1.5,4.15) {$ dx/dt=\lambda_{1}(\varrho_{*}, \upsilon_{*}) $};
			%				\node[above] at (4.55,2.47) {$ dx/dt=\sigma_{2} $};
			\node[below] at (0,0) {$ O $};
			%				\draw [thick,->](1.3,4.3)--(0.55,4.75) ;
			%				\draw [thick,->](-1.3,4.6)--(-0.9,3.75) ;
			
		\end{tikzpicture}
		\caption{ If $ (\varrho_{r}, \upsilon_{r}) $ lies in region-II.}
		%	\label{77}
	\end{subfigure}
	
	\caption{Riemann solutions of \eqref{p1}-\eqref{p2} and \eqref{p4} in the $ (x, t) $-plane.} 
	\label{99}
\end{figure}

	\section{\mbox{Limit Behavior of Riemann solutions }} 
		\subsection{Limit of Riemann solution when ($\varrho_{r}$, $\upsilon_{r}$) lies in region-I}
	
In view of $ a, A\longrightarrow 0,$ the behavior of Riemann solution is studied   when the right state $ (\varrho_{r}, \upsilon_{r}) $ lies in region-I. It can be easily seen that the curve $ \upsilon= -p+\upsilon_{l}+p_{l}$ tends to the curve $ \upsilon= \frac{B}{\varrho^{\kappa}}+\upsilon_{l}-\frac{B}{\varrho_{l}^{\kappa}}$ as $ a, A\longrightarrow 0.$ Also,  the curve $ \upsilon= \frac{B}{\varrho^{\kappa}}+\upsilon_{l}-\frac{B}{\varrho_{l}^{\kappa}}$ has an asymptotic line $ \upsilon= \upsilon_{l}-\frac{B}{\varrho_{l}^{\kappa}}$ parallel to $ \varrho-$axis.
Moreover, for $ \varrho>\varrho_{l}\, (\varrho<\varrho_{l}), $ the curve $ \upsilon=\frac{B}{\varrho^{\kappa}}+\upsilon_{l}-\frac{B}{\varrho_{l}^{\kappa}}$ lies right (left) to the curve $ \upsilon= -p+\upsilon_{l}+p_{l},$ respectively (see, Figure \ref{Fig.5}). Thus, the region-I is divided into the following two sub-regions I(a) and I(b): 
\begin{equation*}
	I(a)=\left\{(\varrho, \upsilon)|\, 0\leq\upsilon \leq \upsilon_{l}-\frac{B}{\varrho_{l}^{\kappa}}, \, 0<\varrho\leq 1/a\right\},
\end{equation*}
\begin{equation*}
	I(b)=\left\{(\varrho, \upsilon)|\, \upsilon_{l}-\frac{B}{\varrho_{l}^{\kappa}} <\upsilon<\upsilon_{l}, \, 0<\varrho\leq 1/a\right\}.
\end{equation*}
Now, we divide this discussion into the following two cases:\\\textbf{Case (i): Existence of $ \delta-$shock.} We establish the existence of $ \delta-$shock in the Riemann solution of \eqref{p1}-\eqref{p2} and \eqref{p4} when ($\varrho_{r}$, $\upsilon_{r}$) belongs to the region-I(a) as $ a, A\longrightarrow 0. $ For any $ a>0 $ and $ A>0,$ let the intermediate state $ (\varrho_{*}, \upsilon_{*}) $ be connected with $ (\varrho_{l}, \upsilon_{l}) $ by $ S $, and $ (\varrho_{r}, \upsilon_{r}) $ by $ J $ with speeds $ \sigma_{1},$ and $\sigma_{2},$ respectively (see, Figure \ref{99}(A)).  Then, we have
	\begin{equation}\label{p13}
		S : \begin{cases} \sigma_{1}=\upsilon_{l}-\frac{\varrho_{*}(p_{*}-p_{l})}{\varrho_{*}-\varrho_{l}}=\upsilon_{*}-\frac{\varrho_{l}(p_{*}-p_{l})}{\varrho_{*}-\varrho_{l}} ,
			\\ \upsilon_{*}+p_{*}=\upsilon_{l}+p_{l}, \qquad \varrho_{*}>\varrho_{l},
		\end{cases}
	\end{equation}
and
	\begin{equation}\label{p14}
	J : \, \sigma_{2}=\upsilon_{*}=\upsilon_{r},
	\qquad \upsilon_{*}=\upsilon_{r}.
\end{equation}
	Eliminating $ \upsilon_{*} $ from $ \eqref{p13}_{2} $ and $ \eqref{p14} $, we have
	\begin{equation}\label{p15}
		\upsilon_{r}=-p_{*}+\upsilon_{l}+p_{l}.
		%\upsilon_{r}=-A\left( \frac{\varrho_{*}}{1-a\varrho_{*}}\right) ^{\Gamma}+\frac{B}{\varrho_{*}^{\kappa}}+\upsilon_{l}+A\left( \frac{\varrho_{l}}{1-a\varrho_{l}}\right) ^{\Gamma}-\frac{B}{\varrho_{l}^{\kappa}}.
	\end{equation}\begin{figure}[ht!]
	\begin{tikzpicture}
		\tikzset{->-/.style={decoration={
					markings,
					mark=at position .5 with {\arrow{>}}},postaction={decorate}}}
		%\draw [help lines] (0,0) grid (10,10);
		\draw [thick,<-](3.1,5)node [left]{$\varrho$} --(3.1,-0.2) ;
		\draw [thick, blue] (9.5,0.5) to [out=178,in=275](3.1,4);
		\draw [thick, green] (9.5,0) to [out=178,in=272](4.4,4.5);
		\node [right] at (3.05,4.05){$ S $};
		\draw[ultra thick, dashed, orange] (6.6,-0.2)--(6.6,4.5);
		\draw [dashed, ultra thick, magenta] (4.3,-0.2)--(4.3,4.5);
		\draw[thick,->] (3.1,-0.2)--(10.1,-0.2);
		\node[right] at (9.87,-0.5){$\upsilon$};
		\node[right] at (9.25,0.7) {$ R $};
		\node[right] at (6.53,4.36){$ J $};
		\node[right] at (6.5,1.05) {$ (\varrho_{l}, \upsilon_{l}) $};
		\node[right] at (6.38,0.83) {$\bullet$};
		\node[right] at (2.88,3.95) {$\bullet$};
		\node[right] at (3.2,0.7) {I(a)};
		\node[right] at (5,0.7) {I(b)};
		\node[right] at (7.6,2) {II};
		\node[below] at (10.4,0.6) {$ \upsilon=\frac{B}{\varrho^{\kappa}}+\upsilon_{l}-\frac{B}{\varrho_{l}^\kappa}$};
		\draw[thick,->] (8.95,0.2)--(8.7,0.13);
		\node[below] at (4.4,-0.2) {$ \upsilon=\upsilon_{l}-\frac{B}{\varrho_{l}^\kappa}$};
		\node[below] at (3,-0.2) {$ O$};
		\node[right] at (6,-0.5) {$\upsilon=\upsilon_{l}$};
	\end{tikzpicture}
	\caption{Elementary wave curves in $ (\varrho, \upsilon)-$phase plane when $ a, A\longrightarrow 0.$}
	\label{Fig.5}
	\end{figure}
	\begin{lemma}\label{l1}
		\begin{equation*}\lim_{a, A\to 0} \varrho_{*} = + \infty.
		\end{equation*}
		\begin{proof}
			Taking  $\lim_{a, A\to 0} $ in \eqref{p15}, with the consideration that $\lim_{a, A\to 0}{ \varrho_{*}} = M $ $ \in $ $ ( \varrho_{l}, + \infty ),$ one can get $\upsilon_{l} - \upsilon_{r} = -\frac{B}{M^{\kappa}} +\frac{B}{\varrho_{l}^{\kappa}}<\frac{B}{\varrho_{l}^{\kappa}}$ which contradicts $ \upsilon_{r} \leqslant \upsilon_{l} -\frac{B}{\varrho_{l}^{\kappa}}.$ Hence, $ \lim_{a, A\to 0} \varrho_{*} = + \infty.$
		\end{proof}
		
	\end{lemma}

\begin{lemma}\label{l6}
	\begin{equation*}\lim_{a, A\to 0}{A\left( \frac{\varrho_{*}}{1-a\varrho_{*}}\right)^{\Gamma}} = \upsilon_{l} - \upsilon_{r} - \frac{B}{\varrho_{l}^{\kappa}}.
	\end{equation*}
\begin{proof}
	On employing  $\lim_{a, A\to 0} $ in \eqref{p15}, we get the desired result.
\end{proof}
\end{lemma}

	\begin{lemma}\label{l2}
		\begin{equation*}\lim_{a, A\to 0} \sigma_{1} = \lim_{a, A\to 0} \sigma_{2}= \upsilon_{r}.
		\end{equation*}
	\begin{proof}
		From $ \eqref{p13}_{1} $, we have 
		\begin{equation}
		\lim_{a, A\to 0} \sigma_{1} =\lim_{a, A\to 0} \left( \upsilon_{l}- \frac{\varrho_{*}(p_{*}-p_{l})}{\varrho_{*}-\varrho_{l}}\right)=  \upsilon_{l}- \left(\upsilon_{l}-\upsilon_{r}-\frac{B}{\varrho_{l}^{\kappa}}+\frac{B}{\varrho_{l}^{\kappa}} \right)=\upsilon_{r}.
		\end{equation} Hence, the proof is completed.
	\end{proof}	
	\end{lemma}
	
	\begin{lemma}\label{l3}
		\begin{equation}\label{p17}
		\lim_{a, A\to 0} \int_{\sigma_{1}}^{\sigma_{2}} \varrho_{*}d\zeta = \varrho_{l}(\upsilon_{l}-\upsilon_{r})\neq0.
		\end{equation}
		\begin{proof}
			In view of $ \eqref{p10}_{1}$ for $ S $ and $ J $, we get the following
			\begin{equation}\label{p18}
				\begin{cases}
					-\sigma_{1}(\varrho_{l}-\varrho_{*})+(\varrho_{l}\upsilon_{l}-\varrho_{*}\upsilon_{*})=0,
					\\ -\sigma_{2}(\varrho_{*}-\varrho_{r})+(\varrho_{*}\upsilon_{*}-\varrho_{r}\upsilon_{r})=0.
				\end{cases}
			\end{equation}
			Furthermore, addition of  $ \eqref{p18}_{1} $ and $ \eqref{p18}_{2} $ yields 
			\begin{equation}\label{p19}
				\varrho_{*}(\sigma_{2}-\sigma_{1})=\sigma_{2}\varrho_{r}-\sigma_{1}\varrho_{l}+\varrho_{l}\upsilon_{l}-\varrho_{r}\upsilon_{r},
			\end{equation}
			which, in turn,  implies the required result.
		\end{proof}
	\end{lemma}
\begin{remark}
	Lemma \ref{l1} demonstrates that  the intermediate density $\varrho_{*}$ approaches to infinity as $a, A \to 0$ for $\upsilon_{r} \leqslant \upsilon_{l}-\frac{B}{\varrho_{l}^{\kappa}}.$ Put differently, we can interpret this as $\varrho_{*}$ transitioning into a Dirac $\delta$-function as  $a, A \to 0$. Also, lemma \ref{l2} implies that   the shock curve coincide with contact discontinuity  curve as  $a, A \to 0$ for $\upsilon_{r} \leqslant \upsilon_{l}-\frac{B}{\varrho_{l}^{\kappa}}.$
\end{remark}

	\begin{theorem}
		Let $\upsilon_{r} \leqslant \upsilon_{l}-\frac{B}{\varrho_{l}^{\kappa}},$  $ (\varrho_{r},\upsilon_{r})$ $\in$  $ I(a)(\varrho_{l}, \upsilon_{l}) $ and for all fixed $a, A>0$, $ (\varrho_{a}, u_{a}) $ be the $ S+J $ Riemann solution to the system \eqref{p1}-\eqref{p2} and \eqref{p4}. Then 
		\begin{equation}\label{p20}
			\lim_{a, A\to 0}\upsilon_{a}(x,t) = 	\begin{cases}
				\upsilon_{l} , \hspace{1cm} x < \upsilon_{r}t,
				\\ \upsilon_{r}, \hspace{1cm} x = \upsilon_{r}t,
				\\\upsilon_{r}, \hspace{0.95cm} x > \upsilon_{r}t,
			\end{cases}
		\end{equation}and $ \varrho_{a} $ converges in  distributional sense. The limit function is the sum of a Dirac-delta function and a step function supported on the curve $ x = \upsilon_{r}t $ with weight $ \varrho_{l}(\upsilon_{l}-\upsilon_{r})t $, as $ a, A\to 0.$ 
		\begin{proof}
			(i) For any $a,A>0,$ the Riemann solution $ S+J $ to the system \eqref{p1}-\eqref{p2} and \eqref{p4} can be expressed as 
			\begin{equation}\label{p21}
				(\varrho_{a}, \upsilon_{a})(\zeta:=x/t) = 	\begin{cases}
					(\varrho_{l}, \upsilon_{l}) , \hspace{2cm} \zeta < \sigma_{1},
					\\ (\varrho_{*}(\zeta), u_{*}(\zeta)), \hspace{0.9cm} \sigma_{1} < \zeta < \sigma_{2},
					\\(\varrho_{r}, \upsilon_{r}), \hspace{1.9cm} \zeta > \sigma_{2}.
				\end{cases}
			\end{equation}
			Now, from \eqref{p6}, we have the following weak formulation
			\begin{equation}\label{p22}
				- \int_{-\infty}^{+\infty} \varrho_{a}(\upsilon_{a} - \zeta ) \phi' d \zeta + \int_{-\infty}^{+\infty} \varrho_{a} \phi d \zeta =0,
			\end{equation}
			for any $ \phi\in C_{0}^{1}(-\infty, +\infty).$ The limit \eqref{p20} can be directly obtained from \eqref{p21}.
			\\ (ii) Consider
			\begin{equation}\label{p23}
				\int_{-\infty}^{+ \infty}\varrho_{a}(\upsilon_{a} -\zeta) \phi' d \zeta = \left( \int_{-\infty}^{\sigma_{1}} + \int_{\sigma_{1}}^{\sigma_{2}} + \int_{\sigma_{2}}^{+ \infty}\right) \varrho_{a}(\upsilon_{a} -\zeta) \phi' d \zeta.
			\end{equation}
			Also, we have
			\begin{eqnarray}\label{p24}
				&&\lim_{a, A\to 0}\left( \int_{-\infty}^{\sigma_{1}}\varrho_{a}(\upsilon_{a} -\zeta) \phi' d \zeta+\int_{\sigma_{2}}^{+ \infty}\varrho_{a}(\upsilon_{a} -\zeta) \phi' d \zeta\right)  \nonumber \\ && = \lim_{a, A\to 0}\left( \int_{-\infty}^{\sigma_{1}}\varrho_{l}(\upsilon_{l} -\zeta) \phi' d \zeta+\int_{\sigma_{2}}^{+ \infty}\varrho_{r}(\upsilon_{r} -\zeta) \phi' d \zeta\right)  \nonumber \\ && =\varrho_{l}(\upsilon_{l}-\upsilon_{r})\phi(\upsilon_{r}) +\int_{-\infty}^{+\infty}H(\zeta-\upsilon_{r})\phi d\zeta,
			\end{eqnarray}
%			\begin{eqnarray}\label{3.14}
%				&&\lim_{a, A\to 0}\left( \int_{-\infty}^{\sigma_{1}}\varrho_{a}(\upsilon_{a}+\eta t -\zeta) \phi' d \zeta+\int_{\sigma_{1}}^{\sigma_{2}}\varrho_{a}(\upsilon_{a}+\eta t -\zeta) \phi' d \zeta+\int_{\sigma_{2}}^{+ \infty}\varrho_{a}(\upsilon_{a}+\eta t -\zeta) \phi' d \zeta\right)  \nonumber \\ && = \lim_{a, A\to 0}\left( \int_{-\infty}^{\sigma_{1}}\varrho_{l}(\upsilon_{l}+\eta t -\zeta) \phi' d \zeta+\int_{\sigma_{1}}^{\sigma_{2}}\varrho_{*}(\upsilon_{*}+\eta t -\zeta) \phi' d \zeta+\int_{\sigma_{2}}^{+ \infty}\varrho_{r}(\upsilon_{r}+\eta t -\zeta) \phi' d \zeta\right)  \nonumber \\ && =\varrho_{l}(\upsilon_{l}-\upsilon_{r})\phi(\upsilon_{r}+\eta t) +\int_{-\infty}^{+\infty}H(\zeta-(\upsilon_{r}+\eta t))\phi d\zeta,
%			\end{eqnarray}
			where
			\begin{equation}
				H(x)=\begin{cases}
					\varrho_{l}, \qquad x<0,
					\\ \varrho_{r}, \qquad x>0.
				\end{cases}
			\end{equation} A simple computation implies that
		\begin{equation}\label{p25}
			\lim_{a, A\to 0}\int_{\sigma_{1}}^{\sigma_{2}}\varrho_{a}(\upsilon_{a} -\zeta) \phi' d \zeta=0.
		\end{equation}
%			Also, consider 
%			\begin{eqnarray}\label{3.15}
%				&&\lim_{a, A\to 0}\int_{\sigma_{1}}^{\sigma_{2}}\varrho_{a}(\upsilon_{a}+\eta t -\zeta) \phi' d \zeta  = \lim_{a, A\to 0}\int_{\sigma_{1}}^{\sigma_{2}}\varrho_{*}(\upsilon_{*}+\eta t -\zeta) \phi' d \zeta \nonumber \\ && = \lim_{a, A\to 0}\varrho_{*}(\sigma_{2}-\sigma_{1})\times \nonumber \\ &&\lim_{a, A\to 0}\left( (\upsilon_{*}+\eta t) \frac{\phi(\sigma_{2})-\phi(\sigma_{1})}{\sigma_{2}-\sigma_{1}}-\frac{\sigma_{2}\phi(\sigma_{2})-\sigma_{1}\phi(\sigma_{1})}{\sigma_{2}-\sigma_{1}}+\frac{1}{\sigma_{2}-\sigma_{1}}\int_{\sigma_{1}}^{\sigma_{2}}\phi d\zeta\right) \nonumber \\ && =\varrho_{l}(\upsilon_{l}-\upsilon_{r})((\upsilon_{r}+\eta t)\phi'(\upsilon_{r}+\eta t)-(\upsilon_{r}+\eta t)\phi'(\upsilon_{r}+\eta t)-\phi(\upsilon_{r}+\eta t)+\phi(\upsilon_{r}+\eta t)) =0.
%			\end{eqnarray}
			Using  equations \eqref{p24},  \eqref{p25} in  \eqref{p22}, we have
			\begin{equation}\label{p26}
				\lim_{a, A\to 0}\int_{-\infty}^{+\infty} \varrho_{a} \phi d \zeta=\varrho_{l}(\upsilon_{l}-\upsilon_{r})\phi(\upsilon_{r}) +\int_{-\infty}^{+\infty}H(\zeta-\upsilon_{r})\phi d\zeta.
			\end{equation}(iii) For any $ \psi \in  C_{0}^{\infty}(\mathbb{R}\times \mathbb{R}^{+}),$ in view of the limit of $\varrho_{a}$ depending on $t,$ and  \eqref{p26}, we have
		\begin{eqnarray}\label{p27}
		&& \lim_{a, A\to 0}	\int_{0}^{+\infty}\int_{-\infty}^{+\infty} \varrho_{a}(x/t)\psi(x, t) dx dt \nonumber = \lim_{a, A\to 0}	\int_{0}^{+\infty} t\left(  \int_{-\infty}^{+\infty} \varrho_{a}(\zeta)\psi(\zeta t, t) d\zeta \right)  dt \nonumber \\ &&= \int_{0}^{+\infty}t \left( \varrho_{l}(\upsilon_{l}-\upsilon_{r})\psi(\upsilon_{r}t, t)  + \int_{-\infty}^{+\infty} H(\zeta- \upsilon_{r}) \psi(\zeta t, t) d\zeta  \right) dt \nonumber \\ && = \int_{0}^{+\infty} \varrho_{l}(\upsilon_{l}-\upsilon_{r})t\psi(\upsilon_{r}t, t)dt+ \int_{0}^{+\infty}\int_{-\infty}^{+\infty} H(x- \upsilon_{r} t) \psi(x, t) dx dt
	\end{eqnarray}
in which
\begin{equation*}
\int_{0}^{+\infty} \varrho_{l}(\upsilon_{l}-\upsilon_{r})t\psi(\upsilon_{r}t, t)dt=\langle w(\cdot)\delta_{C}, \psi(\cdot,\cdot) \rangle
\end{equation*}
with \begin{equation}
w(t)= \varrho_{l}(\upsilon_{l}-\upsilon_{r})t.
\end{equation}
\end{proof}
\end{theorem}
Hence, we conclude that the Riemann solution $S+J$ converges to $\delta$-shock solution whenever $\upsilon_{r}\leqslant \upsilon_{l}-\frac{B}{\varrho_{l}^{\kappa}} $ and $a, A\to 0.$ 
	\\ \textbf{Case (ii).} Here, we analyze behavior of the Riemann solution to \eqref{p1}-\eqref{p2} and \eqref{p4}  whenever $ \upsilon_{l}-\frac{B}{\varrho_{l}^{\kappa}} <\upsilon_{r}<\upsilon_{l} $. 
	For any $ a>0 $ and $ A>0,$ the intermediate state $ (\varrho_{*}, \upsilon_{*}) $ is connected with $ (\varrho_{l}, \upsilon_{l}) $ by $ S $  and $ (\varrho_{r}, \upsilon_{r}) $ by $ J $ with speeds  $ \sigma_{1} $  and $ \sigma_{2},$ respectively.  Then, we have
	\begin{equation}\label{p28}
		S : \begin{cases} \sigma_{1}=\upsilon_{l}-\frac{\varrho_{*}(p_{*}-p_{l})}{\varrho_{*}-\varrho_{l}}=\upsilon_{*}-\frac{\varrho_{l}(p_{*}-p_{l})}{\varrho_{*}-\varrho_{l}},
			\\ \upsilon_{*}+p_{*}=\upsilon_{l}+p_{l},\qquad \varrho_{*}>\varrho_{l},
		\end{cases}
	\end{equation}
	and 
	\begin{equation}\label{p29}
		J : \, \sigma_{2}=\upsilon_{*}=\upsilon_{r},
			\qquad \upsilon_{*}=\upsilon_{r}.
	\end{equation} 
	From $ \eqref{p28}_{2} $ and $ \eqref{p29} $, $\varrho_{*}$ satisfies
	\begin{equation}\label{p30}
		\upsilon_{r}=-p_{*}+\upsilon_{l}+p_{l}.
	\end{equation}
On employing  $\lim_{a, A\to 0} $ in \eqref{p30}, with the assumption that $\lim_{a, A\to 0} \varrho_{*} = \infty, $ one can obtain $\upsilon_{r}\leqslant \upsilon_{l}-\frac{B}{\varrho_{l}^{\kappa}}; $ which gives a contradiction with $ \upsilon_{l}-\frac{B}{\varrho_{l}^{\kappa}} <\upsilon_{r}. $ Hence, $\lim_{a, A\to 0} \varrho_{*} = \text{finite}. $

	\subsection{Limit of Riemann solutions when ($\varrho_{r}$, $\upsilon_{r}$) lies in region-II}
	We study the limit $ a, A\to 0 $ of the Riemann solution to \eqref{p1}-\eqref{p2} with \eqref{p4} in the case  $ \upsilon_{l}\leq\upsilon_{r} $.  For any $ a>0 $ and $ A>0 $, the intermediate state $ (\varrho_{*}, \upsilon_{*}) $ is connected with $ (\varrho_{l}, \upsilon_{l}) $ by $ R $ and $ (\varrho_{r}, \upsilon_{r}) $ by $ J $, respectively. Then, we have
	\begin{equation}\label{p31}
		R : \begin{cases} \zeta=\lambda_{1}=\upsilon  - \varrho p^{\prime},
			\\ \upsilon+p=\upsilon_{l}+p_{l}, \qquad \varrho_{*}<\varrho < \varrho_{l},\end{cases}
	\end{equation}
	and 
	\begin{equation}\label{p32}
		J : \, \sigma_{2}=\upsilon_{*}=\upsilon_{r},
		\qquad \upsilon_{*}=\upsilon_{r}.
	\end{equation} 
	From $ \eqref{p31}_{2} $ and $ \eqref{p32} $, $\varrho_{*}$ satisfies
	\begin{equation}\label{p33}
		\upsilon_{r}=-p_{*}+\upsilon_{l}+p_{l}.		
	\end{equation}
	Taking  $\lim_{a, A\to 0} $ in \eqref{p33}, with the consideration that $\lim_{a, A\to 0} \varrho_{*} = 0, $ one can obtain $ \upsilon_{r} = +\infty,$ which is absurd. Hence, there is no vacuum in the Riemann solution of the system \eqref{p1}-\eqref{p2} and \eqref{p4}.

 \section{Numerical Simulations}
 Here, we present  numerical simulations to verify our theoretical analysis. The system \eqref{p1}-\eqref{p2} can be written as
 \begin{equation}\label{p101}
 	U_{t}+BU_{x}=0,
 \end{equation} where
\begin{equation}\label{p102}
	U=\begin{bmatrix}
		\varrho \\ \upsilon
	\end{bmatrix}, \qquad B=\begin{bmatrix}
\upsilon & \varrho \\ 0 & \upsilon-\varrho p^{\prime}
\end{bmatrix}.
\end{equation} Further, we have 
 \begin{equation}\label{p103}
 	B=R\Lambda L,
 \end{equation} with 
\begin{equation}\label{p104}
	R=\begin{bmatrix}
		\frac{1}{ p^{\prime}} & -\frac{1}{ p^{\prime}}\\ 0 & 1
	\end{bmatrix}, \qquad \Lambda=\begin{bmatrix}
	\upsilon & 0 \\ 0 & \upsilon-\varrho p^{\prime}
\end{bmatrix}, \qquad L=\begin{bmatrix}
p^{\prime} & 1\\ 0 & 1
\end{bmatrix}.
\end{equation}

Now, we employ the first order upwind scheme (see, \cite{Lu1996, chakravarthy1980, liu2025}) based on the split coefficient matrix method  which is represented as
\begin{equation}\label{p105}
	U_{j}^{n+1}=U_{j}^{n}-\frac{\Delta t}{\Delta x} \left\lbrace B_{j}^{-}(U_{j+1}^{n}-U_{j}^{n})+B_{j}^{+}(U_{j}^{n}-U_{j-1}^{n})\right\rbrace, 
\end{equation} 
 where 
 \begin{equation}\label{p106}
 	B_{j}^{-}=\frac{B_{j}^{n}-|B_{j}^{n}|}{2}, \qquad B_{j}^{+}=\frac{B_{j}^{n}+|B_{j}^{n}|}{2}, \qquad |B_{j}^{n}|=R_{j}^{n}|\Lambda_{j}^{n}|L_{j}^{n}.
 \end{equation}

\noindent \textbf{Case (i).} To verify the formation of delta shock wave in the limiting case, first we consider the following initial data 
 \begin{equation}\label{p107}
 	(\varrho_{l, r}, \upsilon_{l, r})= \begin{cases}
 		(1, 5), \qquad x<0,
 		\\ (1, 2), \qquad x>0,
 	\end{cases}
 \end{equation} and $ \kappa =0.25, \Gamma=2, B=1.$  The numerical results for different pairs of values of $a$ and $A$ are shown in figures (see, Figures \ref{p111}-\ref{p112}). Additionally, Figures \ref{p113}-\ref{p114} illustrate how density and velocity behave as $a$ and $A$ decrease for $\kappa=0.75,$ $\Gamma=3$ and the initial datum satisfying \eqref{p107}.
\begin{figure}[h!]
	\begin{minipage}{0.24\textwidth}
		\centering
		\includegraphics[width=\linewidth]{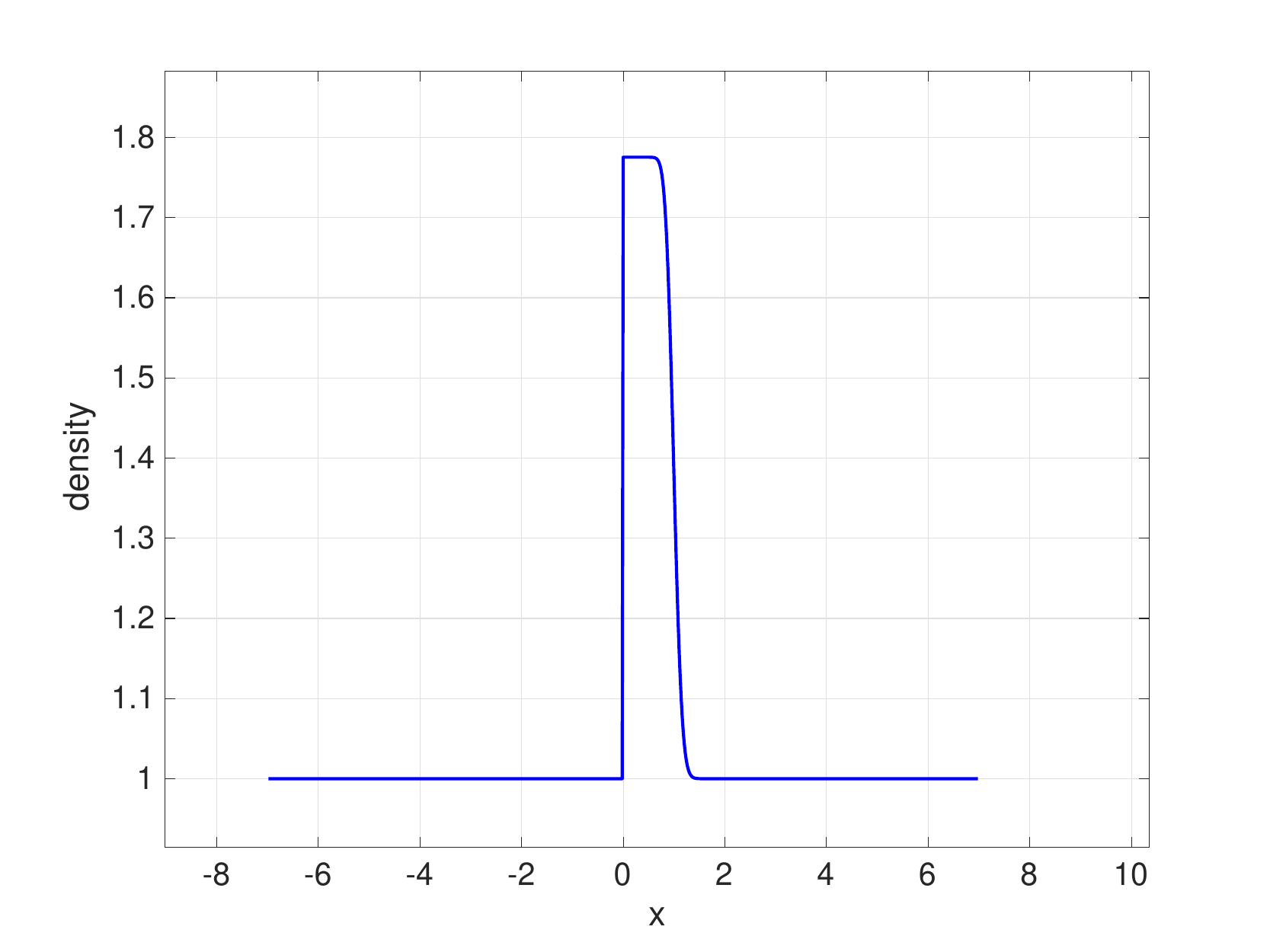}
		%	\caption*{(a)} % Optional label for the subfigure
	\end{minipage}\hfill
	\begin{minipage}{0.24\textwidth}
		\centering
		\includegraphics[width=\linewidth]{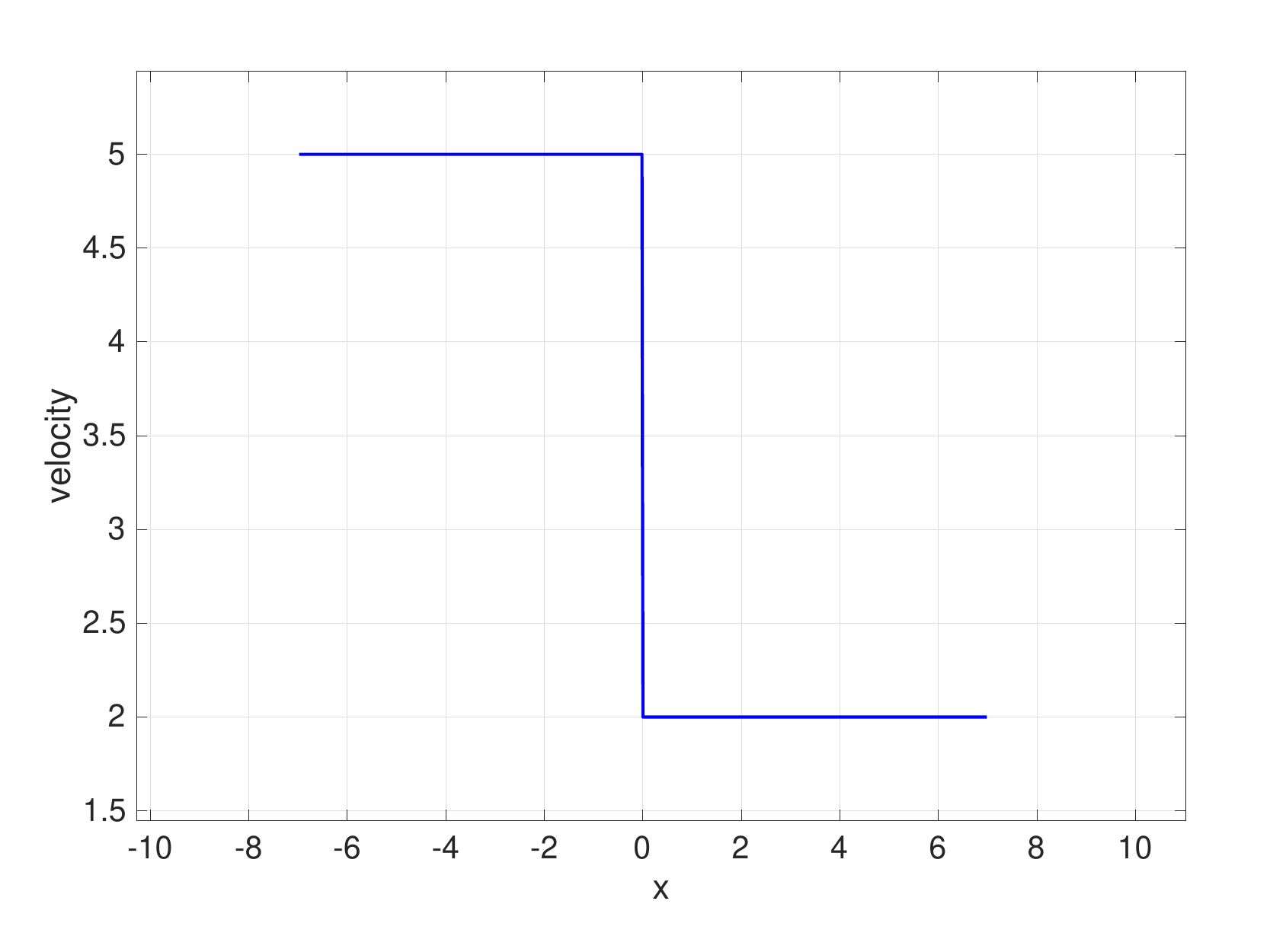}
		%	\caption*{(b)} % Optional label for the subfigure
	\end{minipage}\hfill
	\begin{minipage}{0.24\textwidth}
		\centering
		\includegraphics[width=\linewidth]{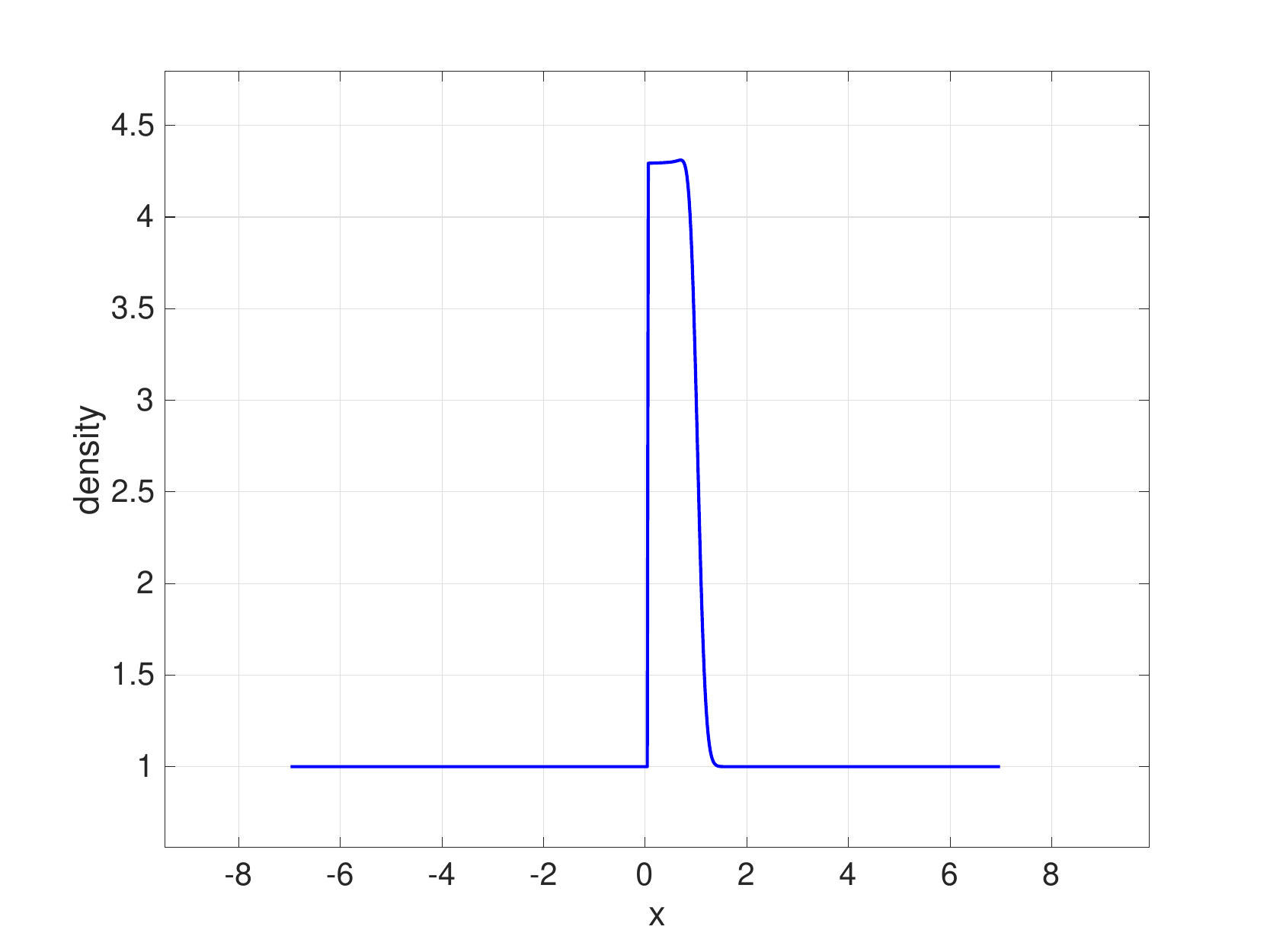}
		%	\caption*{(c)} % Optional label for the subfigure
	\end{minipage}\hfill
	\begin{minipage}{0.24\textwidth}
		\centering
		\includegraphics[width=\linewidth]{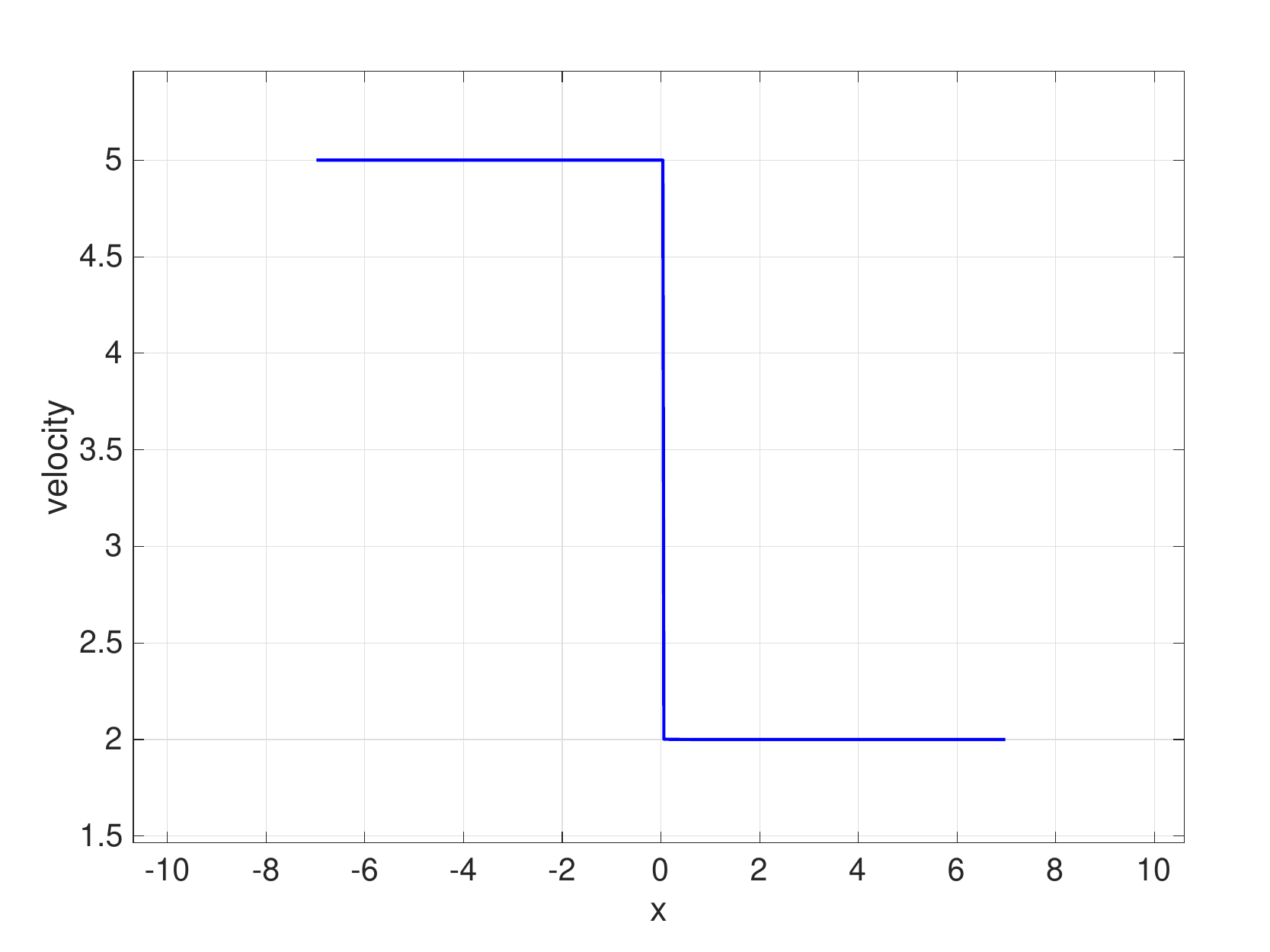}
		%	\caption*{(d)} % Optional label for the subfigure
	\end{minipage}
	\caption{Density and velocity for A=1, a=0.01 and A=0.1, a=0.001, respectively.}
	\label{p111}
\end{figure}
\begin{figure}[h!]
	\begin{minipage}{0.24\textwidth}
		\centering
		\includegraphics[width=\linewidth]{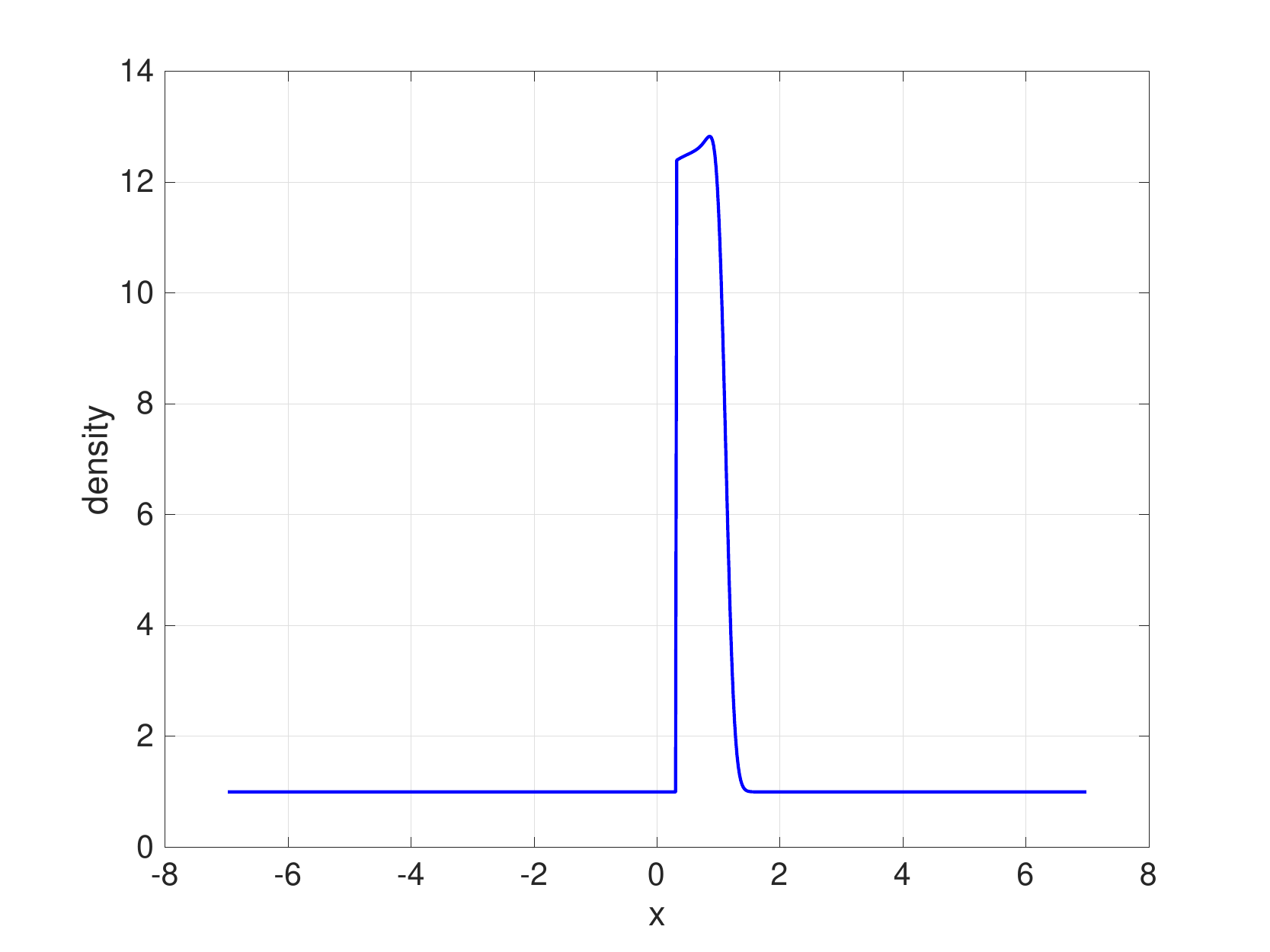}
		%	\caption*{(a)} % Optional label for the subfigure
	\end{minipage}\hfill
	\begin{minipage}{0.24\textwidth}
		\centering
		\includegraphics[width=\linewidth]{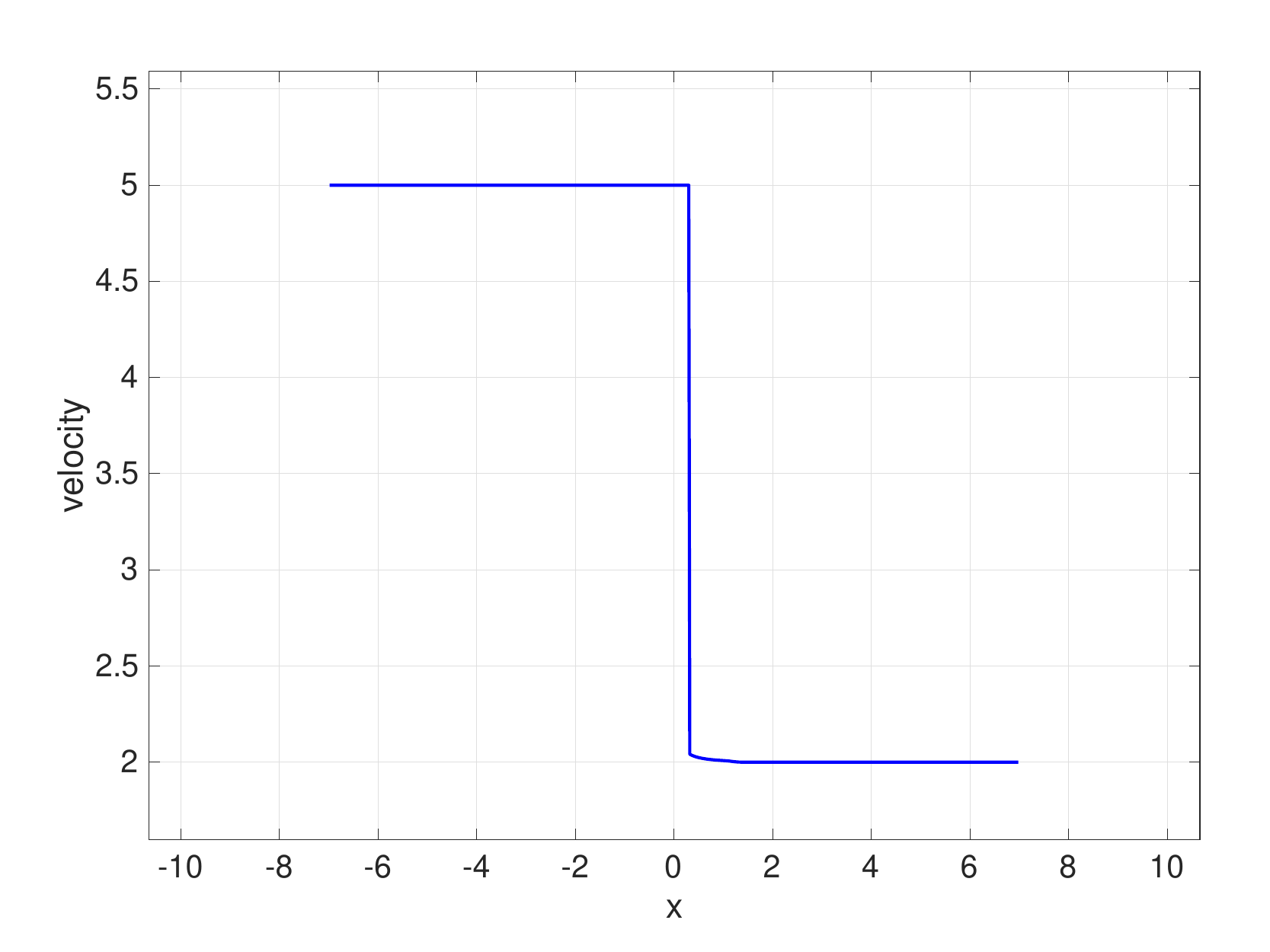}
		%	\caption*{(b)} % Optional label for the subfigure
	\end{minipage}\hfill
	\begin{minipage}{0.24\textwidth}
		\centering
		\includegraphics[width=\linewidth]{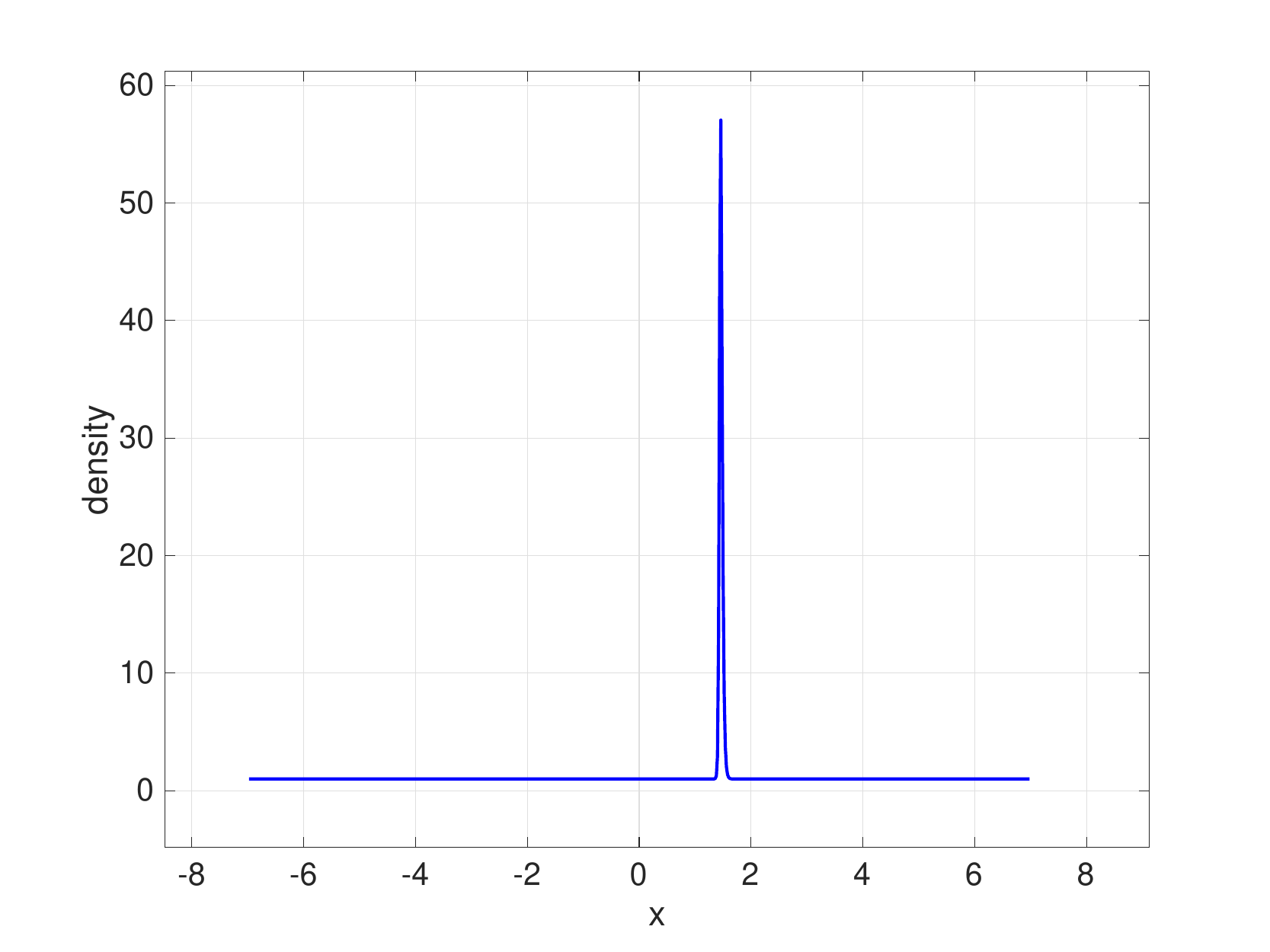}
		%	\caption*{(c)} % Optional label for the subfigure
	\end{minipage}\hfill
	\begin{minipage}{0.24\textwidth}
		\centering
		\includegraphics[width=\linewidth]{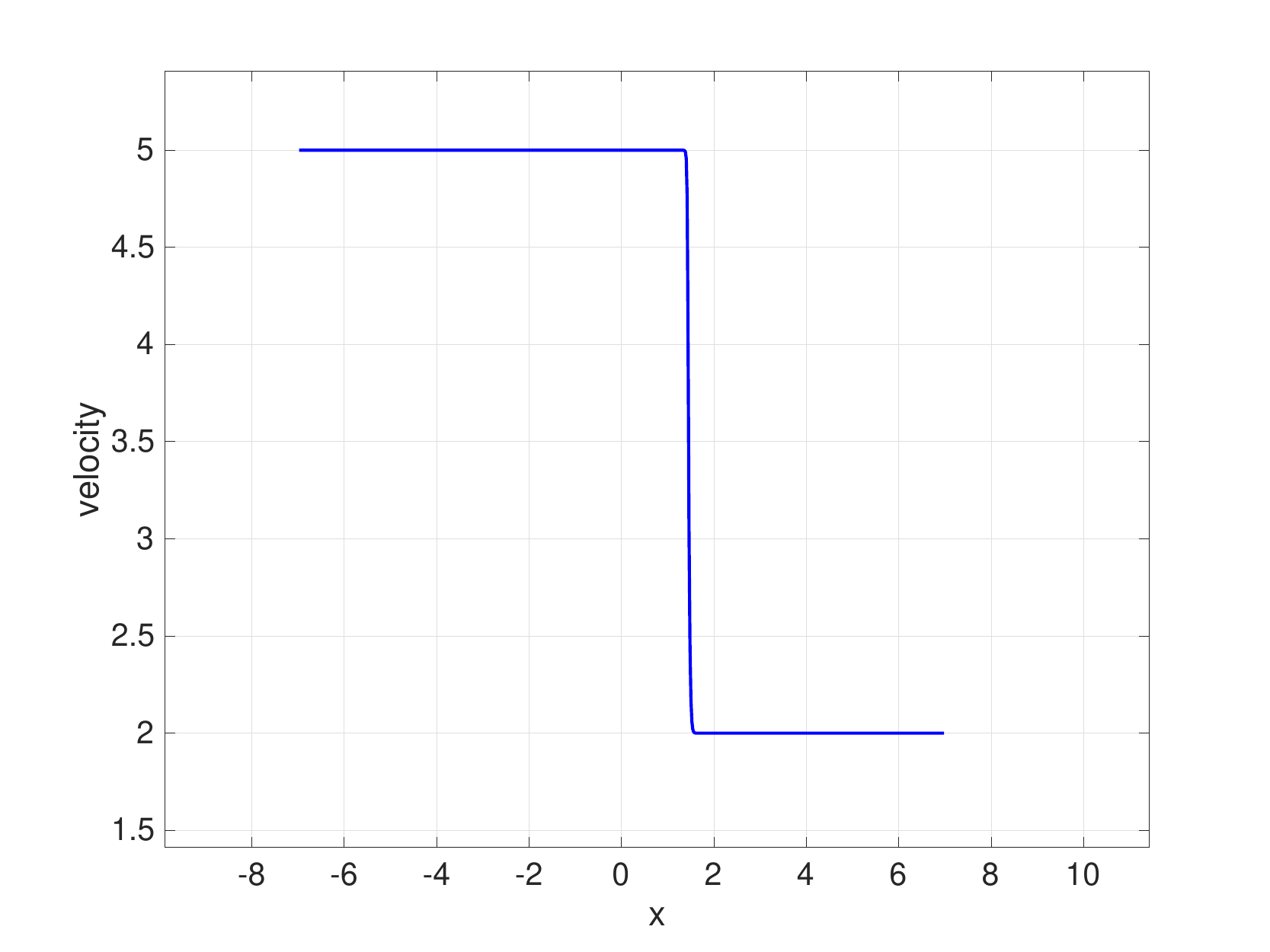}
		%	\caption*{(d)} % Optional label for the subfigure
	\end{minipage}
	\caption{Density and velocity for A=0.01, a=0.0001 and A=0.0001, a=0.000001, respectively.}
	\label{p112}
\end{figure}
\begin{figure}[h!]
	\begin{minipage}{0.24\textwidth}
		\centering
		\includegraphics[width=\linewidth]{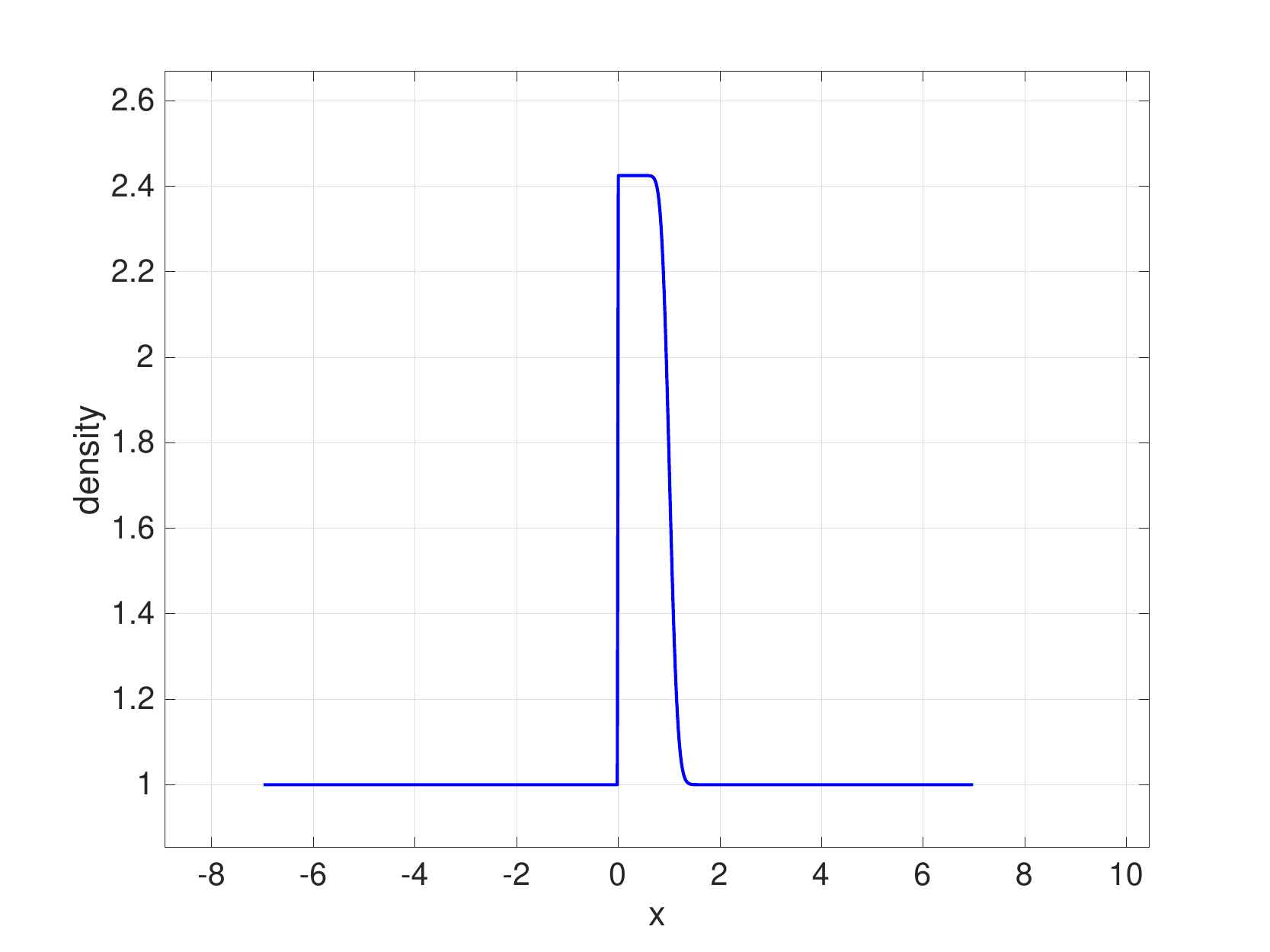}
		%	\caption*{(a)} % Optional label for the subfigure
	\end{minipage}\hfill
	\begin{minipage}{0.24\textwidth}
		\centering
		\includegraphics[width=\linewidth]{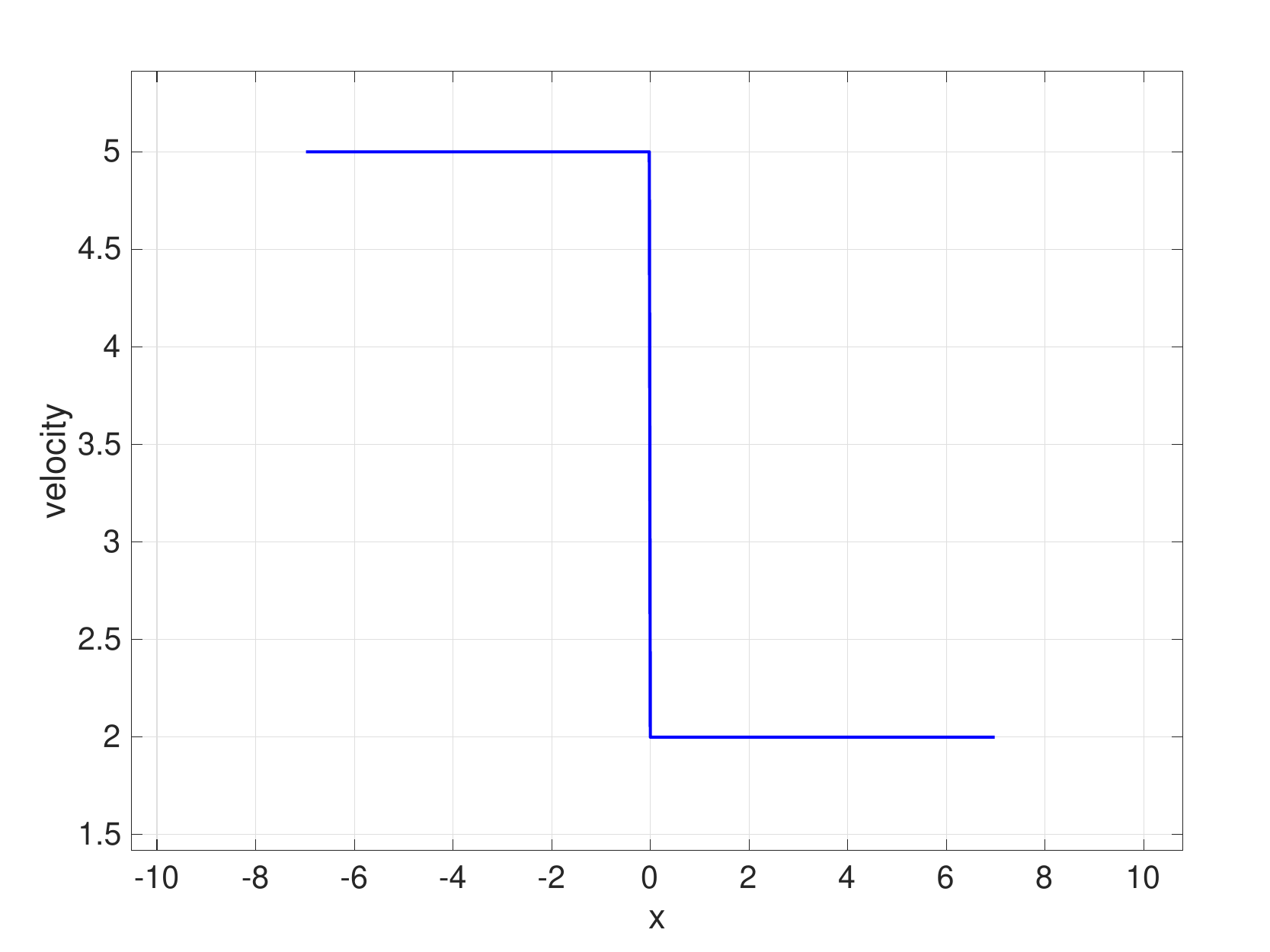}
		%	\caption*{(b)} % Optional label for the subfigure
	\end{minipage}\hfill
	\begin{minipage}{0.24\textwidth}
		\centering
		\includegraphics[width=\linewidth]{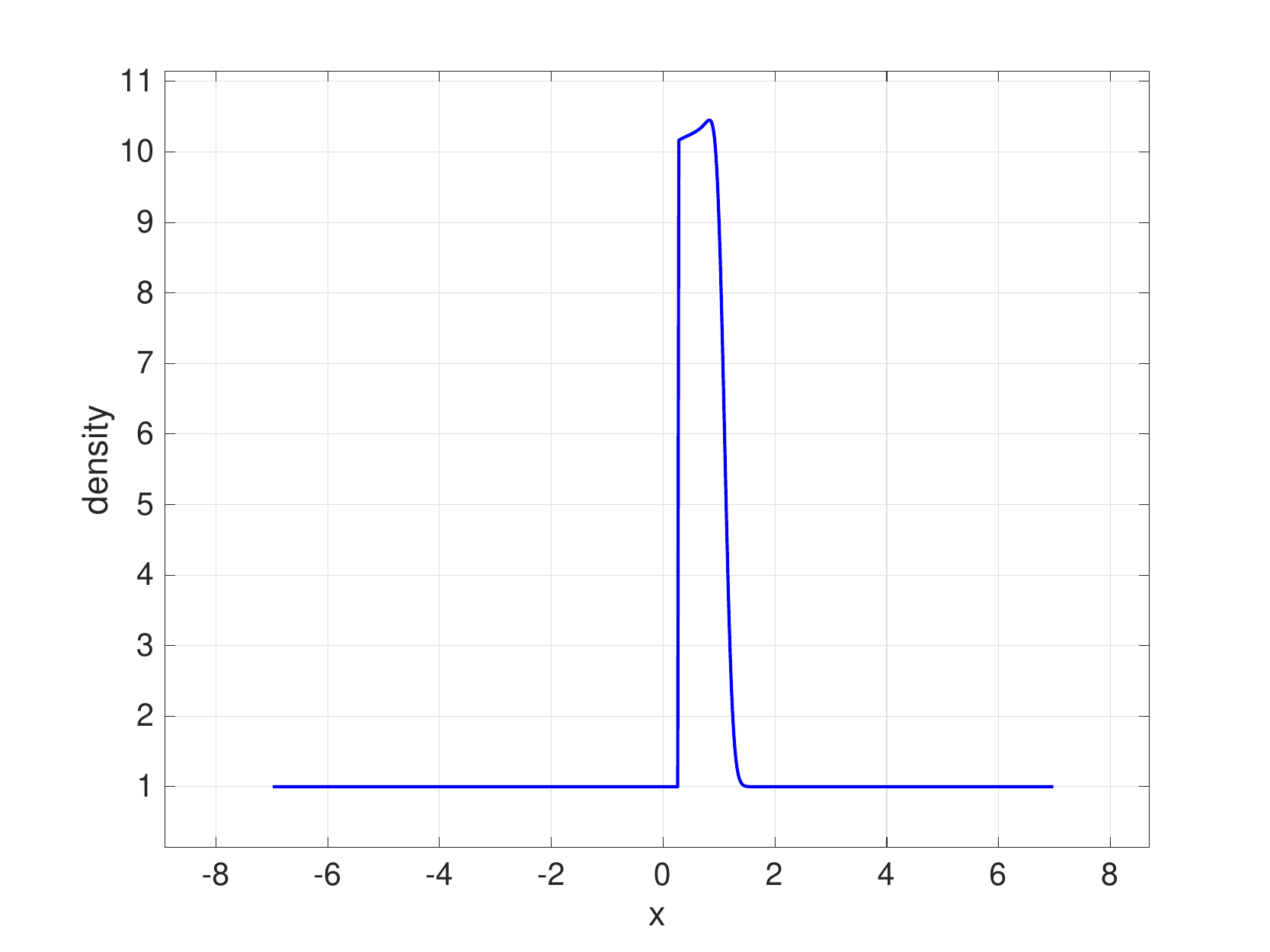}
		%	\caption*{(c)} % Optional label for the subfigure
	\end{minipage}\hfill
	\begin{minipage}{0.24\textwidth}
		\centering
		\includegraphics[width=\linewidth]{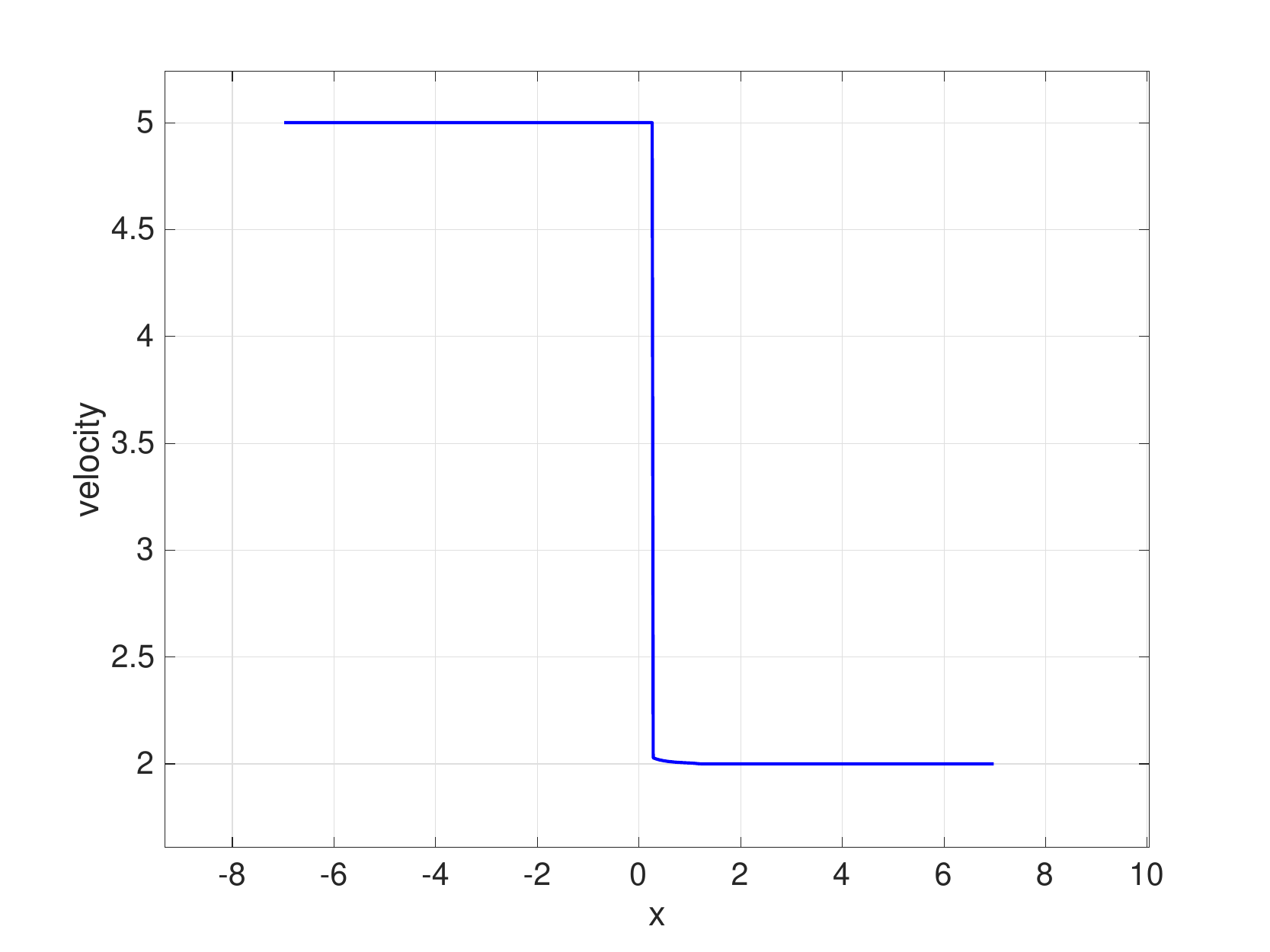}
		%	\caption*{(d)} % Optional label for the subfigure
	\end{minipage}
	\caption{Density and velocity for A=0.1, a=0.01 and A=0.001, a=0.0001, respectively.}
	\label{p113}
\end{figure}

\begin{figure}[h!]
	\begin{minipage}{0.24\textwidth}
		\centering
		\includegraphics[width=\linewidth]{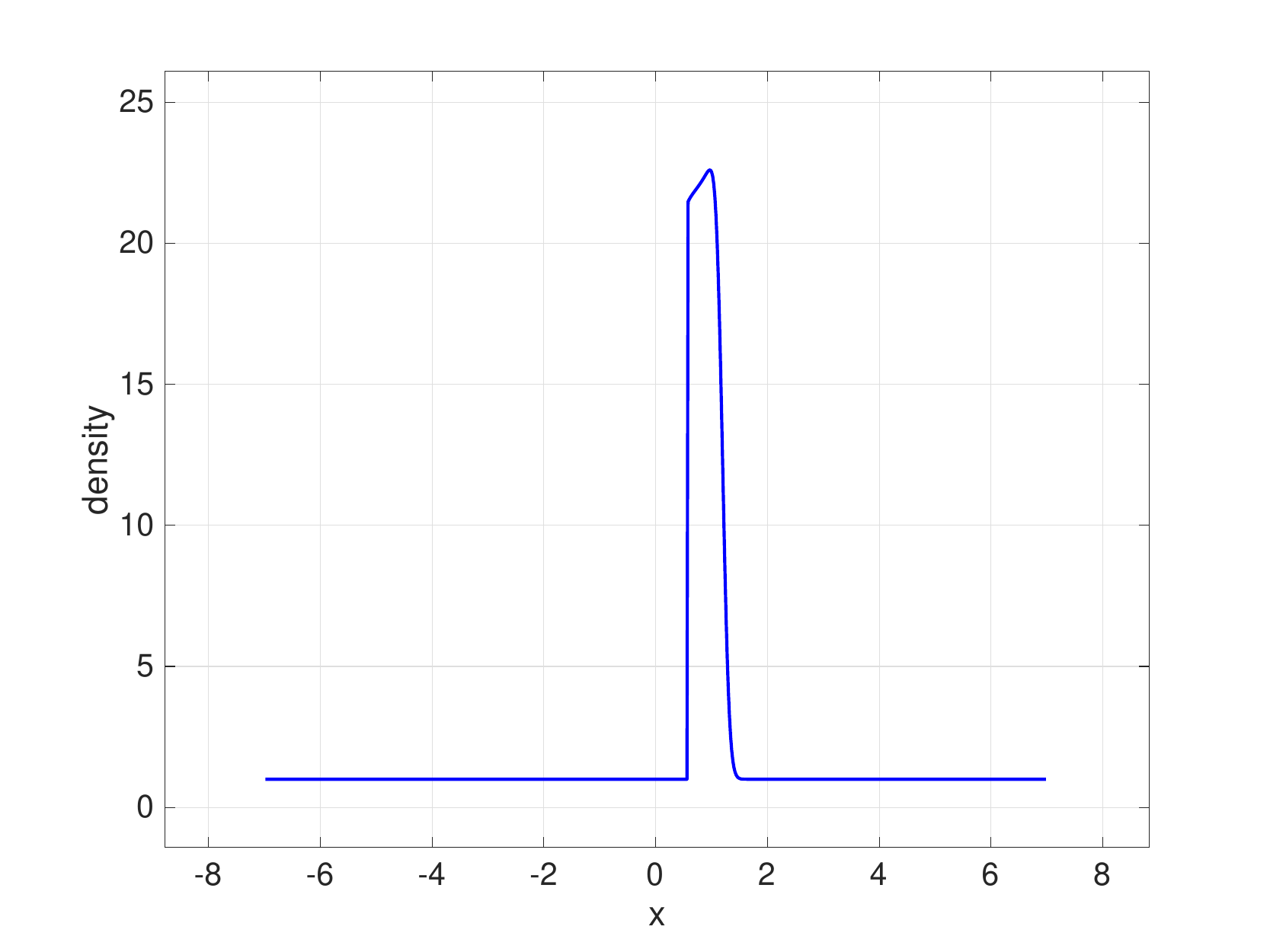}
		%	\caption*{(a)} % Optional label for the subfigure
	\end{minipage}\hfill
	\begin{minipage}{0.24\textwidth}
		\centering
		\includegraphics[width=\linewidth]{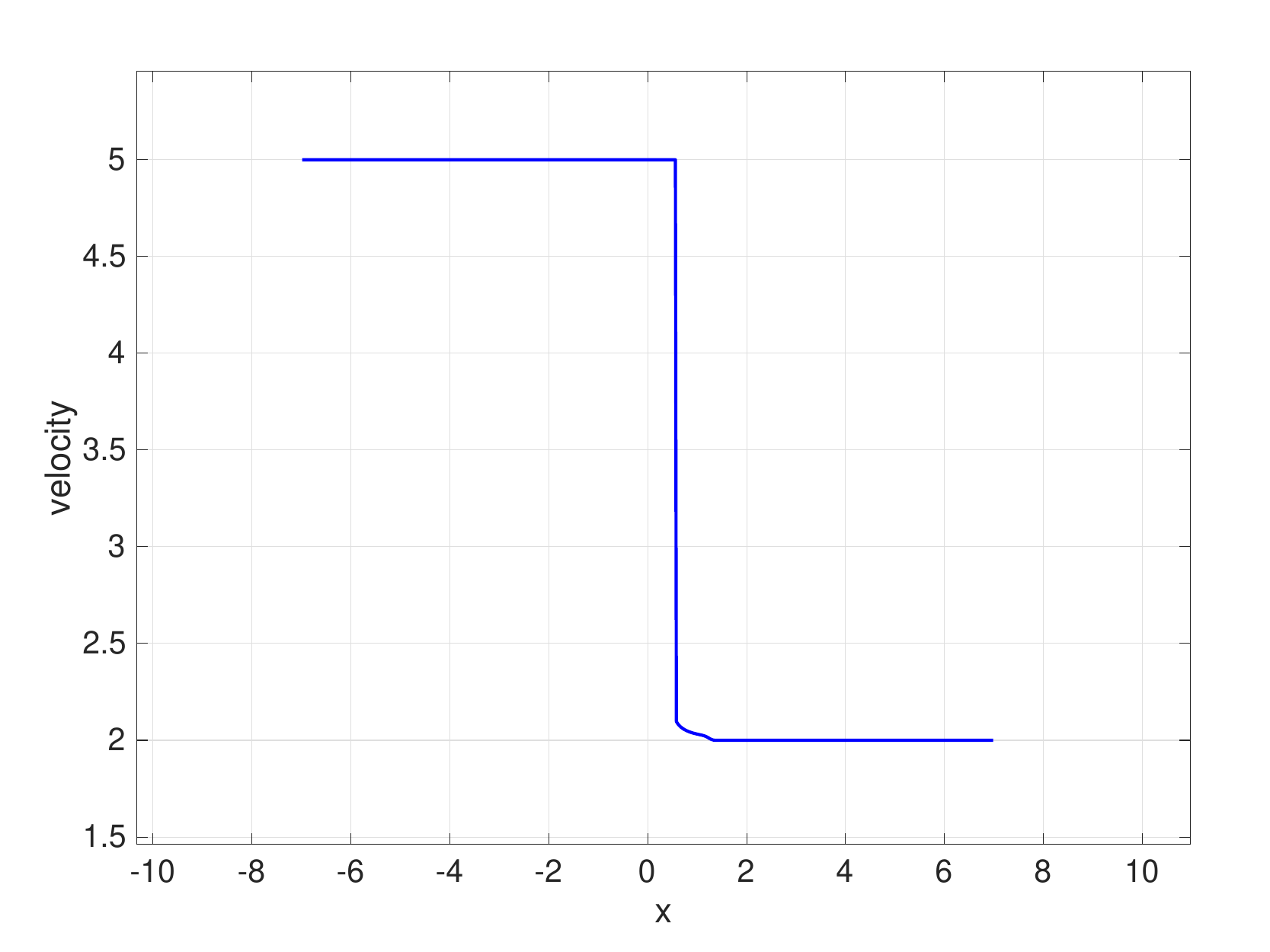}
		%	\caption*{(b)} % Optional label for the subfigure
	\end{minipage}\hfill
	\begin{minipage}{0.24\textwidth}
		\centering
		\includegraphics[width=\linewidth]{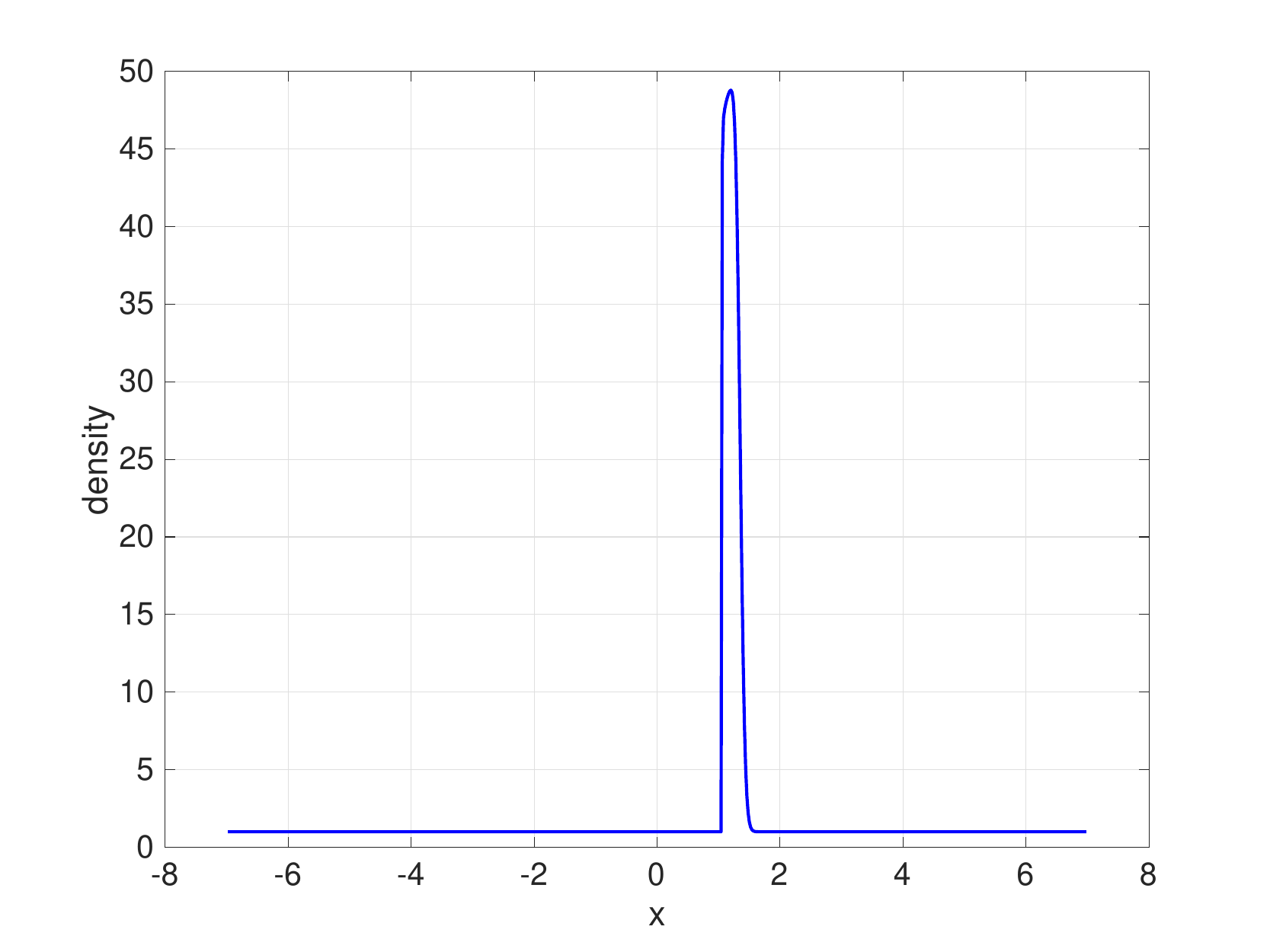}
		%	\caption*{(c)} % Optional label for the subfigure
	\end{minipage}\hfill
	\begin{minipage}{0.24\textwidth}
		\centering
		\includegraphics[width=\linewidth]{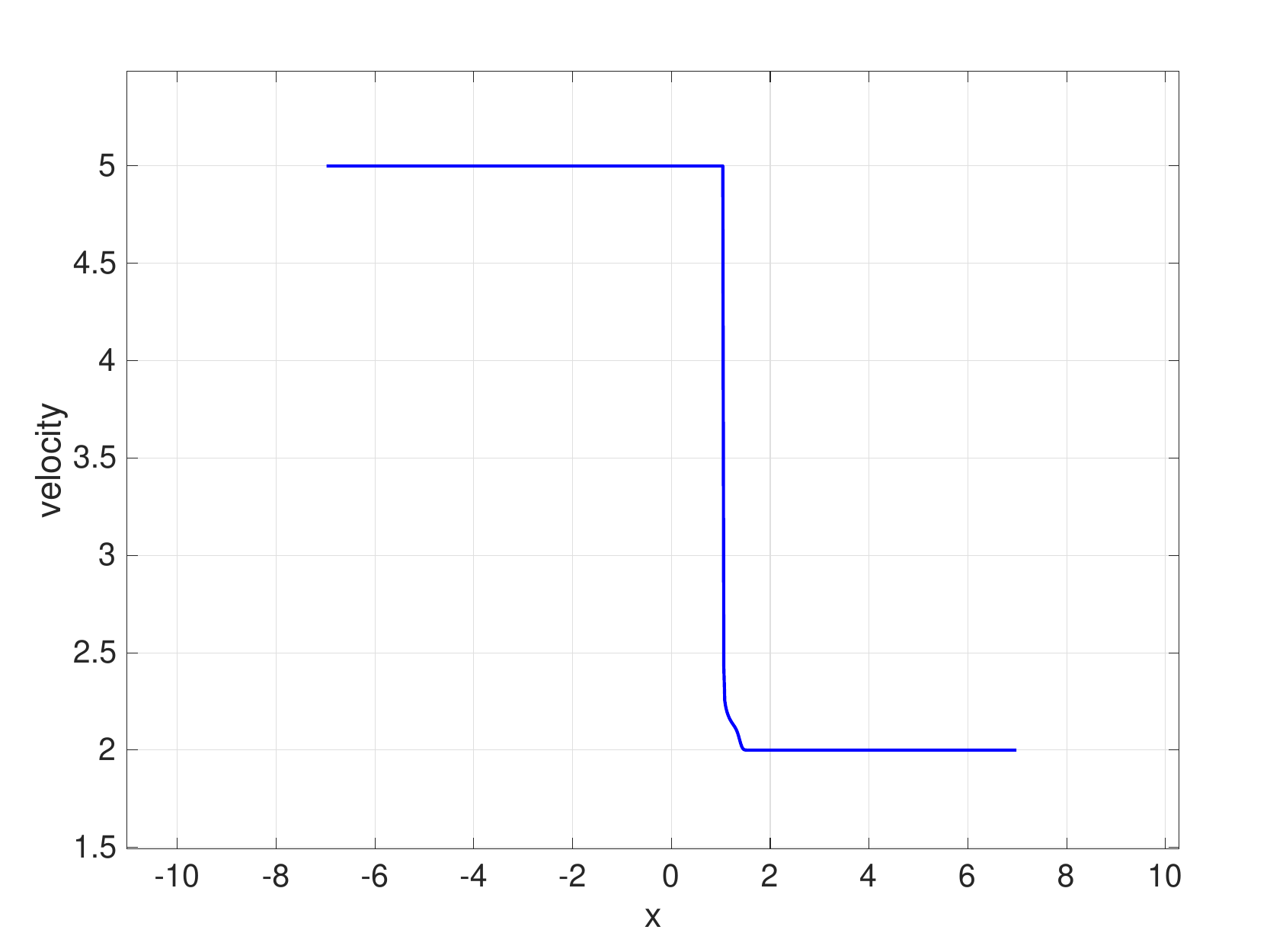}
		%	\caption*{(d)} % Optional label for the subfigure
	\end{minipage}
	\caption{Density and velocity for A=0.0001, a=0.00001 and A=0.00001, a=0.000001, respectively.}
	\label{p114}
\end{figure}
%Again, we take the same initial data \eqref{p107} but change the values of $\kappa$ and $\Gamma.$ Now, we consider $\kappa=0.75$ and $\Gamma=3.$ Figures \ref{p113}-\ref{p114} represent the behaviors of density and velocity as $a$ and $A$ decrease. 

Thus, the numerical computations show that the intermediate density increases dramatically whenever $a$ and $A$ decrease. Hence, the numerical simulations are consistent with our theoretical analysis. 

 \noindent\textbf{Case (ii).} Corresponding to the case  
$\upsilon_{l}-\frac{B}{\varrho_{l}^{\kappa}} < \upsilon_{r}< \upsilon_{l},$ we assume the following initial data
\begin{equation}\label{p108}
	(\varrho_{l, r}, \upsilon_{l, r})= \begin{cases}
		(2, 5), \qquad x<0,
		\\ (1, 4.5), \qquad x>0,
	\end{cases}
\end{equation} with $B=1$ and observe the numerical results first for $\kappa=0.5,$ and $\Gamma=2$ (see, Figures \ref{p115}-\ref{p116}) and secondly for $\kappa=0.25,$ and $\Gamma=1$ (see, Figures \ref{p117}-\ref{p118}). 
\begin{figure}[h!]
	\begin{minipage}{0.24\textwidth}
		\centering
		\includegraphics[width=\linewidth]{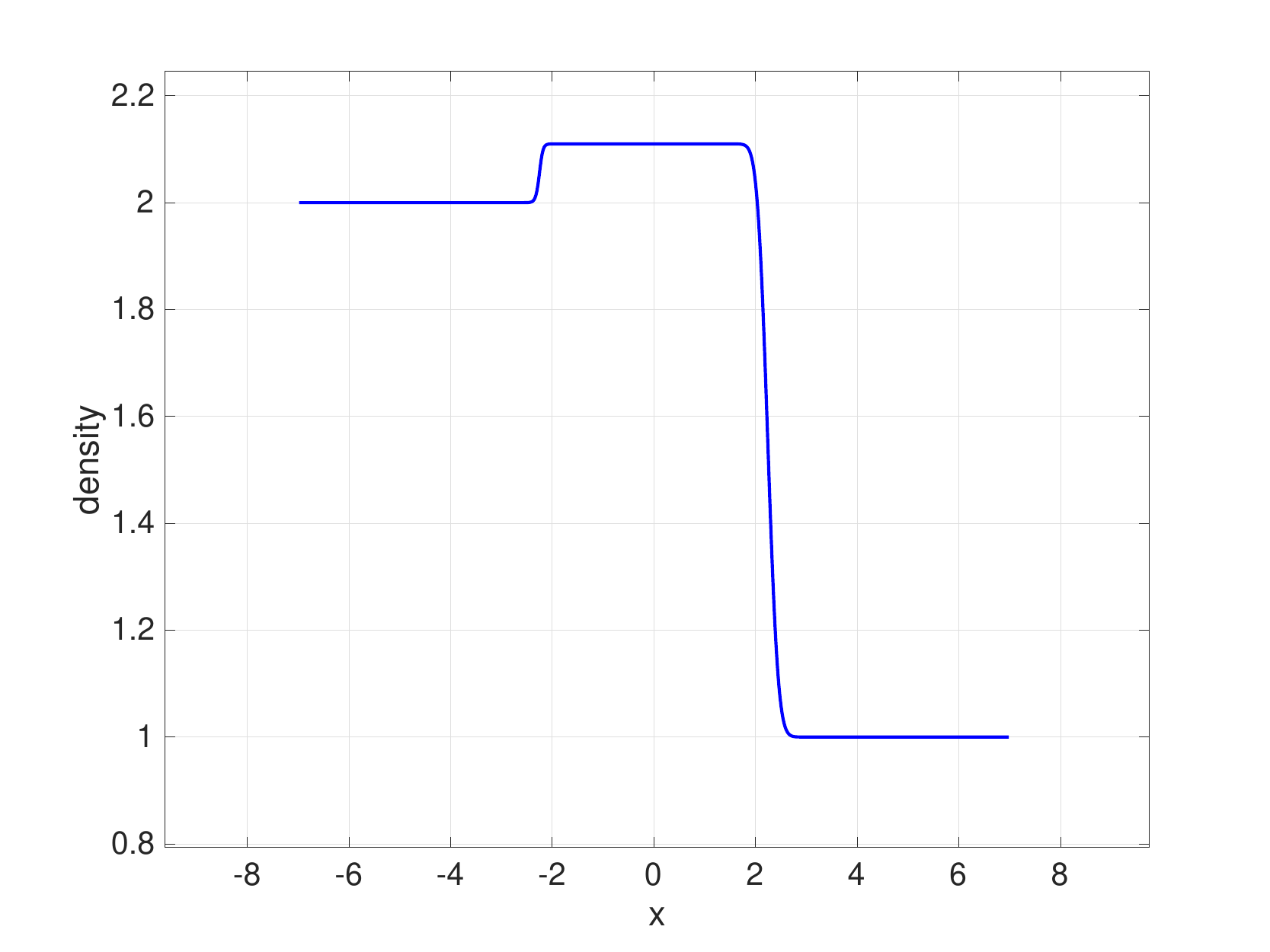}
		%	\caption*{(a)} % Optional label for the subfigure
	\end{minipage}\hfill
	\begin{minipage}{0.24\textwidth}
		\centering
		\includegraphics[width=\linewidth]{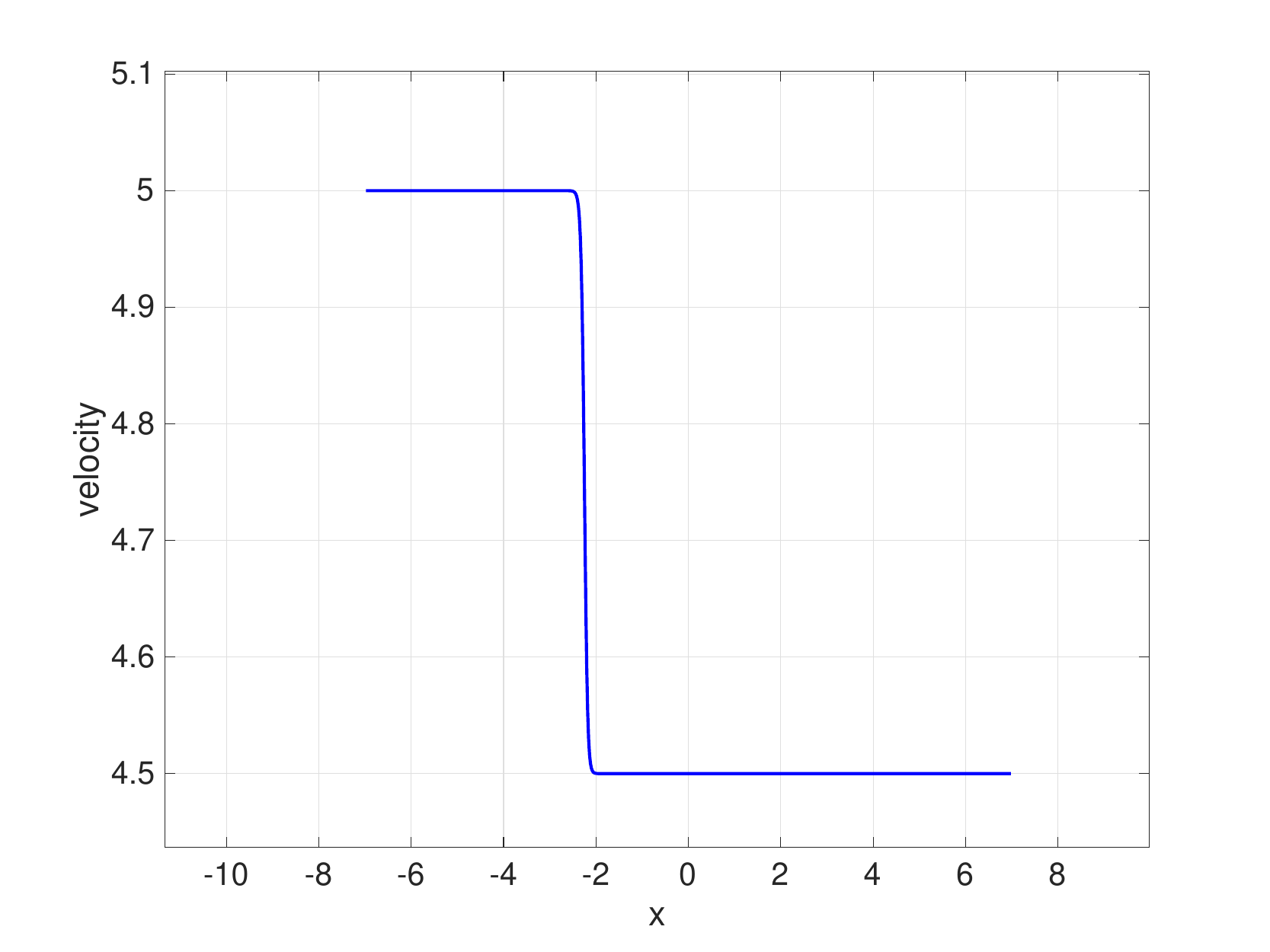}
		%	\caption*{(b)} % Optional label for the subfigure
	\end{minipage}\hfill
	\begin{minipage}{0.24\textwidth}
		\centering
		\includegraphics[width=\linewidth]{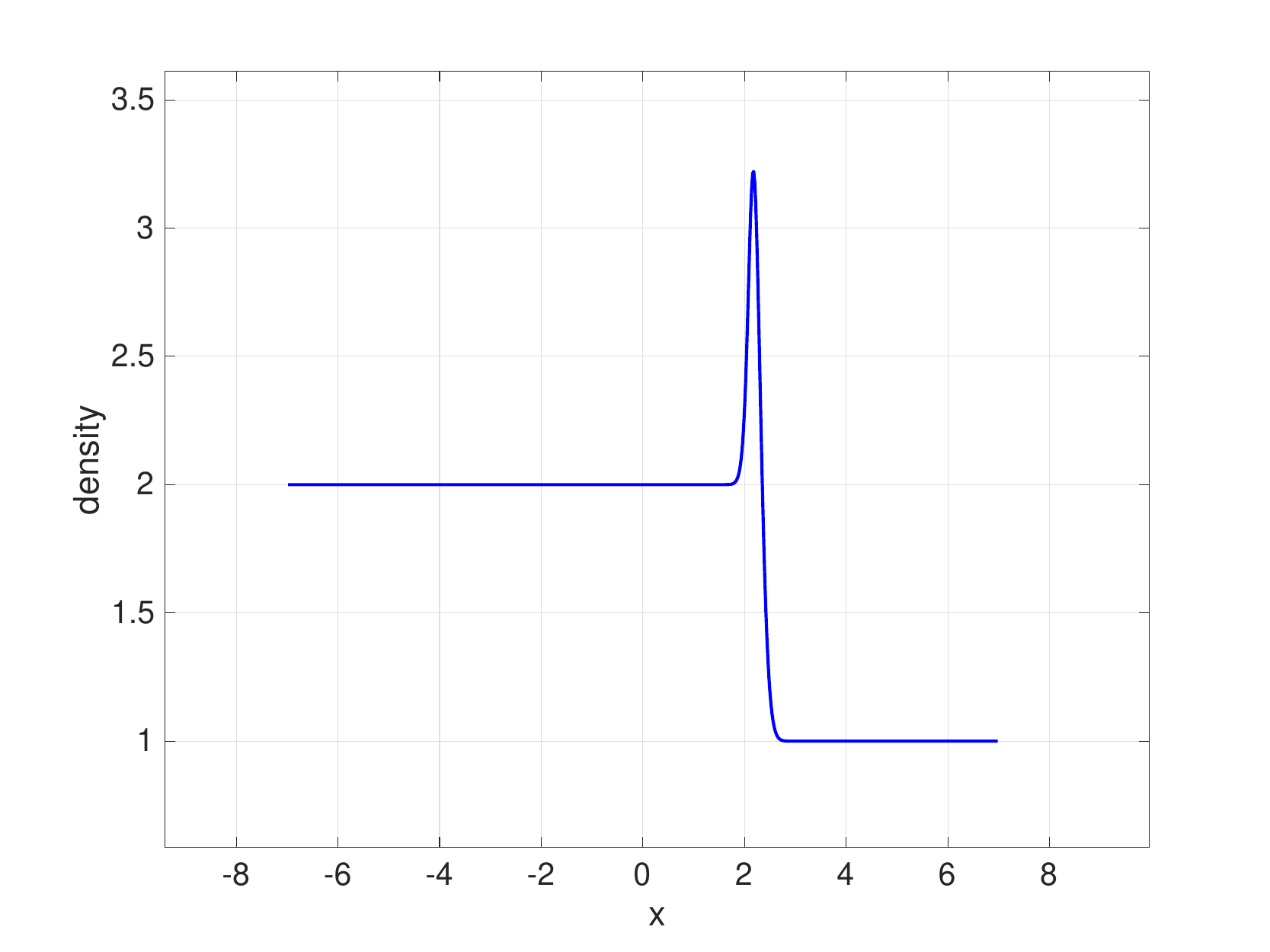}
		%	\caption*{(c)} % Optional label for the subfigure
	\end{minipage}\hfill
	\begin{minipage}{0.24\textwidth}
		\centering
		\includegraphics[width=\linewidth]{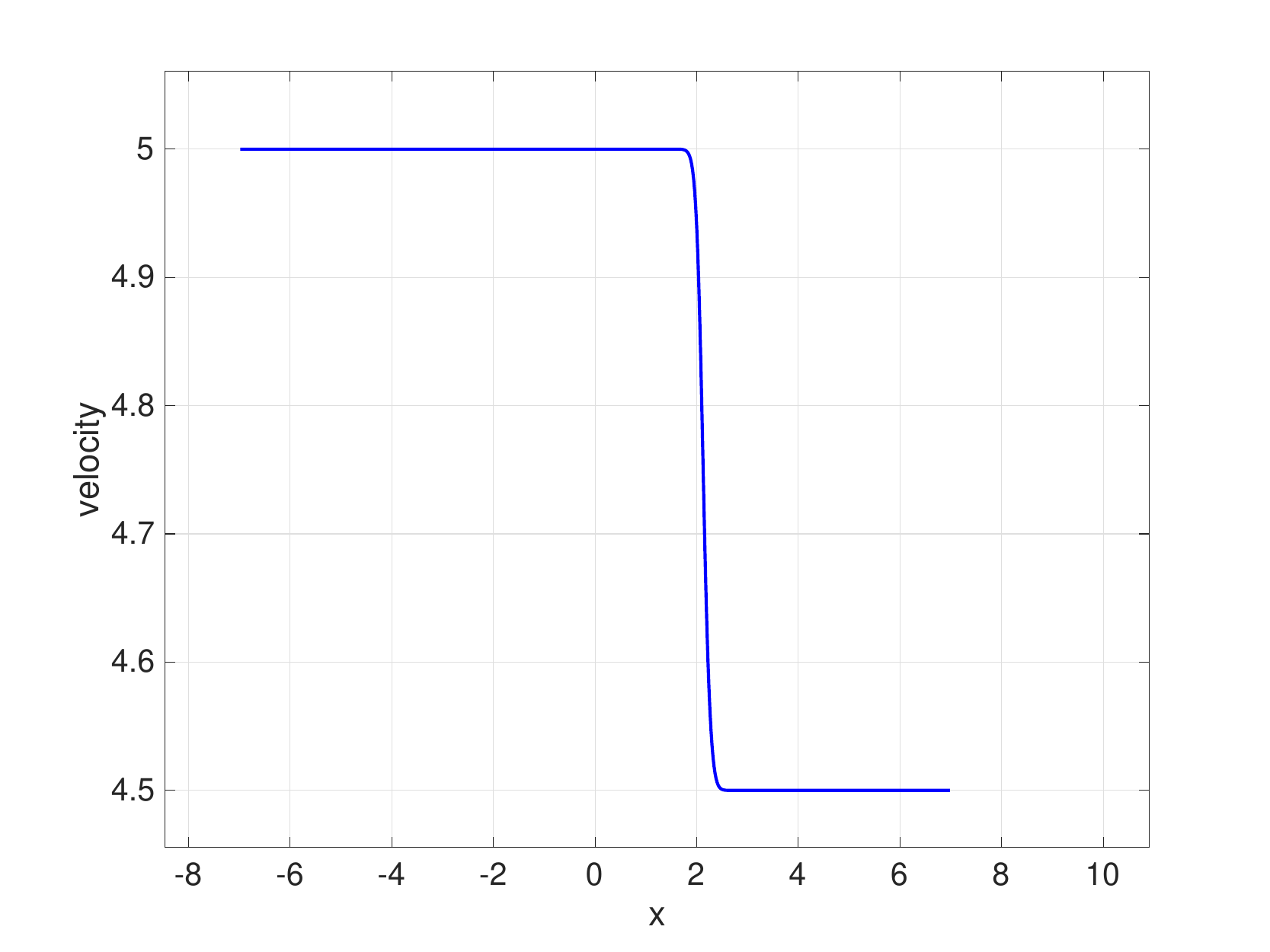}
		%	\caption*{(d)} % Optional label for the subfigure
	\end{minipage}
	\caption{Density and velocity for A=1, a=0.01 and A=0.01, a=0.001, respectively.}
	\label{p115}
\end{figure}

\begin{figure}[h!]
	\begin{minipage}{0.24\textwidth}
		\centering
		\includegraphics[width=\linewidth]{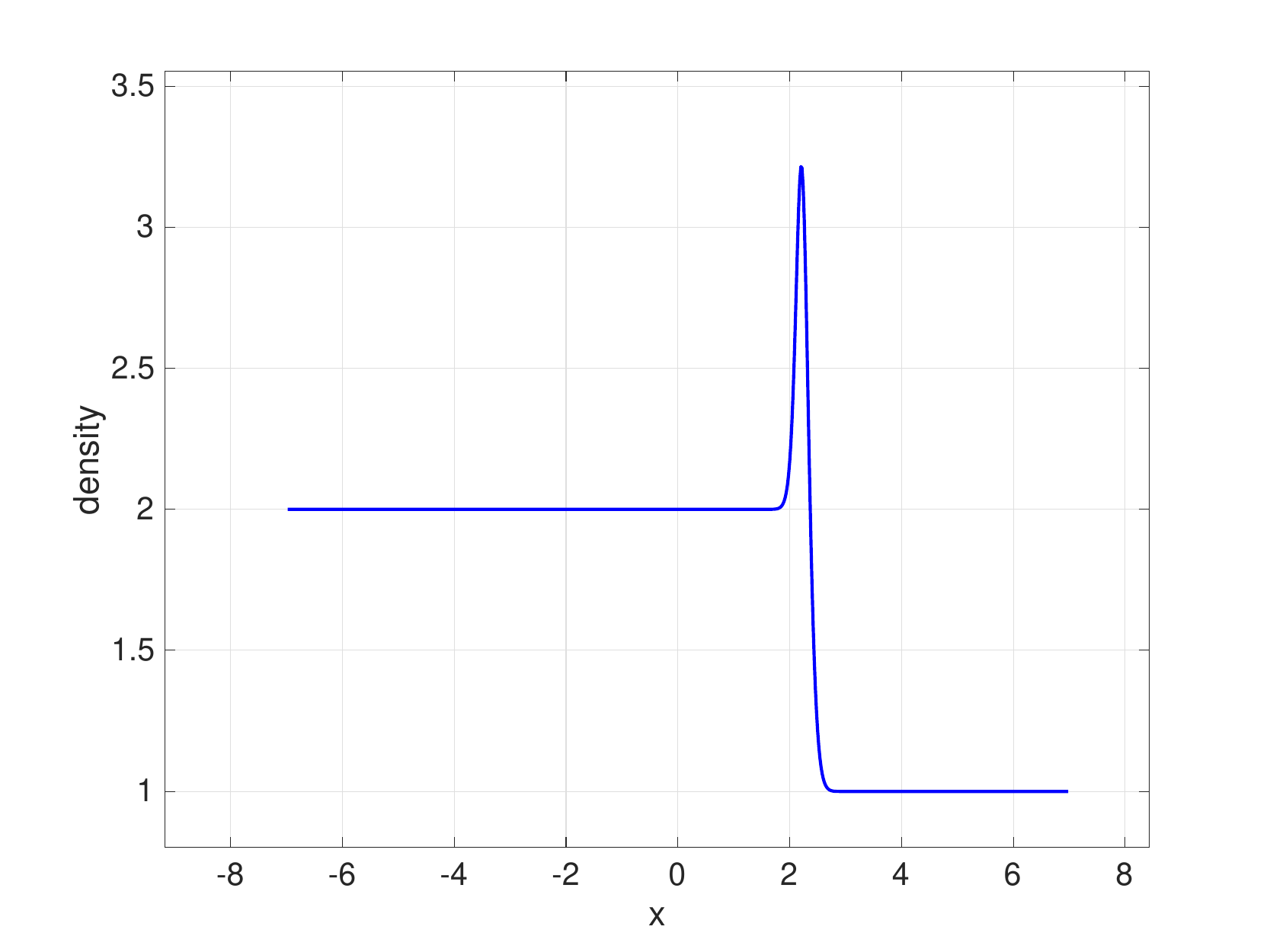}
		%	\caption*{(a)} % Optional label for the subfigure
	\end{minipage}\hfill
	\begin{minipage}{0.24\textwidth}
		\centering
		\includegraphics[width=\linewidth]{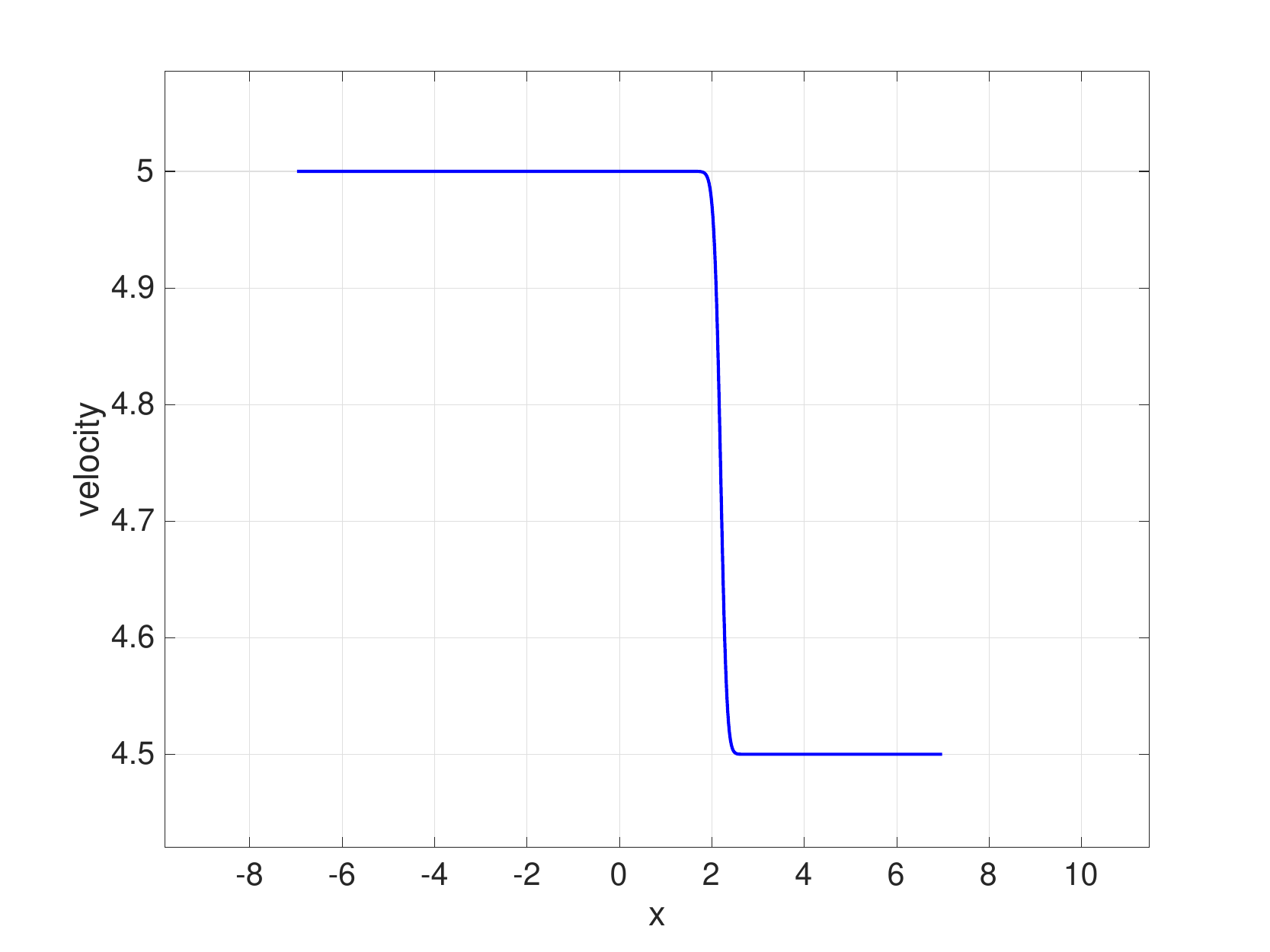}
		%	\caption*{(b)} % Optional label for the subfigure
	\end{minipage}\hfill
	\begin{minipage}{0.24\textwidth}
		\centering
		\includegraphics[width=\linewidth]{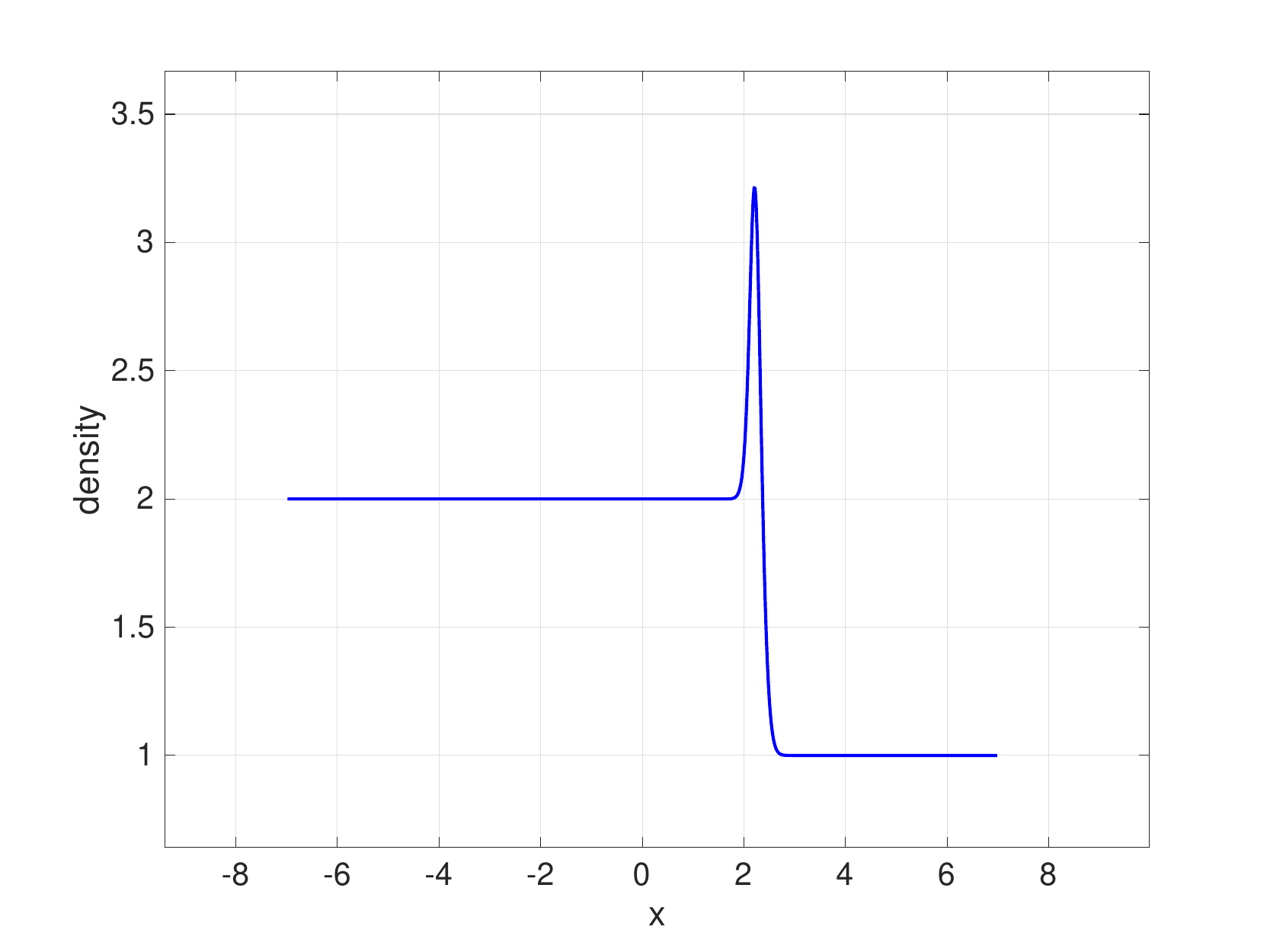}
		%	\caption*{(c)} % Optional label for the subfigure
	\end{minipage}\hfill
	\begin{minipage}{0.24\textwidth}
		\centering
		\includegraphics[width=\linewidth]{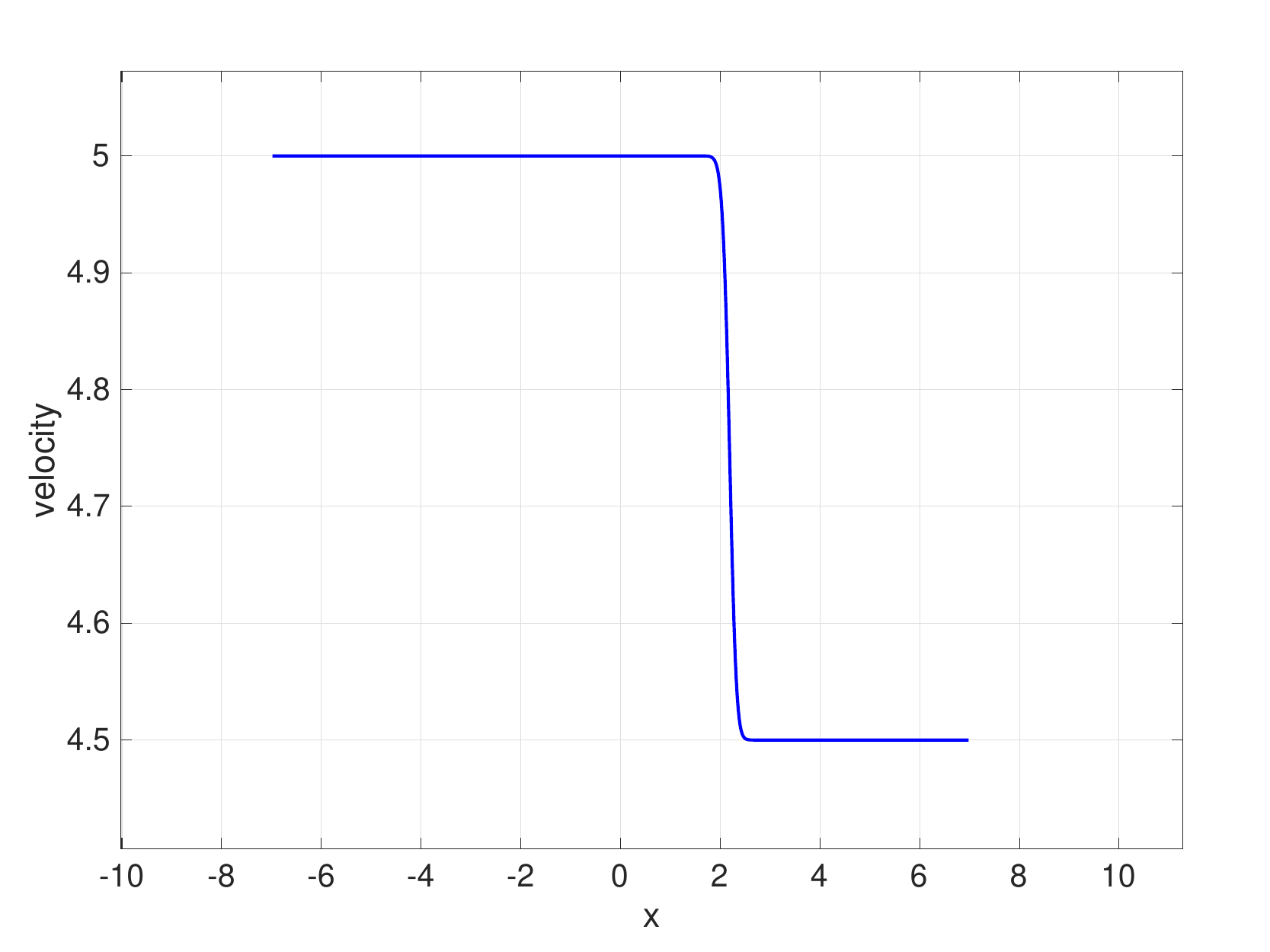}
		%	\caption*{(d)} % Optional label for the subfigure
	\end{minipage}
	\caption{Density and velocity for A=0.0001, a=0.00001 and A=0.00001, a=0.000001, respectively.}
	\label{p116}
\end{figure}
\begin{figure}[h!]
	\begin{minipage}{0.24\textwidth}
		\centering
		\includegraphics[width=\linewidth]{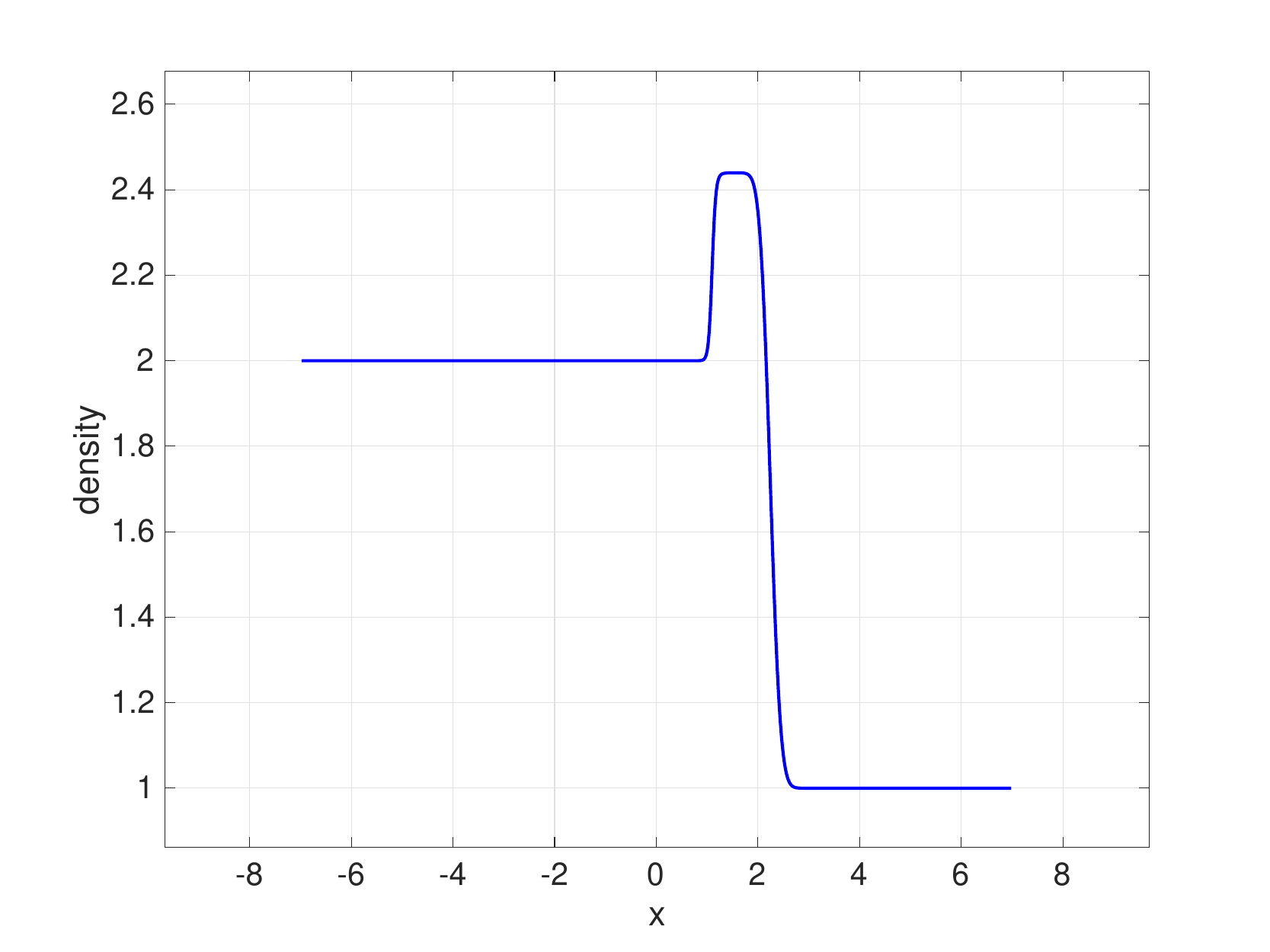}
		%	\caption*{(a)} % Optional label for the subfigure
	\end{minipage}\hfill
	\begin{minipage}{0.24\textwidth}
		\centering
		\includegraphics[width=\linewidth]{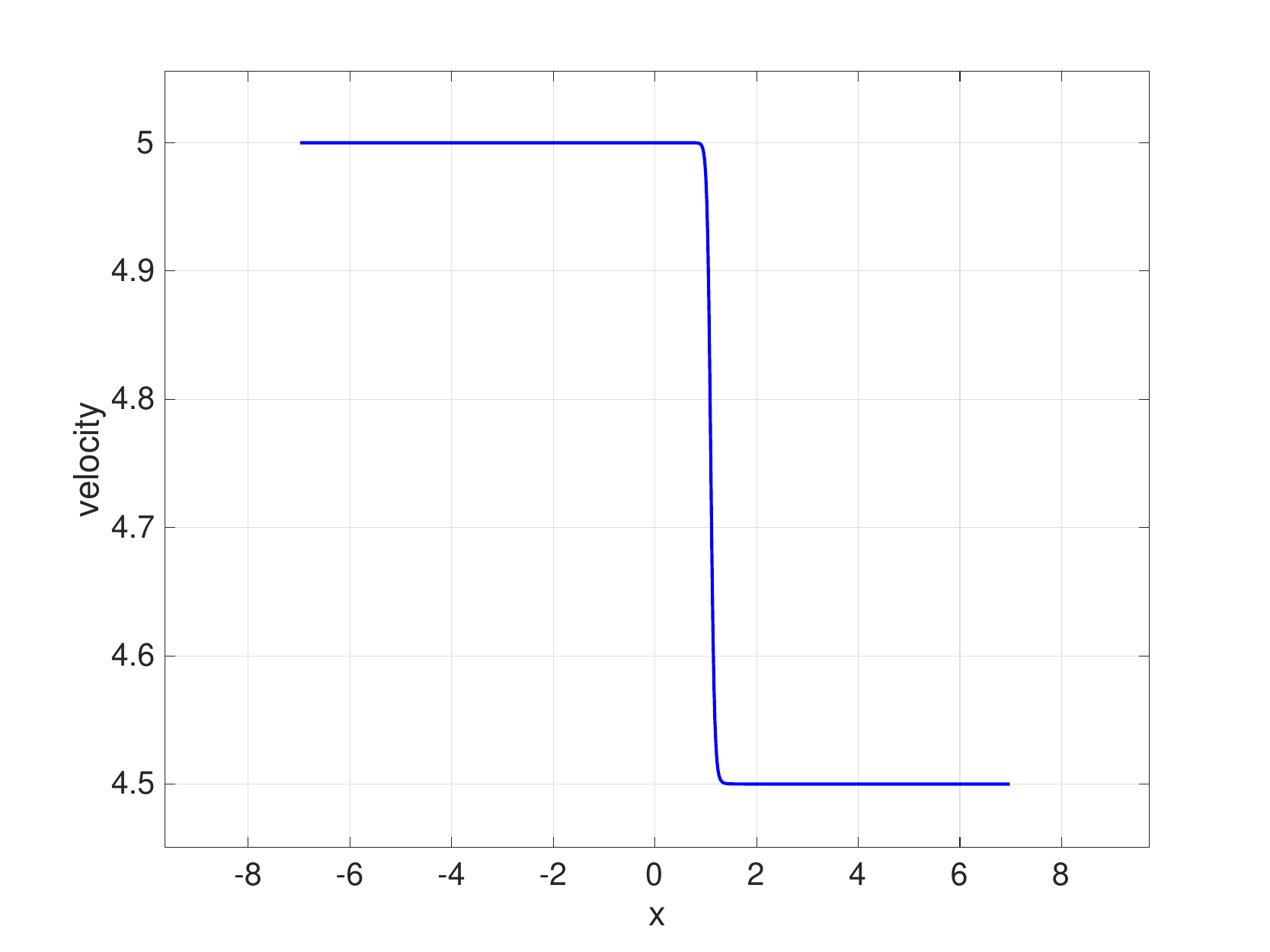}
		%	\caption*{(b)} % Optional label for the subfigure
	\end{minipage}\hfill
	\begin{minipage}{0.24\textwidth}
		\centering
		\includegraphics[width=\linewidth]{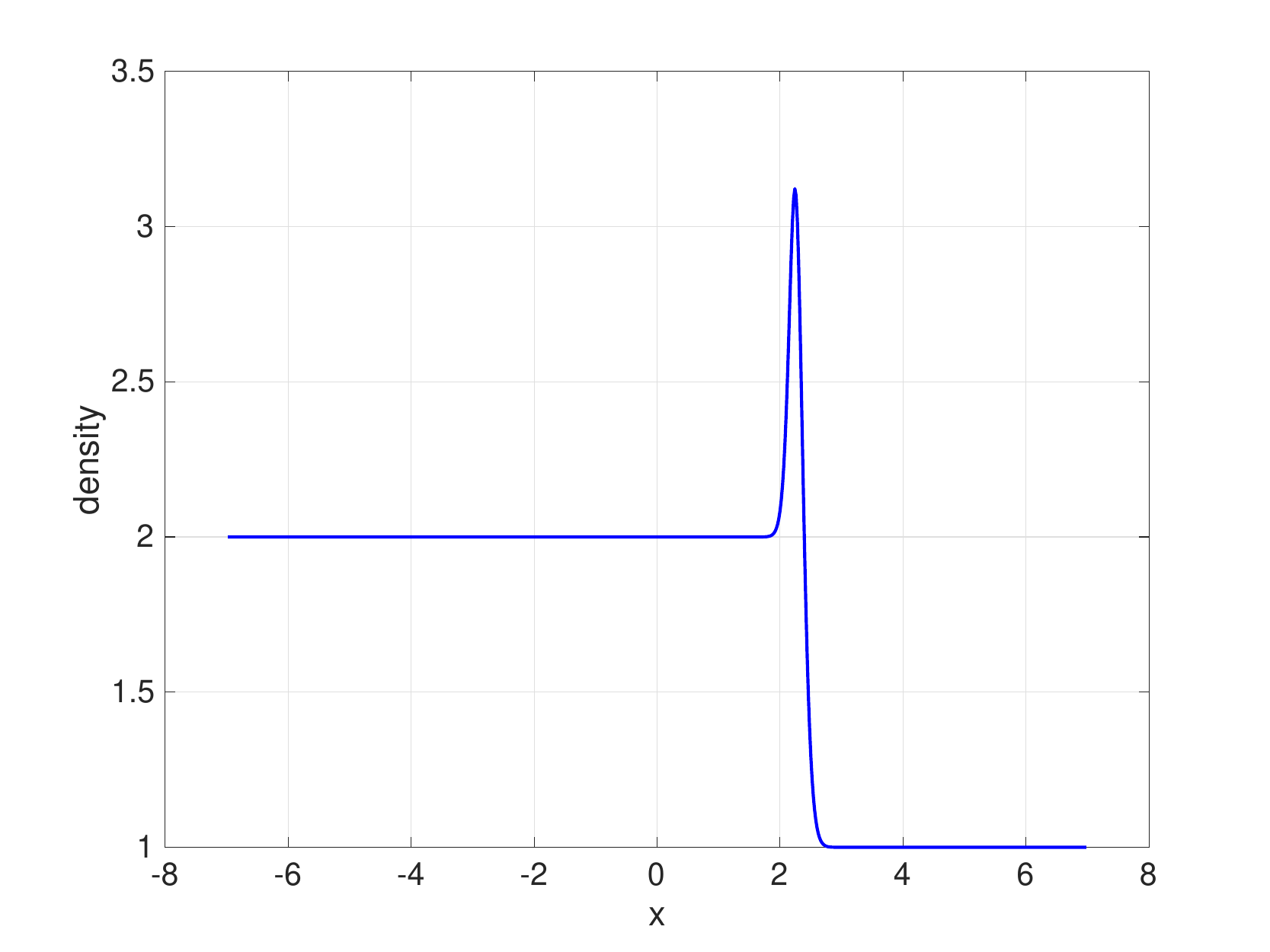}
		%	\caption*{(c)} % Optional label for the subfigure
	\end{minipage}\hfill
	\begin{minipage}{0.24\textwidth}
		\centering
		\includegraphics[width=\linewidth]{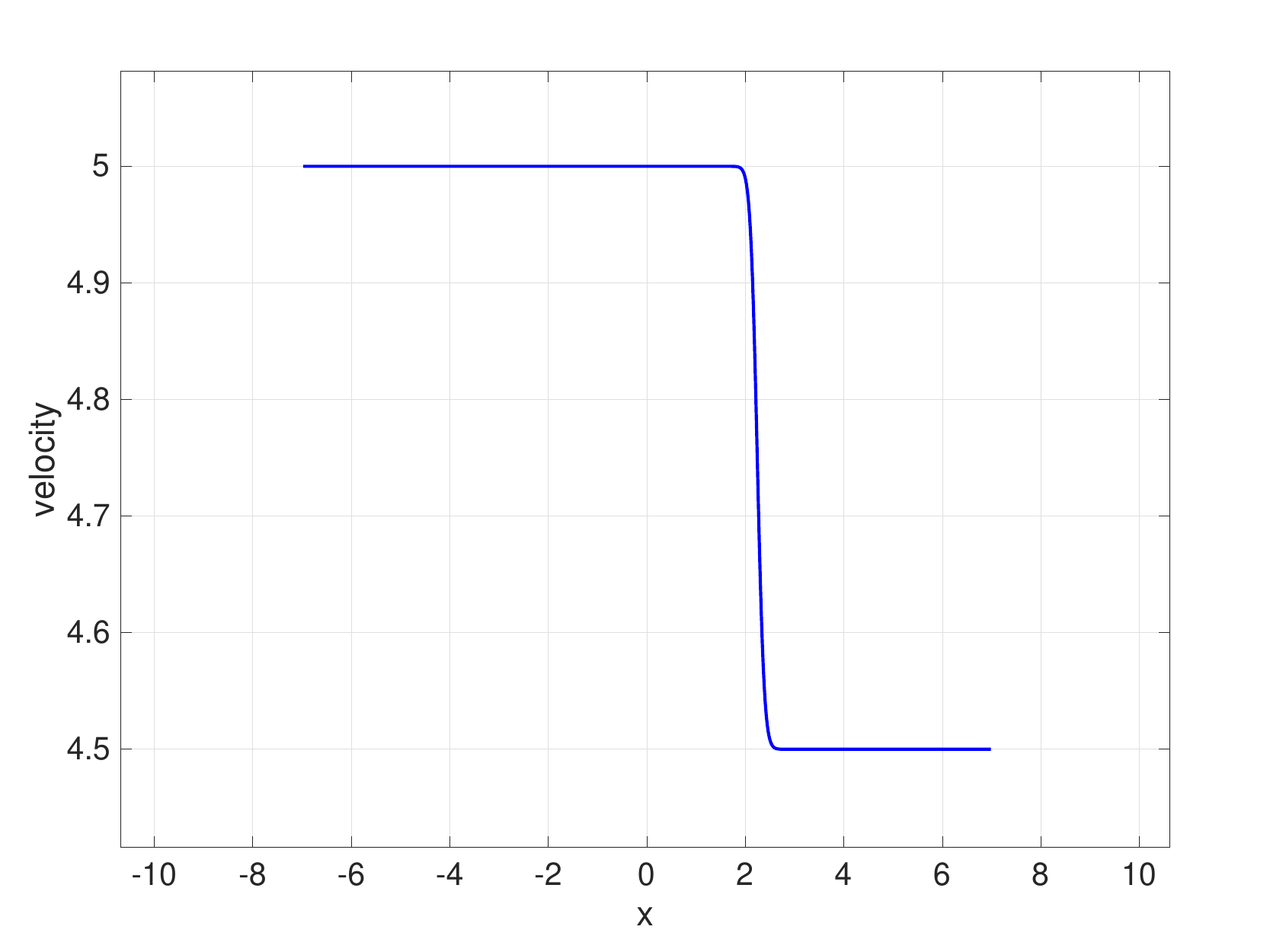}
		%	\caption*{(d)} % Optional label for the subfigure
	\end{minipage}
	\caption{Density and velocity for A=1, a=0.01 and A=0.01, a=0.001, respectively.}
	\label{p117}
\end{figure}

\begin{figure}[h!]
	\begin{minipage}{0.24\textwidth}
		\centering
		\includegraphics[width=\linewidth]{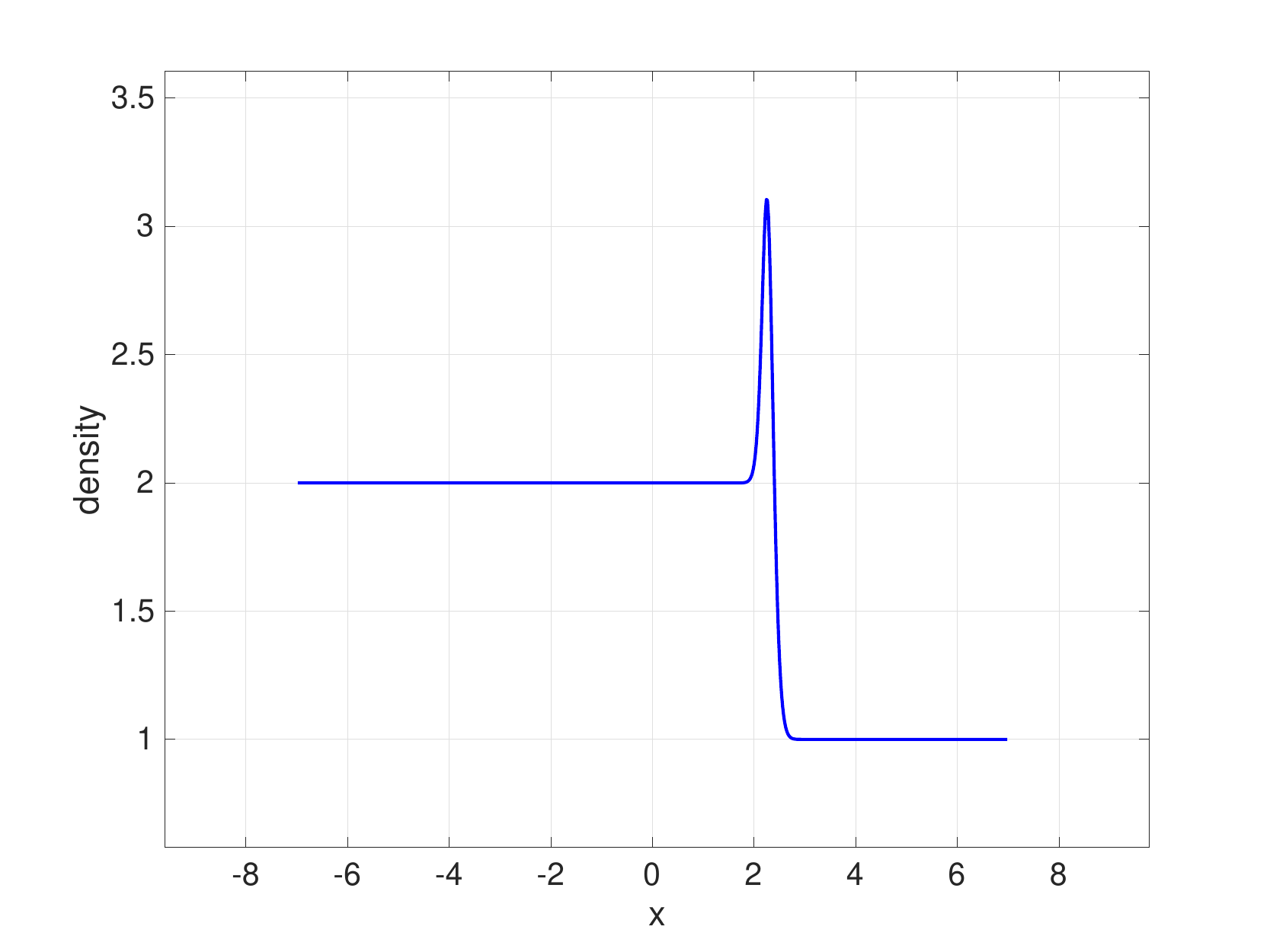}
		%	\caption*{(a)} % Optional label for the subfigure
	\end{minipage}\hfill
	\begin{minipage}{0.24\textwidth}
		\centering
		\includegraphics[width=\linewidth]{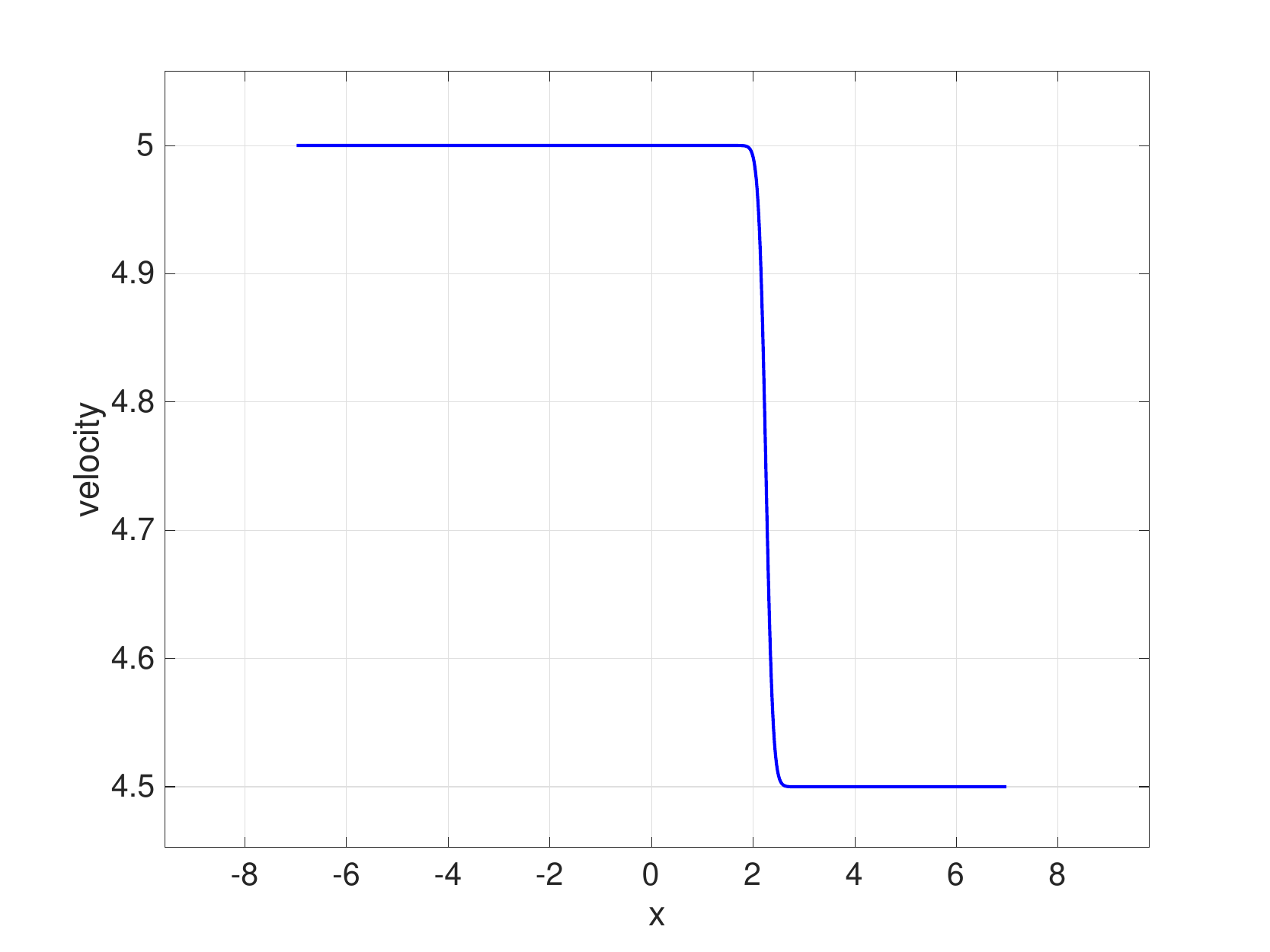}
		%	\caption*{(b)} % Optional label for the subfigure
	\end{minipage}\hfill
	\begin{minipage}{0.24\textwidth}
		\centering
		\includegraphics[width=\linewidth]{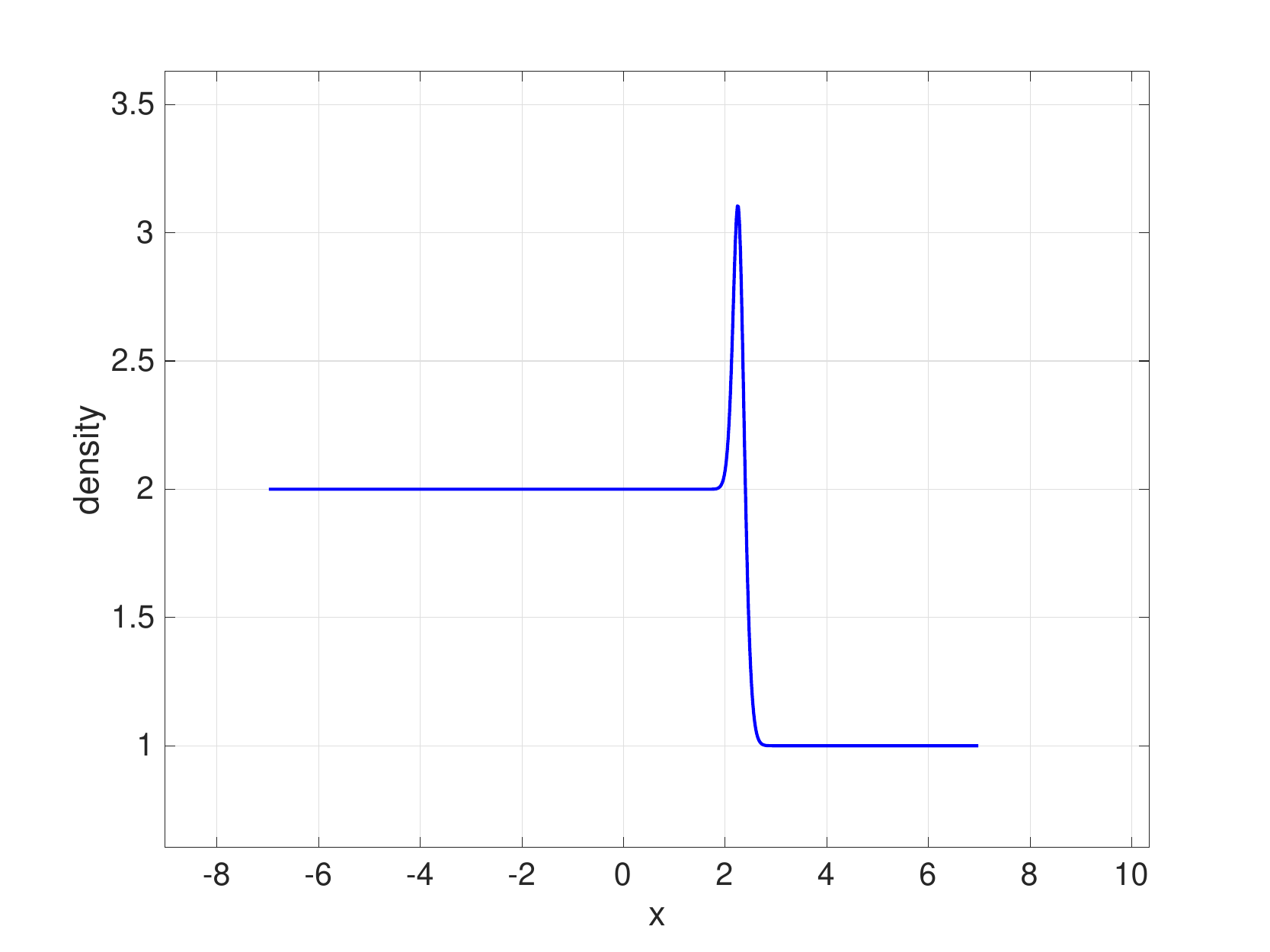}
		%	\caption*{(c)} % Optional label for the subfigure
	\end{minipage}\hfill
	\begin{minipage}{0.24\textwidth}
		\centering
		\includegraphics[width=\linewidth]{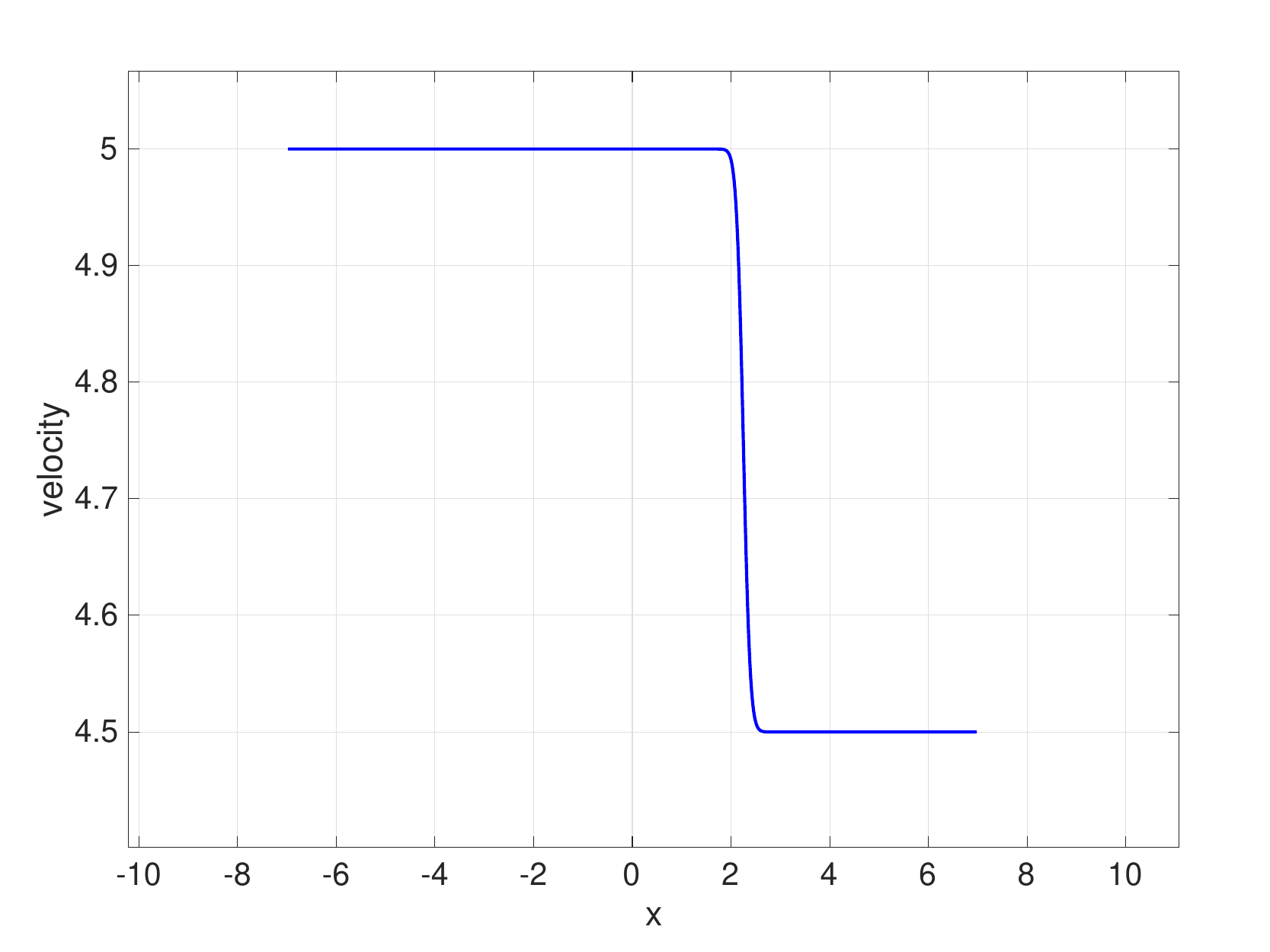}
		%	\caption*{(d)} % Optional label for the subfigure
	\end{minipage}
	\caption{Density and velocity for A=0.0001, a=0.00001 and A=0.00001, a=0.000001, respectively.}
	\label{p118}
\end{figure}
\noindent \textbf{Case (iii).} For the case $ \upsilon_{r}> \upsilon_{l},$ we take the following initial data
\begin{equation}\label{p109}
	(\varrho_{l, r}, \upsilon_{l, r})= \begin{cases}
		(1, 5), \qquad x<0,
		\\ (2, 7), \qquad x>0,
	\end{cases}
\end{equation}with $B=1.$ The density and velocity for $\kappa=0.5,$ and $\Gamma=2$ with different pairs of values of $A$ and  $a$ are shown in Figures \ref{p119}-\ref{p120}. Similarly, for $\kappa=0.25,$ and $\Gamma=1,$ density and velocity are shown in Figures \ref{p121}-\ref{p122}. \begin{figure}[h!]
\begin{minipage}{0.24\textwidth}
	\centering
	\includegraphics[width=\linewidth]{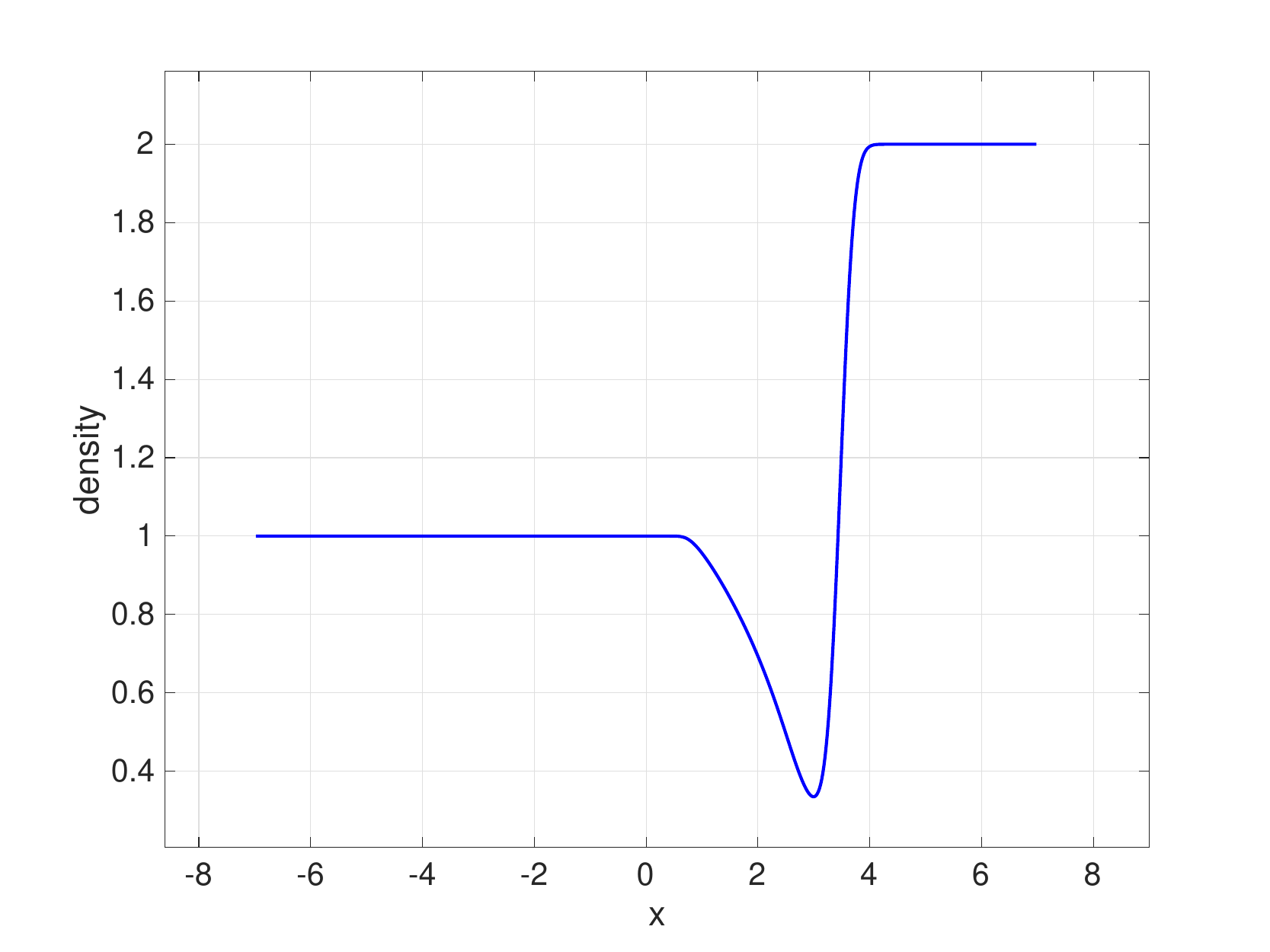}
	%	\caption*{(a)} % Optional label for the subfigure
\end{minipage}\hfill
\begin{minipage}{0.24\textwidth}
	\centering
	\includegraphics[width=\linewidth]{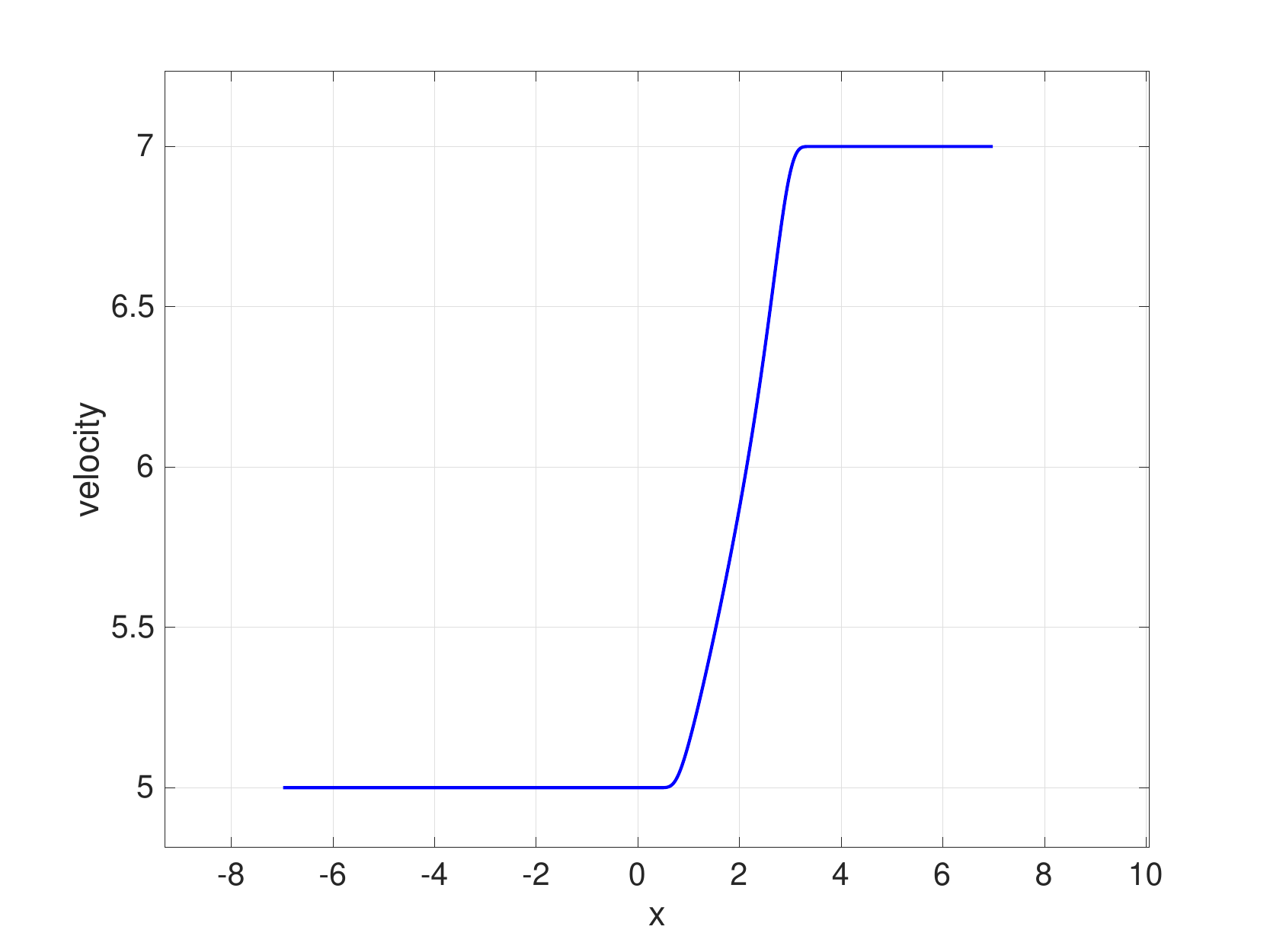}
	%	\caption*{(b)} % Optional label for the subfigure
\end{minipage}\hfill
\begin{minipage}{0.24\textwidth}
	\centering
	\includegraphics[width=\linewidth]{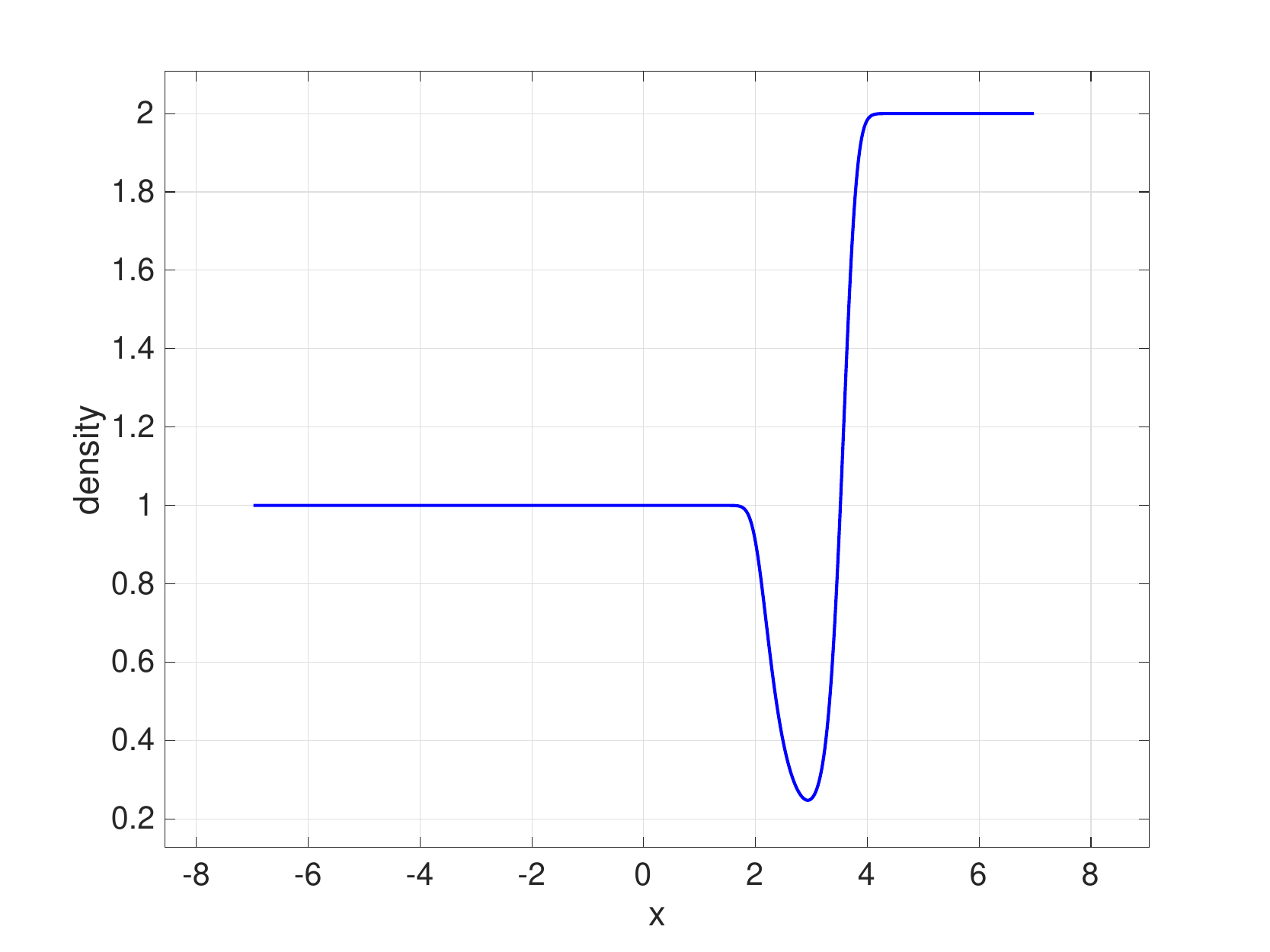}
	%	\caption*{(c)} % Optional label for the subfigure
\end{minipage}\hfill
\begin{minipage}{0.24\textwidth}
	\centering
	\includegraphics[width=\linewidth]{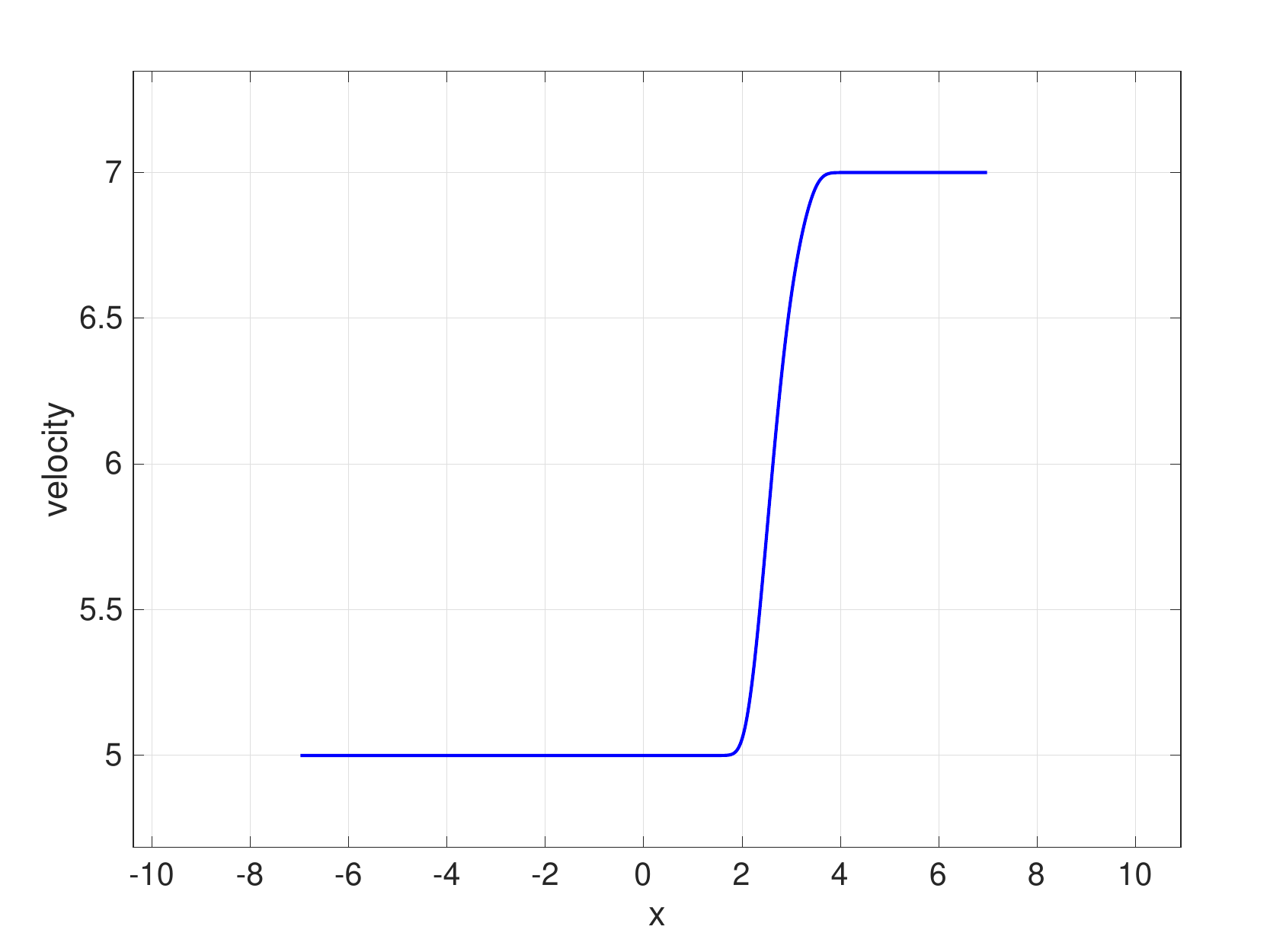}
	%	\caption*{(d)} % Optional label for the subfigure
\end{minipage}
\caption{Density and velocity for A=1, a=0.1 and A=0.01, a=0.001, respectively.}
\label{p119}
\end{figure}

\begin{figure}[h!]
\begin{minipage}{0.24\textwidth}
	\centering
	\includegraphics[width=\linewidth]{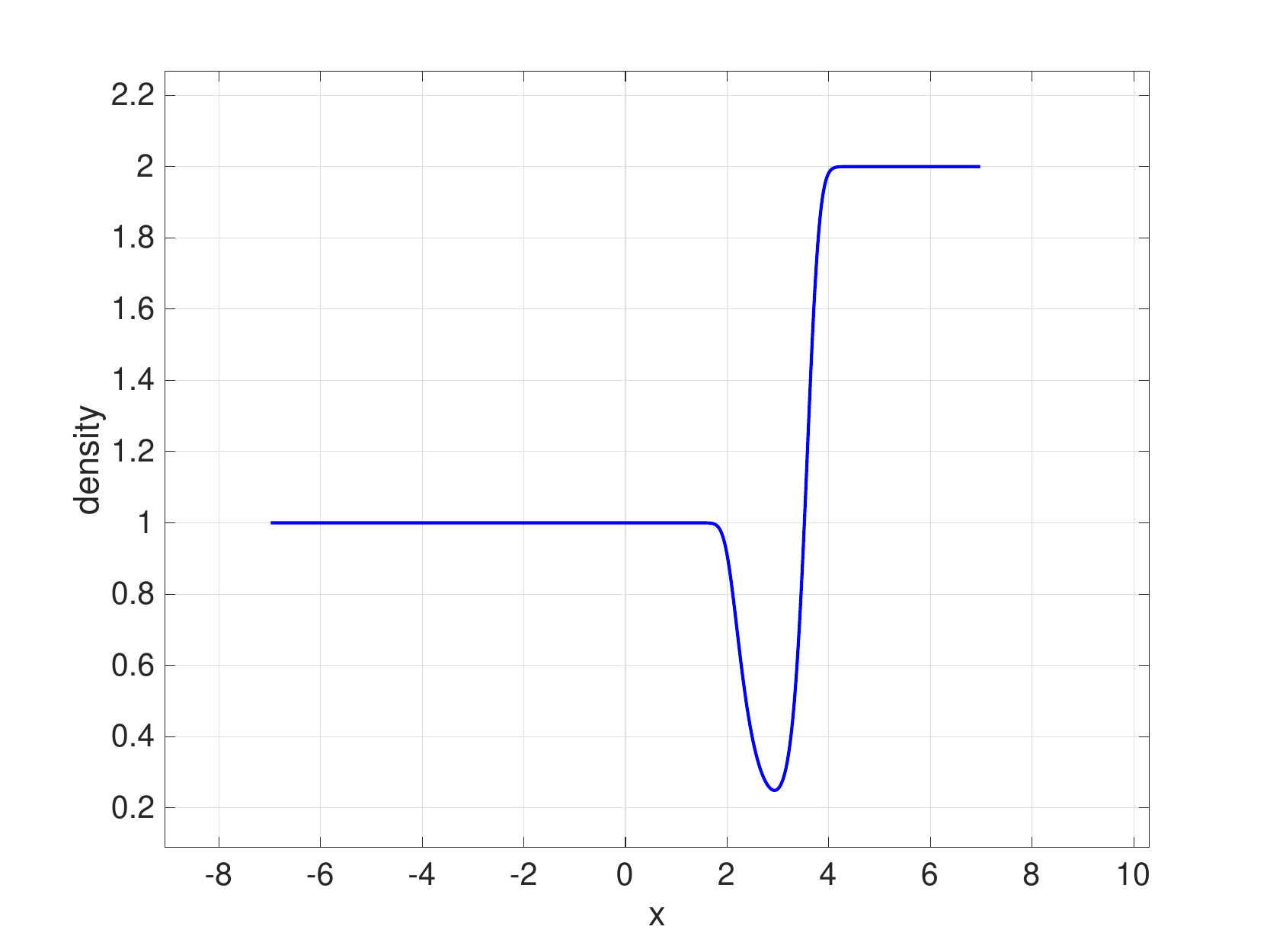}
	%	\caption*{(a)} % Optional label for the subfigure
\end{minipage}\hfill
\begin{minipage}{0.24\textwidth}
	\centering
	\includegraphics[width=\linewidth]{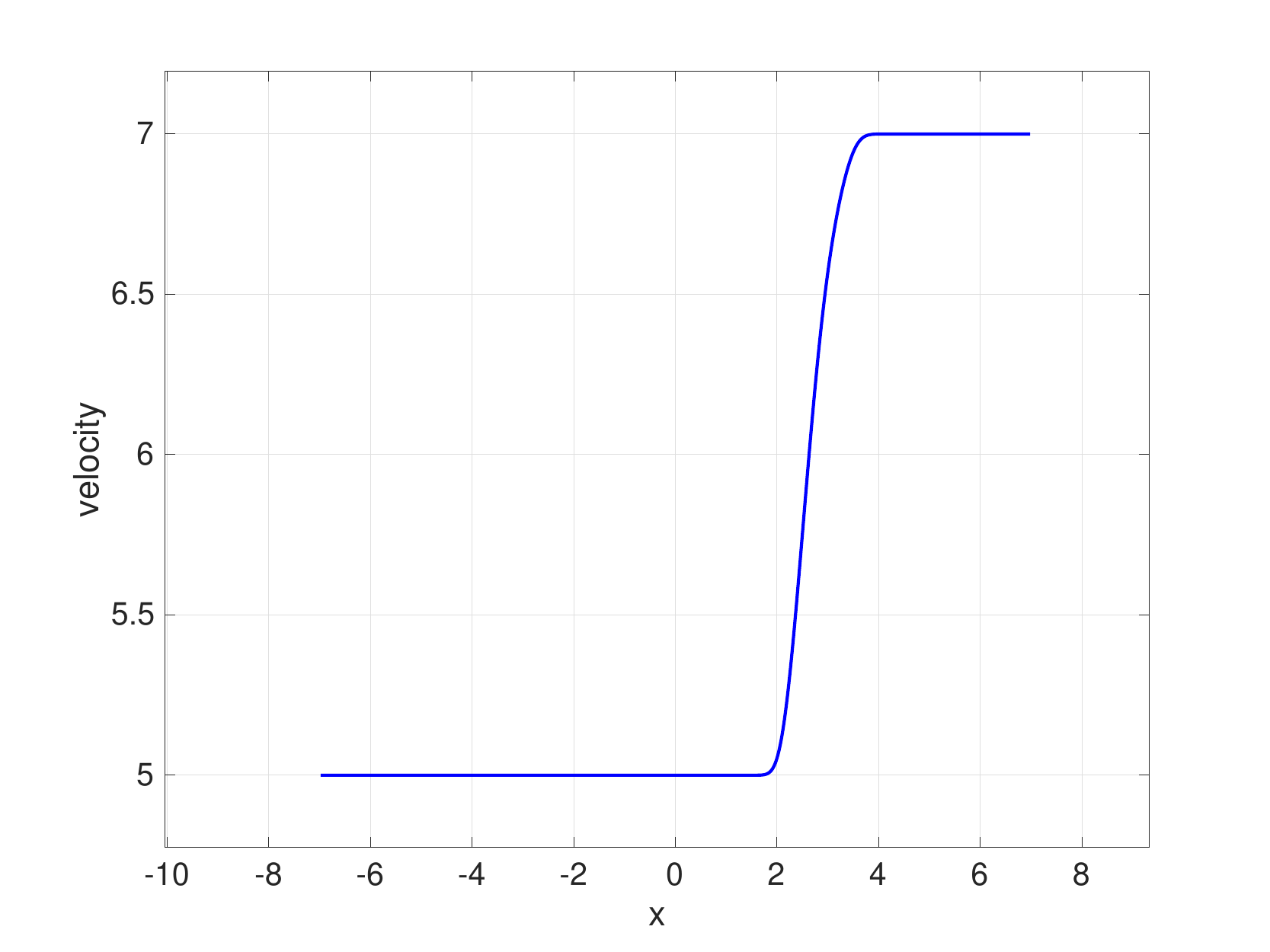}
	%	\caption*{(b)} % Optional label for the subfigure
\end{minipage}\hfill
\begin{minipage}{0.24\textwidth}
	\centering
	\includegraphics[width=\linewidth]{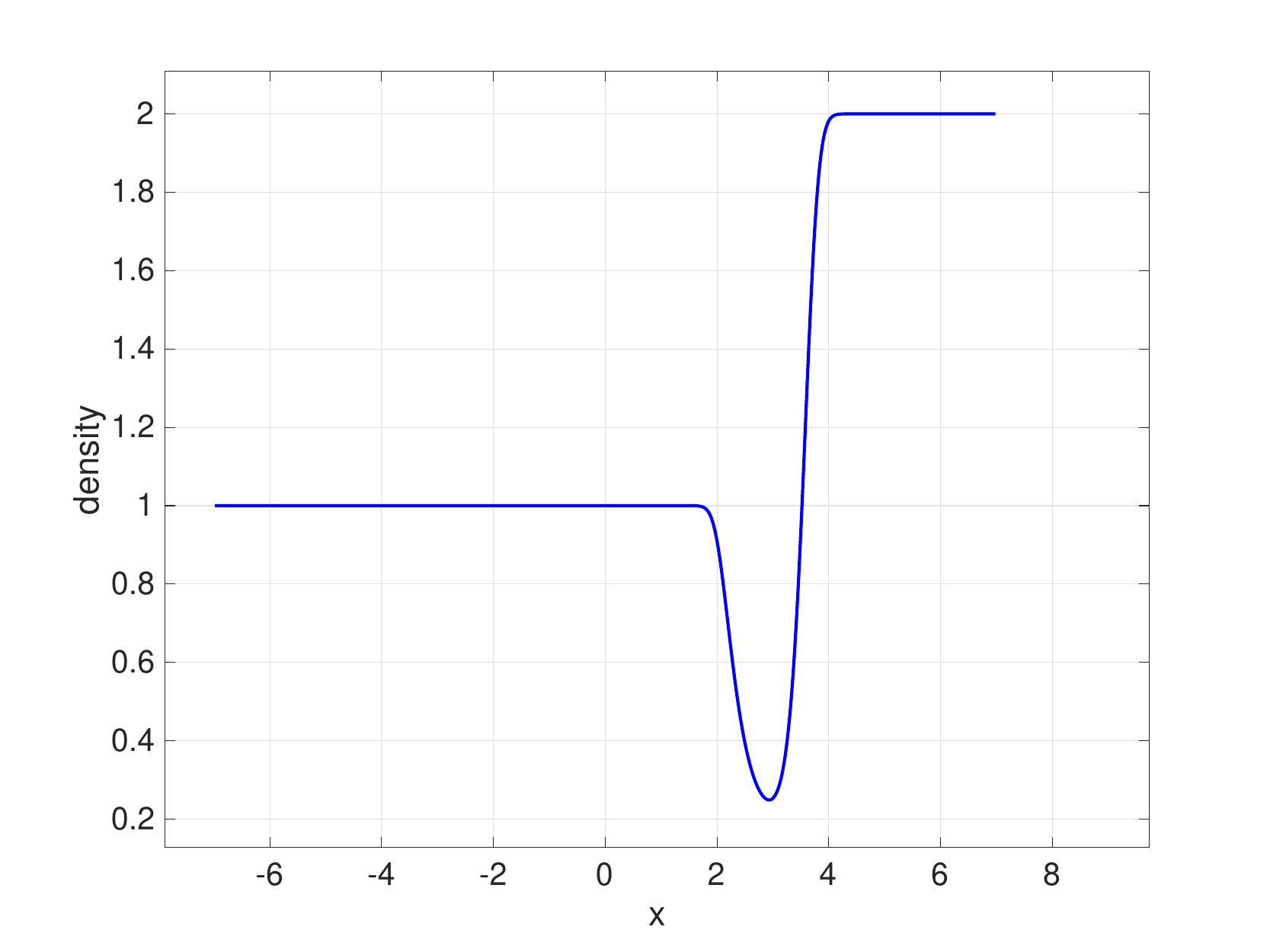}
	%	\caption*{(c)} % Optional label for the subfigure
\end{minipage}\hfill
\begin{minipage}{0.24\textwidth}
	\centering
	\includegraphics[width=\linewidth]{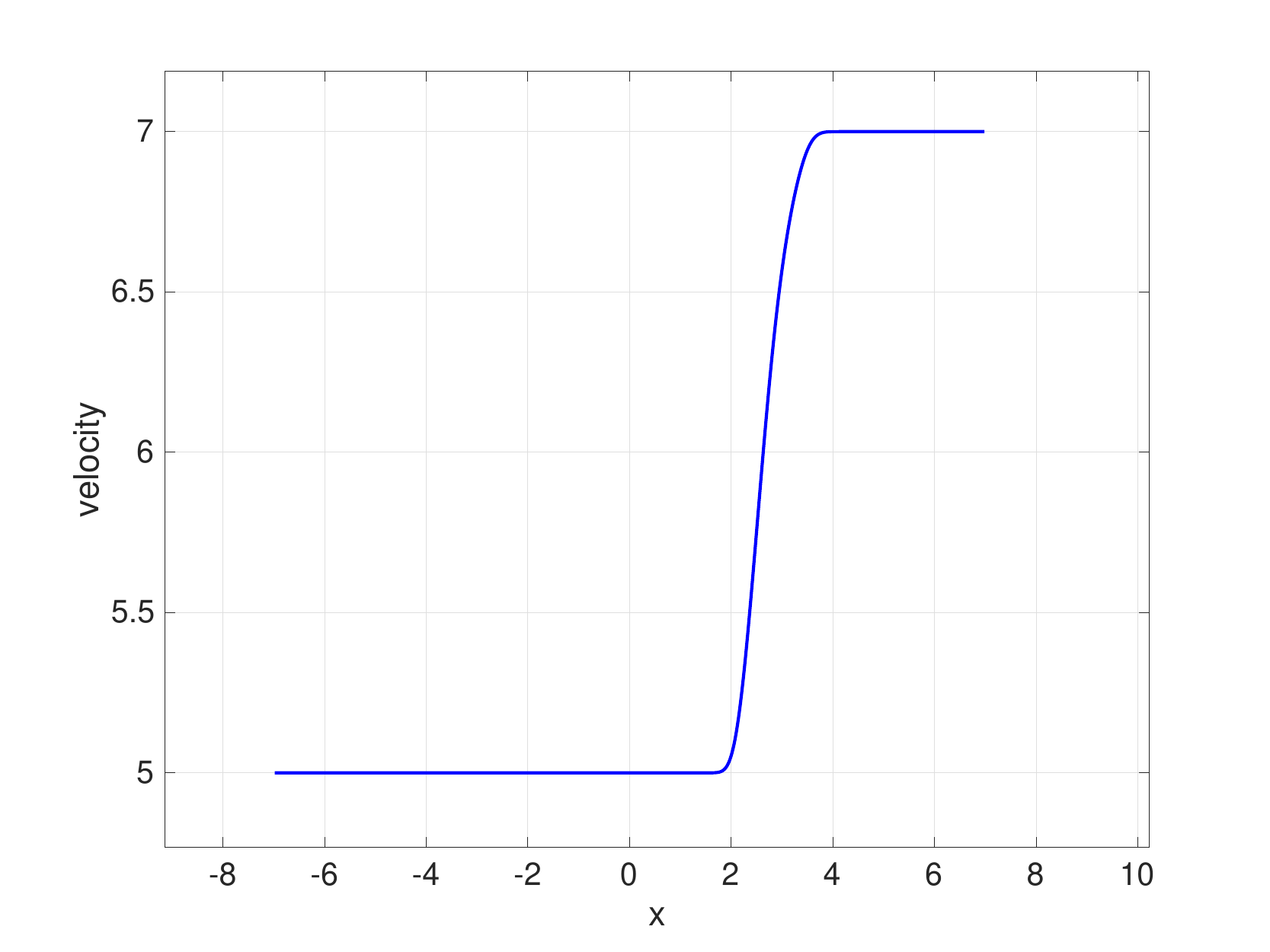}
	%	\caption*{(d)} % Optional label for the subfigure
\end{minipage}
\caption{Density and velocity for A=0.001, a=0.0001 and A=0.00001, a=0.000001, respectively.}
\label{p120}
\end{figure}

 \begin{figure}[h!]
	\begin{minipage}{0.24\textwidth}
		\centering
		\includegraphics[width=\linewidth]{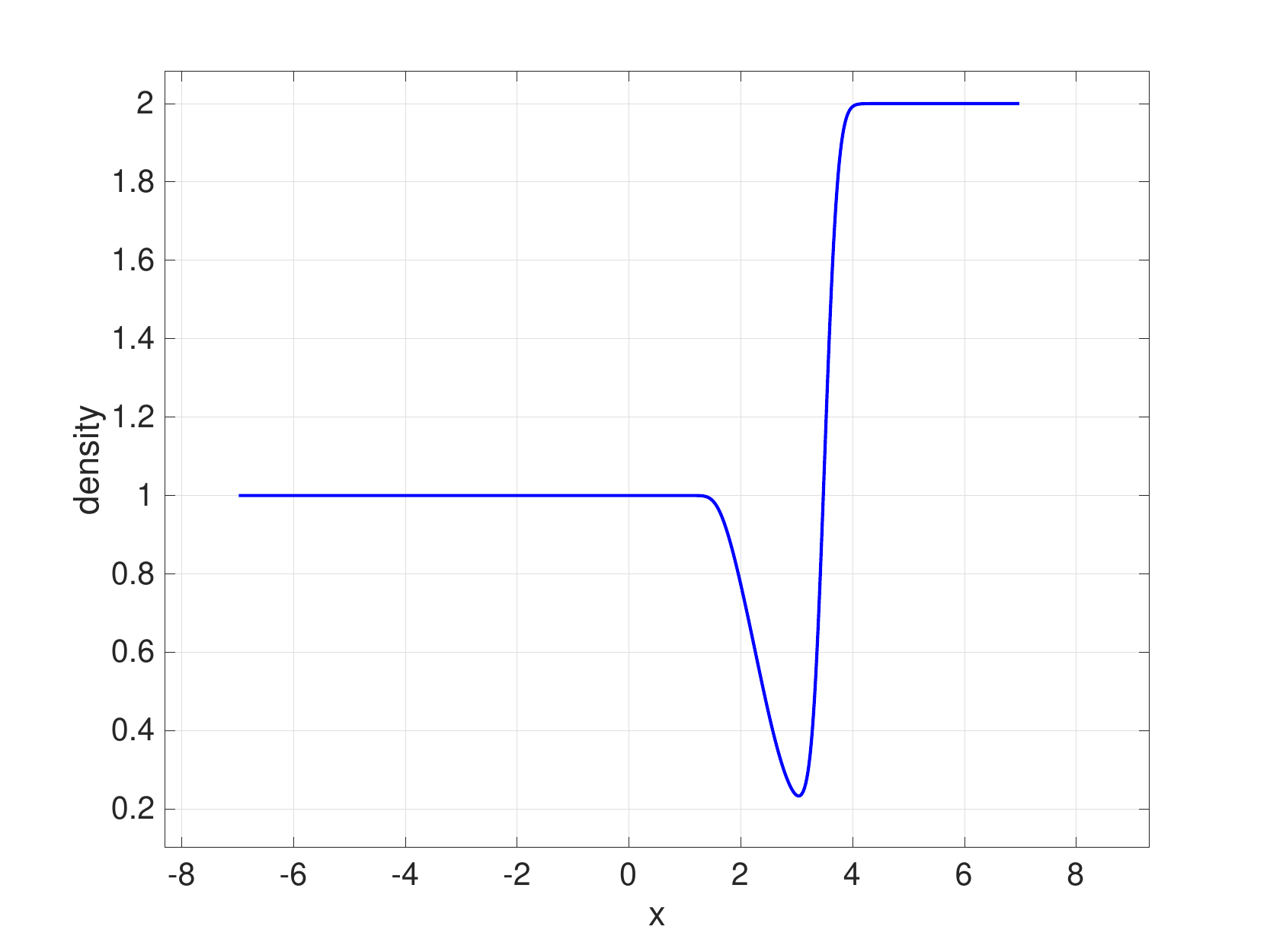}
		%	\caption*{(a)} % Optional label for the subfigure
	\end{minipage}\hfill
	\begin{minipage}{0.24\textwidth}
		\centering
		\includegraphics[width=\linewidth]{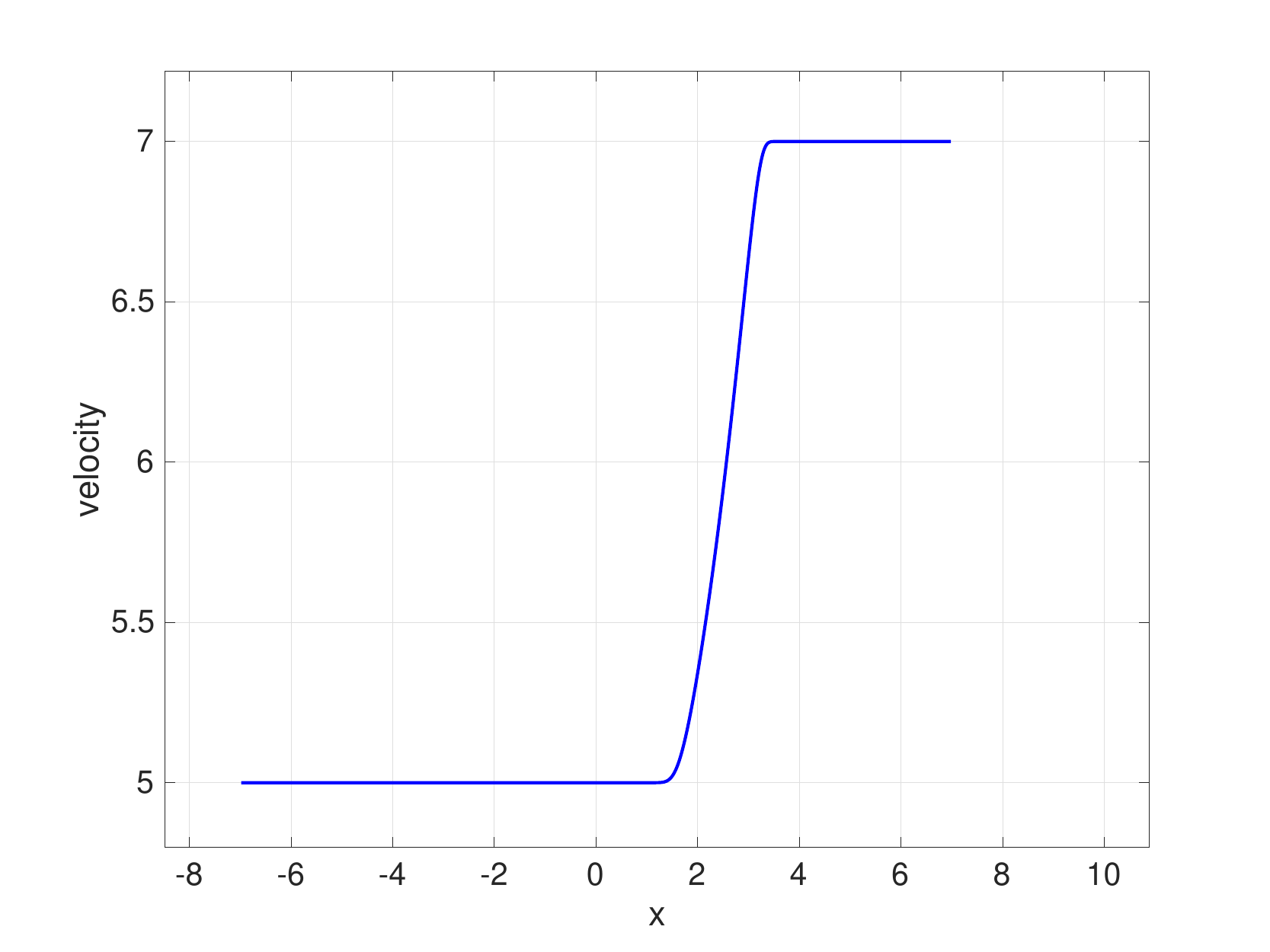}
		%	\caption*{(b)} % Optional label for the subfigure
	\end{minipage}\hfill
	\begin{minipage}{0.24\textwidth}
		\centering
		\includegraphics[width=\linewidth]{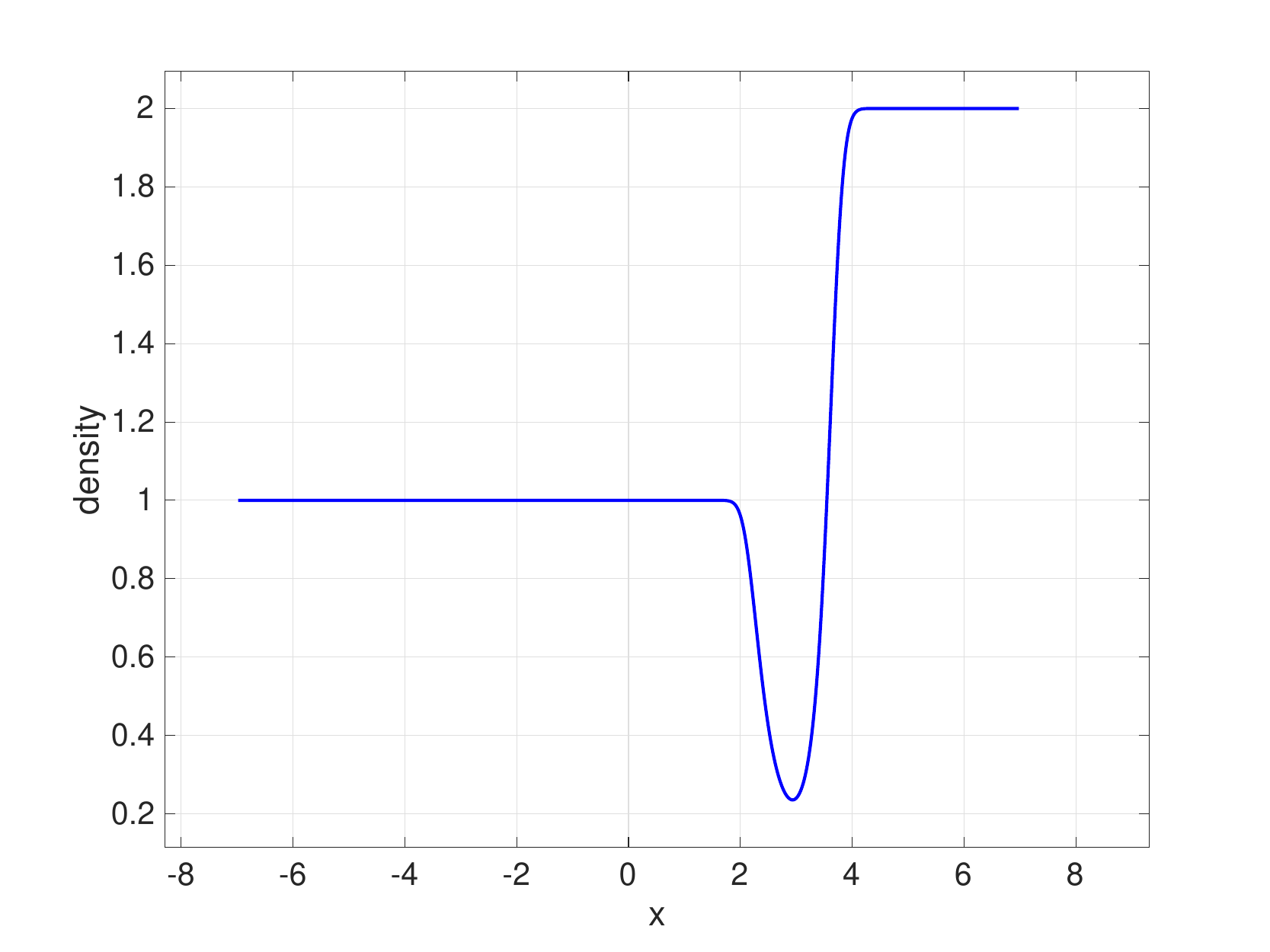}
		%	\caption*{(c)} % Optional label for the subfigure
	\end{minipage}\hfill
	\begin{minipage}{0.24\textwidth}
		\centering
		\includegraphics[width=\linewidth]{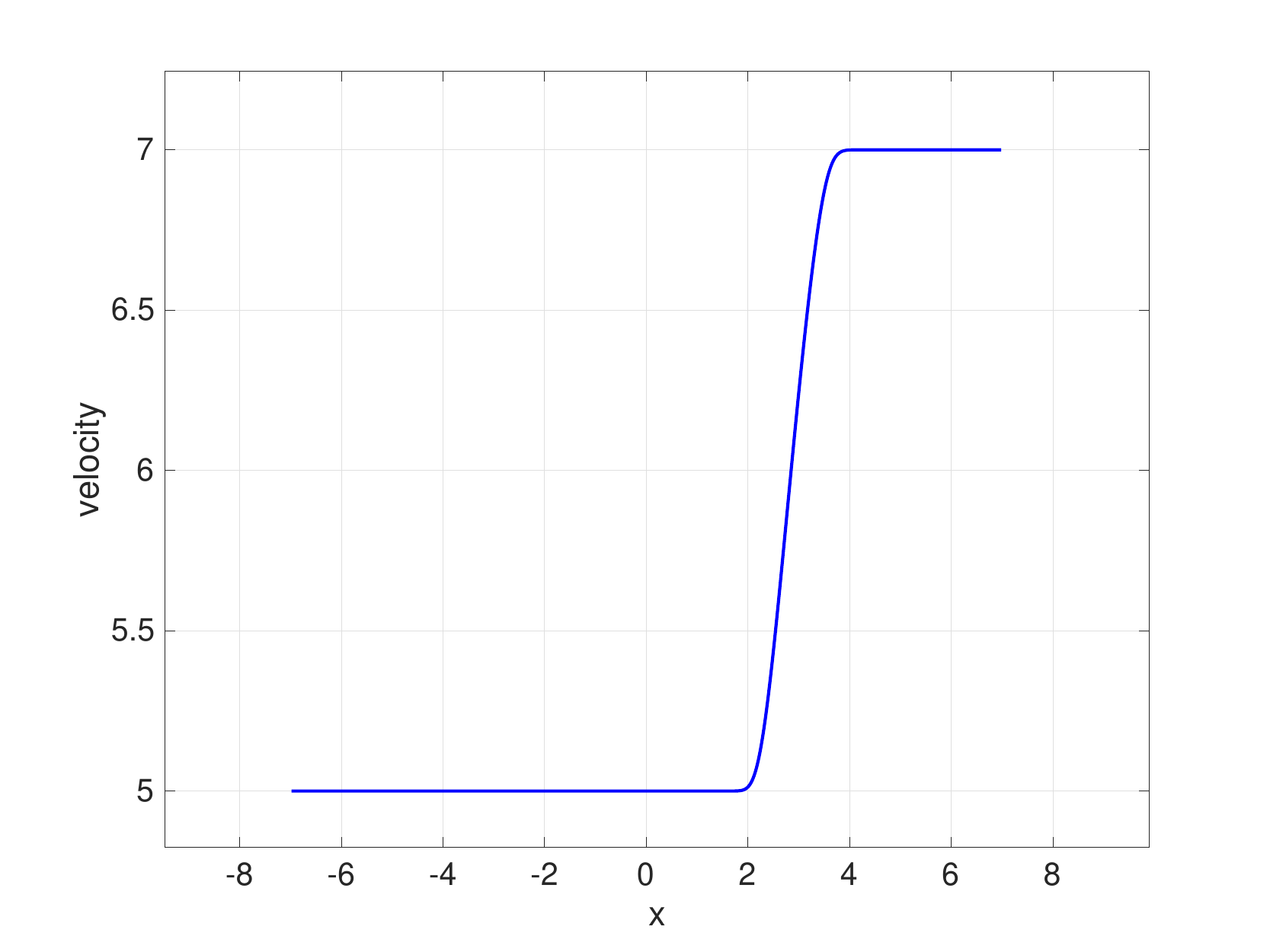}
		%	\caption*{(d)} % Optional label for the subfigure
	\end{minipage}
	\caption{Density and velocity for A=1, a=0.1 and A=0.01, a=0.001, respectively.}
	\label{p121}
\end{figure}

\begin{figure}[h!]
	\begin{minipage}{0.24\textwidth}
		\centering
		\includegraphics[width=\linewidth]{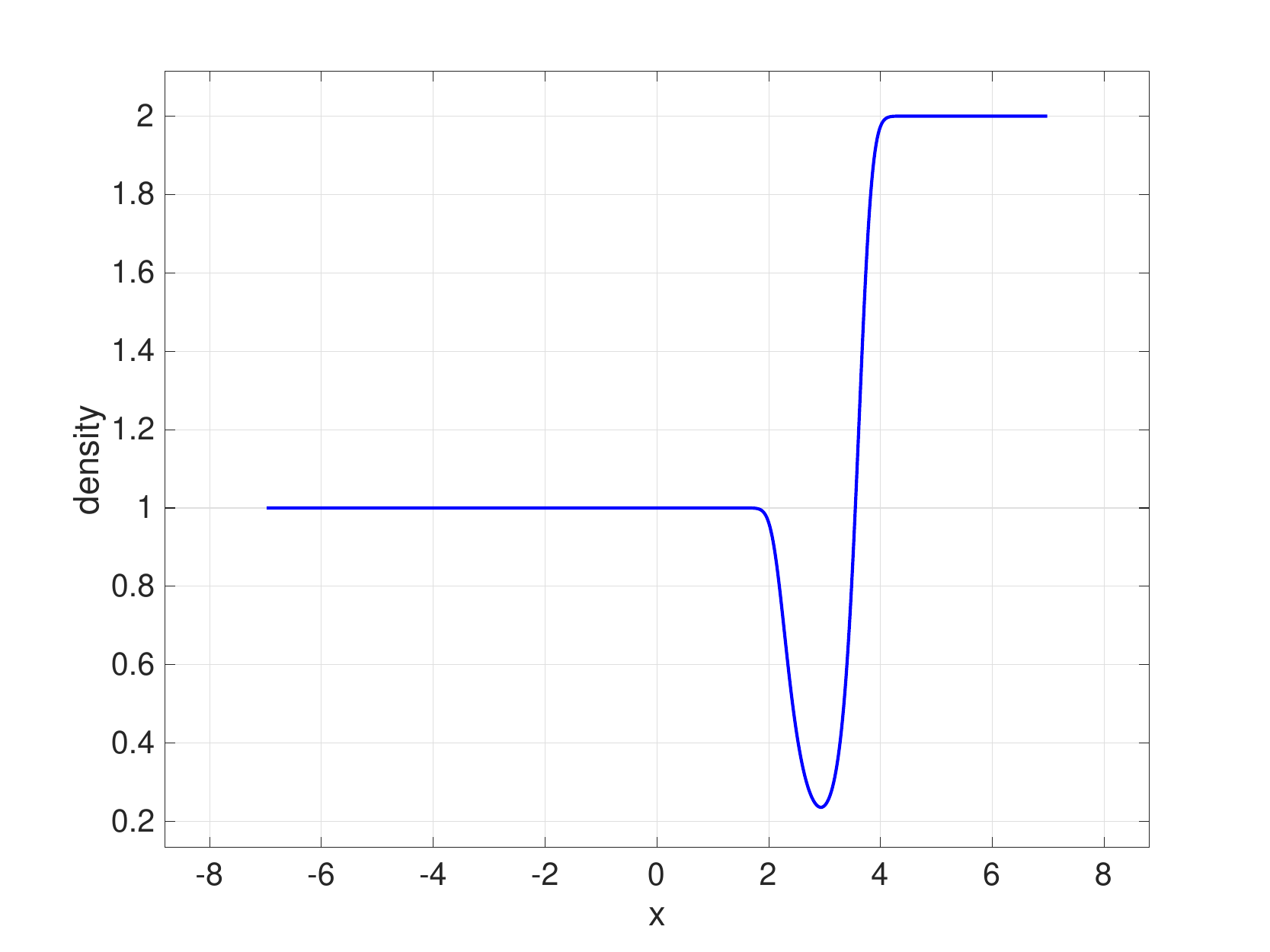}
		%	\caption*{(a)} % Optional label for the subfigure
	\end{minipage}\hfill
	\begin{minipage}{0.24\textwidth}
		\centering
		\includegraphics[width=\linewidth]{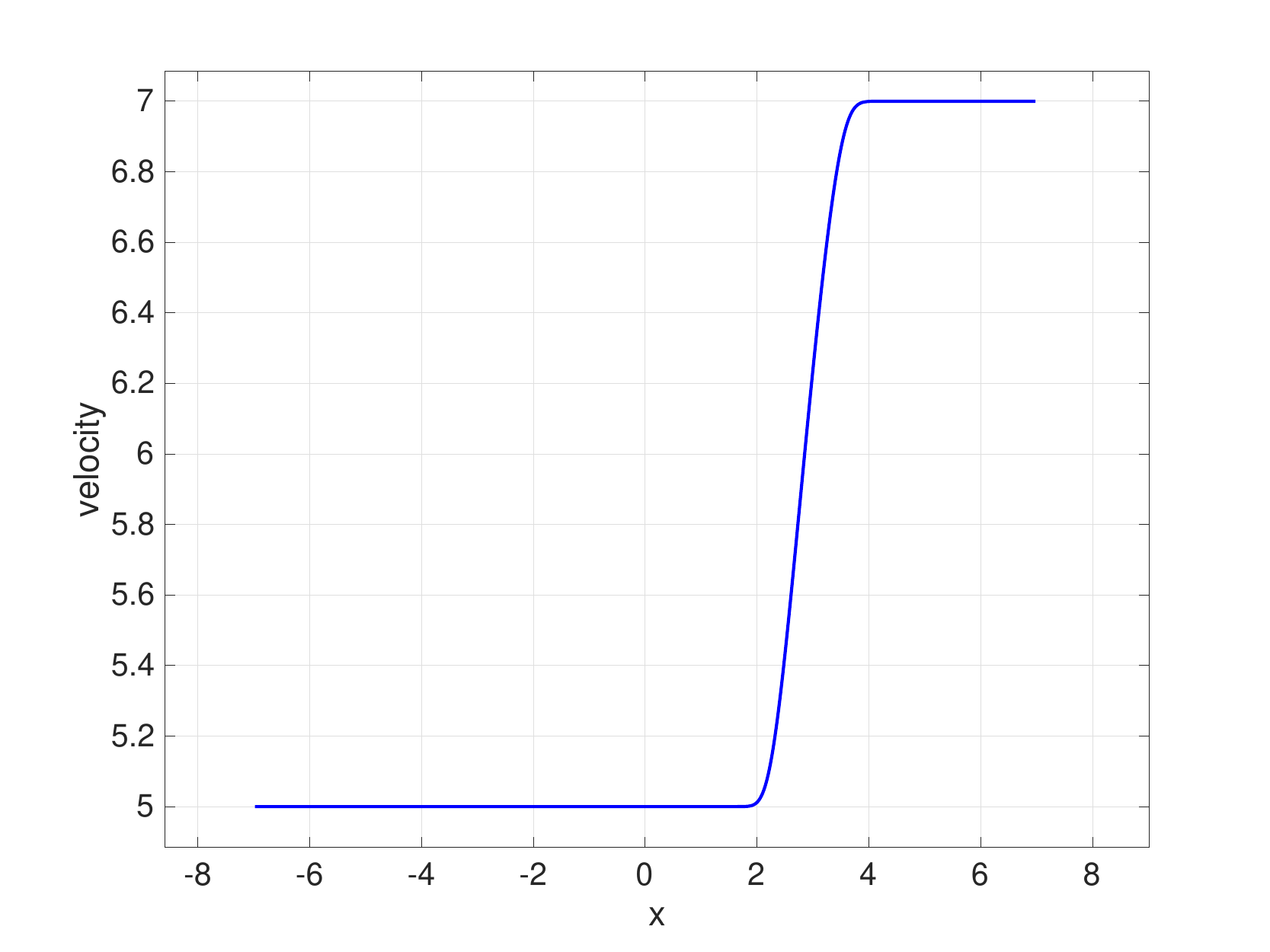}
		%	\caption*{(b)} % Optional label for the subfigure
	\end{minipage}\hfill
	\begin{minipage}{0.24\textwidth}
		\centering
		\includegraphics[width=\linewidth]{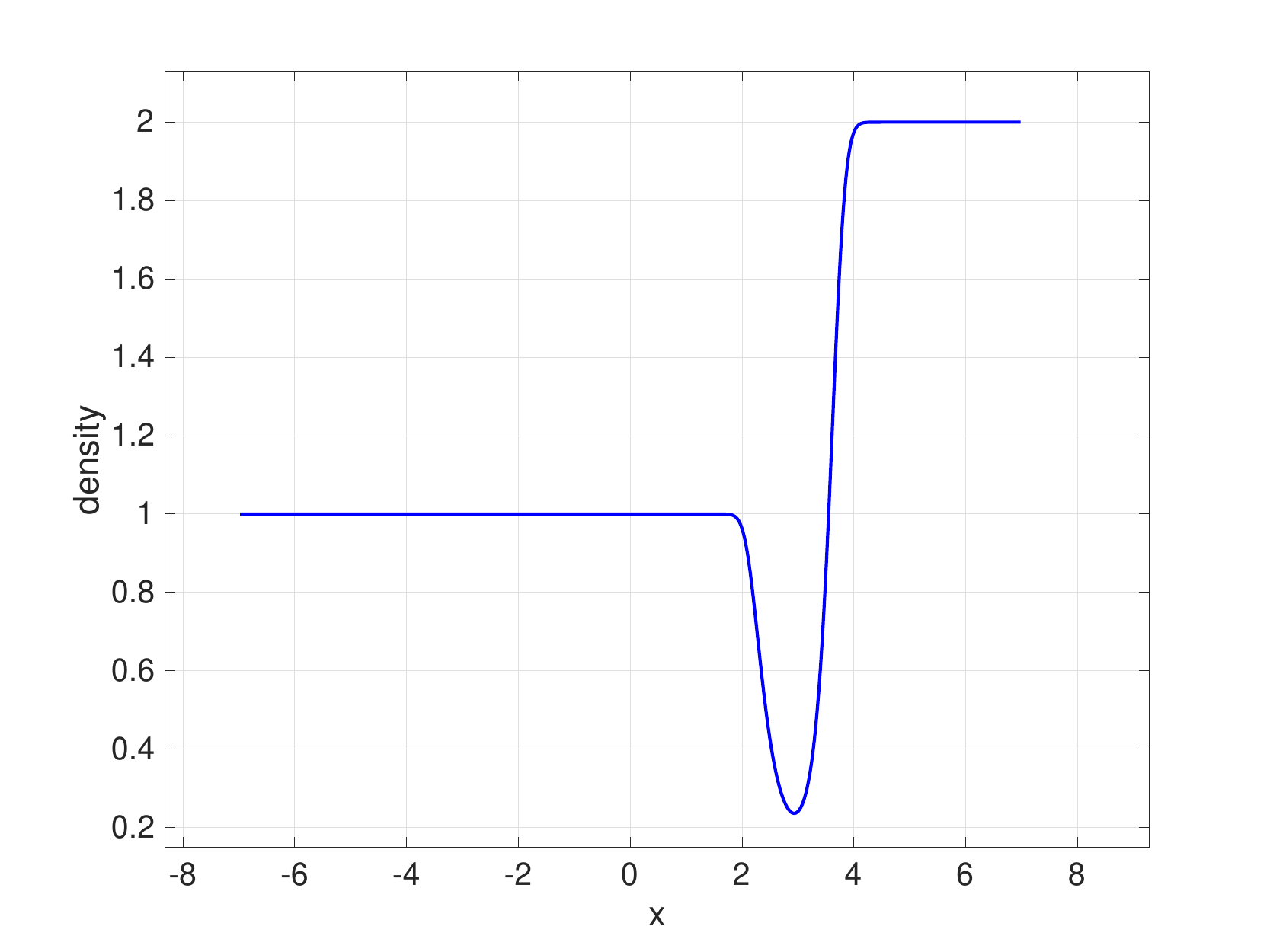}
		%	\caption*{(c)} % Optional label for the subfigure
	\end{minipage}\hfill
	\begin{minipage}{0.24\textwidth}
		\centering
		\includegraphics[width=\linewidth]{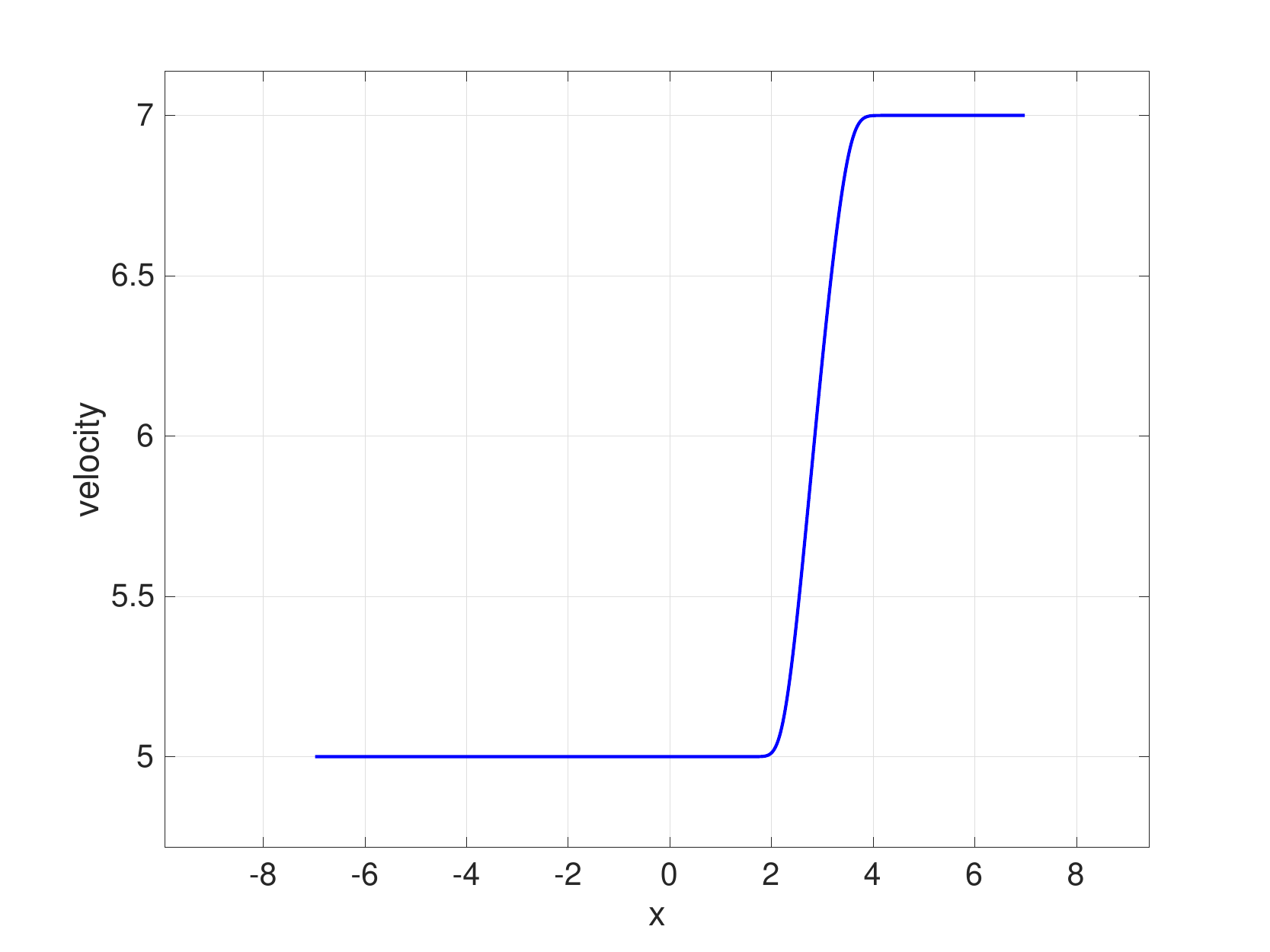}
		%	\caption*{(d)} % Optional label for the subfigure
	\end{minipage}
	\caption{Density and velocity for A=0.001, a=0.0001 and A=0.00001, a=0.000001, respectively.}
	\label{p122}
\end{figure}

	%	\section{Conclusion and Discussion}
	%	In this work, we have obtained the Riemann solutions of \eqref{p1} for a more realistic version of extended Chaplygin gas. In section-$3,$ we found that the Riemann solution $S+J$ converges to $\delta$-shock solution whenever $\upsilon_{r}\leqslant \upsilon_{l}-\frac{B}{\varrho_{l}^{\kappa}}$ and $a, A\to 0.$  We have constructed the Riemann solutions of $\eqref{p1}$ by considering the more generalized value of the pressure $ p.$ With the use of these Riemann solutions, one can directly obtain the Riemann solutions of \eqref{p1} and \eqref{p3} with various equations of state (according to the requirement) by inserting the particular values of $ A, $ $ B, $ $ \Gamma, $ $ a, $ and $ \kappa $ in these solutions. \\

\noindent		\textbf{Acknowledgements.} 
		The authors are  grateful to the reviewers for a thorough read of our paper and valuable comments/suggestions. The supports provided by UGC, India (Ref. No.:201610063559) and NBHM, India  (Ref. No.: NBHM(RP)/R\&D II/7857) are gratefully acknowledged by Priyanka and M. Zafar, respectively.\\
		
		\noindent\textbf{Data Availability.}  Data sharing is not applicable to this article as no new data were created or analyzed in this study.
		
%	\bibliographystyle{plain}
%		\bibliography{pmasterbib_Final}
\def\cprime{$'$}

	\end{document}